\documentclass{amsart}
\usepackage{amsmath, amssymb,amscd, hyperref, bm, color, thmtools , mathrsfs, mathtools}
\usepackage{appendix}
\usepackage{cleveref}
\usepackage{bm}
\usepackage[dvipsnames]{xcolor}

\usepackage{verbatim}
\usepackage{enumerate}
\usepackage{stmaryrd}  
\usepackage[mathcal]{eucal}
\usepackage{array,float}
\usepackage{longtable}
\usepackage{multirow}

\usepackage{cite}
\usepackage[left=3cm, right=3cm]{geometry}

\usepackage{xy}
\input xy
\xyoption{all}
\usepackage{pdflscape} 
\usepackage[alphabetic, msc-links]{amsrefs}
\usepackage{hyperref}
\hypersetup{
	colorlinks=true,
	linkcolor=blue,
    citecolor=RubineRed,
	filecolor=magenta,      
	urlcolor=cyan,
}

\usepackage[draft]{graphicx}

\numberwithin{equation}{section}
\newtheorem{thm}{Theorem}[section]
\newtheorem{proposition}[thm]{Proposition}

\newtheorem{corollary}[thm]{Corollary}
\newtheorem{lemma}[thm]{Lemma}

\theoremstyle{definition}

\newtheorem*{remark*}{Remark}
\newtheorem{remark}[thm]{Remark}

\newcommand{\gunonsq}{\varepsilon} 
\newcommand{\F}{\mathbb{F}}
\newcommand{\lup}{{\ell_2}}
\newcommand{\ldown}{{\ell_1}}
\newcommand{\g}{\mathfrak{g}}
\newcommand{\gl}{\mathfrak{gl}}
\newcommand{\gu} {\mathfrak{gu}}

\newcommand{\GU}{\mathrm{GU}}
\newcommand{\GUU}[2]{\mathrm{GU}_{#1}(
\cO_{#2})}

\newcommand{\GL}{\mathrm{GL}}

\newcommand{\B}{\mathrm{B}}
\newcommand{\U}{\mathrm{U}}

\newcommand{\cO}{\mathfrak{o}}
\newcommand{\smat}[4]{\left[\begin{smallmatrix}
		#1 & #2 \\  #3 & #4 \end{smallmatrix}\right]} 
\newcommand{\mat}[4]{\left[ \begin{matrix}
		#1 & #2 \\  #3 & #4 \end{matrix}\right] }
\newcommand{\cmat}[2]{\left[\begin{smallmatrix}
		#1 \\  #2 \end{smallmatrix}\right]}

\newcommand{\Z}{\mathrm{Z}}

\renewcommand{\det}{{\mathrm{det}}}

\newcommand{\G}{\mathrm{G}}
\newcommand{\C}{\mathrm{C}}

\newcommand{\co}{{\mathfrak{o}}}

\newcommand{\sns}{{\mathbf{sns}}}
\renewcommand{\ss}{\mathbf{ss}}
\newcommand{\cus}{\mathbf{cus}}
\newcommand{\ind}{\mathrm{Ind}}
\newcommand{\K}{\mathrm{K}}
\newcommand{\T}{\mathrm{T}}
\newcommand{\I}{\mathrm{I}}

\newcommand{\Lri}{\mathfrak{O}}
\newcommand{\lri}{\cO}
\renewcommand{\tt}{\mathfrak{t}}
\newcommand{\nreg}{\mathbf{nreg}}
\newcommand{\nonsq}{\epsilon} 
\newcommand{\mfp}{\mathfrak{p}}
\newcommand{\cOl}{{\cO_{\ell}}}
\newcommand{\cOd}{{\cO_{\ldown}}}

\newcommand{\god}{\g(\cOd)}

\newcommand{\Kd}{\K^{\ldown}}

\newcommand{\Gol}
{\G(\cOl)}

\newcommand{\mfP}{\mathfrak{P}}

\newcommand{\tC}{\mathrm{C}}
\newcommand{\rl}{R_\ell}
\newcommand{\dtag}{D (\tilde{\alpha_1},\tilde{\alpha_2}, g_1, g_2)}
\newcommand{\ti}[1]{\tilde{#1}}

\title[On Tensor products of regular characters of $\GL_2$ and $\GU_2$]{On Tensor products of regular characters of  the General \\ Linear  and unitary  groups of degree two over the  \\ principal ideal   local  rings of finite length} 

\author{Archita Gupta}
\address{Department of Mathematics and Statistics, IIT Kanpur,  Kanpur 208016, India}
\email{architagup20@iitk.ac.in}
\author{M Hassain}
\address{Statistics and Mathematics Unit, Indian Statistical Institute, Bangalore 560059, India}
\email{hassainm\_pd@isibang.ac.in}

\author{Pooja Singla}
\address{Department of Mathematics and Statistics, IIT Kanpur,  Kanpur 208016, India}
\email{psingla@iitk.ac.in}


\begin{document}

\begin{abstract}
	Let $R$ be a principal ideal local ring of finite length with a finite residue field of odd characteristic. Let $G(R)$ denote either the general linear group or the general unitary group of degree two over 
	$R$. We study the decomposition of tensor products of irreducible representations of 
	$G(R)$. It is known that the irreducible representations of 
	$G(R)$ are built from regular representations, which are classified into three types: cuspidal, split semisimple, and split non-semisimple.
	
	We prove that the tensor product of any two regular irreducible representations of distinct types has irreducible constituents with multiplicity at most two. Moreover, we show that the regular part of the tensor product of a cuspidal representation with any other regular representation is multiplicity free. When both factors are of split semisimple type, we show that the multiplicity of any regular irreducible constituent is at most $\mathrm{length}(R) + 1$, and that this bound is achieved only when the constituent is also split semisimple. In contrast, we demonstrate that the multiplicity in the tensor product of two split non-semisimple representations can grow with the cardinality of the residue field when the length of the ring is at least two.
	
	In the case when $R$ is a finite field, all such tensor product multiplicities are uniformly bounded above by two. This highlights a significant difference between the behaviour of tensor products in the field case and in the more general finite local ring setting.  
\end{abstract}
\maketitle
\tableofcontents
    
\section{Introduction}
 The tensor product problem, a classical question in representation theory, concerns decomposing the tensor product of two irreducible representations into a direct sum of irreducible representations. This problem appears widely across mathematics. For instance, in Schur–Weyl duality, the decomposition of tensor powers of the standard representation of $\mathrm{GL}_n$ illustrates the rich interplay between linear and symmetric group representations. Similarly, in the context of finite groups, tensor product decompositions are central to  understanding the structure of representations of groups.

The problem has been extensively studied for various families of groups. In the case of the polynomial representations of $\GL_n(\mathbb{C})$, Littlewood and Richardson~\cite{Littlewood-Richardson}, and independently Robinson~\cite{MR1507943}, proposed a rule describing the decomposition of such tensor products. This rule was rigorously proved later in~\cites{MR498826, MR511739}. The tensor product problem for irreducible characters of the symmetric and alternating groups, as well as their double covers, has been studied in depth in~\cites{MR1201916, MR1722888, MR1725703, MR1847134}. Although the problem remains open in general, a complete classification of irreducible representations of $S_n$ with multiplicity-free tensor products was obtained in~\cite{MR3720803}, and analogous results for plethysms of Schur functions appeared in~\cite{MR4439501}.

For finite general linear groups, Hiss and Lübeck~\cite{MR2125073} proved that for $\GL_n(\mathbb{F}_q)$ and $\GU_n(\mathbb{F}_q)$, the multiplicity of a unipotent character in the tensor product of two unipotent characters is a polynomial in $q$ with rational coefficients. In most cases, the tensor square of the Steinberg representation of a finite simple group of Lie type contains every irreducible character~\cite{MR3056296}.

  In recent work, Letellier-Nam~\cite{letellier2025saxlconjecturetensorsquare} established an analogue of the Saxl conjecture for the tensor square of unipotent characters of $\GL_n(\mathbb{F}_q)$. The tensor products of generic irreducible characters of $\GL_n(\mathbb{F}_q)$ were studied in~\cites{MR3022764,MR3034296}, and those of split semisimple (not necessarily generic) irreducible characters in~\cite{Scognamiglio_2024}. Further, Letellier and Rodriguez-Villegas~\cite{letellier2024ennoladualitydecompositiontensor} investigated Ennola duality in the decomposition of tensor products of unipotent and generic characters of $\GL_n(\mathbb{F}_q)$ and $\GU_n(\mathbb{F}_q)$, by relating the multiplicities of irreducible characters in these groups. Despite this progress, the tensor product problem for $\GL_n(\mathbb{F}_q)$ and $\GU_n(\mathbb{F}_q)$ even for $n \geq 3$ remains open in general. 
A few partial results for $\GL_2(\F_q)$ and
$\GL_3(\mathbb F_q)$ are included in \cites{MR1757476, MR3201448}. 
For $\GL_2(\mathbb{F}_q)$, a complete decomposition of the tensor product was independently obtained in \cite{kaur2023gl2} and \cite{gupta-Hassain2025tensor}.

In this article, we 
study the tensor product problem for the general linear and unitary groups of degree two over the principal ideal local rings. These groups are natural generalization of $\GL_2(\mathbb F_q)$ and $\GU_2(\mathbb F_q).$
 
Let $\lri$ be a complete discrete valuation ring with residue field $\mathsf{k}$ of odd characteristic. Let~$\mfp$ be the maximal ideal and let $\pi$ be a fixed uniformizer. Let $\Lri$ be an unramified quadratic extension of $\lri$.   For $\ell \in \mathbb N$, we let $\lri_\ell=\lri/\mfp^\ell$ denote the finite quotient.
Let $\G$ denote either the general
linear group $\GL_2$ or the unitary group $\GU_2$ associated with $\Lri$.

The representation theory of $\G(\co_\ell)$ is well studied, see \cites{MR2588859, MR2456275, MR3737836, Campbell-thesis}. It is known that the irreducible representations of $\G(\cO_\ell)$ fall into two categories: {\bf regular} and {\bf non-regular}. The non-regular representations arise, up to a twist, via induction from the regular representations of $\G(\cO_i)$ for some $i < \ell$. In this spirit, the {\bf regular representations} are the building blocks of the representation theory of $\G(\cO_\ell)$. For $\GL_2$, regular representations coincide with the so-called generic representations~\cite{MR4399251}. Any regular representation $\rho$ of $\G(\cO_\ell)$ for $\ell \geq 2$ is known to have its dimension in the set
\[
\{ (q-1)q^{\ell-1}, (q+1)q^{\ell-1}, (q^2-1)q^{\ell-2} \}.
\]
Based on these dimensions and their constructions, regular representations are classified into types $\tt(\rho)$ as follows:
\begin{itemize}
\item Cuspidal: $\tt(\rho) = \cus$  if $\dim(\rho) = (q-1)q^{\ell-1}$,

\item Split semisimple: $\tt(\rho) = \ss$ if $\dim(\rho) = (q+1)q^{\ell-1}$,

\item Split non-semisimple: $\tt(\rho) = \sns$ if $\dim(\rho) = (q^2-1)q^{\ell-2}$.
\end{itemize} 
For $\ell = 1$, the dimension formulas differ slightly. To describe results uniformly, we define all non-linear irreducible representations of $\G(\cO_1)$ as regular, with types determined analogously:
 \begin{itemize} 
\item $ \cus$ if $\dim(\rho) = q-1$,

\item $ \ss$ if $\dim(\rho) = q+1$,

\item $\sns$ if $\dim(\rho) = q$.
\end{itemize}

Our focus here is on the {\bf tensor product of regular representations} of $\G(\cO_\ell),$ particularly determining the {\bf multiplicity} of regular constituents in such products.  This problem for $\G = \GL_2$ and $\ell = 1$ has been previously studied in \cite{gupta-Hassain2025tensor}, we extend those results to $\ell \geq 1$ for $\GL_2(\cO_\ell)$ and also include the results for $\GU_2(\cO_\ell)$. In particular, we aim to classify pairs of regular representations $\rho_1$ and $\rho_2$  such that their tensor product $\rho_1 \otimes \rho_2$ is {\bf multiplicity free}.

Let $\lambda, \mu, \nu$ be regular irreducible representations of $\G(\cO_\ell)$. We denote the multiplicity of $\nu$ in $\lambda \otimes \mu$ by $g_{\lambda \mu}^\nu$. 
Our main results provide sharp upper bounds for the multiplicities of regular constituents in tensor products of regular representations, classified according to the types involved.

\begin{thm}
\label{thm:main-theorem}
Let $\ell \geq 1$, and let $\lambda, \mu, \nu$ be regular irreducible representations of $\G(\cO_\ell)$. 
\begin{enumerate} 
\item If $\cus \in \{\tt(\lambda), \tt(\mu), \tt(\nu)\},$ then 
\[
g_{\lambda \mu}^\nu \leq 1. 
\]
\item If the set $\{\tt(\lambda), \tt(\mu), \tt(\nu)\}$ consists of exactly two types, then 
\[
g_{\lambda \mu}^\nu \leq 2,
\]
with equality occurring only when the triple $(\tt(\lambda), \tt(\mu), \tt(\nu))$ is a permutation of $(\ss, \sns, \ss).$
\item If all three representations are of type $\ss$, i.e., $\{\tt(\lambda), \tt(\mu), \tt(\nu)\} = \{\ss\}$, then
\[
g_{\lambda \mu}^\nu \leq  \ell+1.
\]
\end{enumerate}     
\end{thm}

\begin{corollary}
\label{cor:tensor-product-results}
  Let $\lambda$ and $\mu$ be regular irreducible representations of $\G(\cO_\ell)$ with $\tt(\lambda) = \cus$.
\begin{enumerate} 
\item If $\tt(\lambda) \neq \tt(\mu)$, then the tensor product $\lambda \otimes \mu$ is multiplicity free.

\item The {\bf regular part} of $\lambda \otimes \mu$ that is, the sum of regular irreducible constituents of $\lambda \otimes \mu$  is multiplicity free.
\end{enumerate} 
\end{corollary}

\begin{thm}
\label{thm:sns-theorem}
Let $\ell \geq 1$ and let $\lambda, \mu, \nu$ be regular irreducible representations of $\G(\co_\ell)$ such that 
\[
\{ \tt(\lambda), \tt(\mu), \tt(\nu) \} = \{ \sns \}.
\]
\begin{enumerate}
    \item For $\ell = 1$, we have $g_{\lambda \mu}^\nu \leq 1$. 
    \item For $\ell \geq 2$, there exist representations $\lambda, \mu, \nu$ such that
    \[
    g_{\lambda \mu}^\nu \geq (q - 2)q^{\lfloor \frac{\ell}{2} \rfloor-1}.
    \]
\end{enumerate}
\end{thm}

\begin{corollary}
\label{cor:residue-dependence}
For $\ell \geq 2$, there exist regular irreducible representations $\lambda, \mu, \nu$ of $\G(\cO_\ell)$ such that the multiplicity $g_{\lambda \mu}^\nu$ depends on the cardinality of the residue field.
\end{corollary}

From the dimension formulae, it is clear that Ennola duality holds between $\GL_2(\cO_\ell)$ and $\GU_2(\cO_\ell)$, parallel to $\GL_n(\mathbb F_q)$ and $\GU_n(\mathbb F_q)$ case (see \cite{MR156900} for details on Ennola duality). However Ennola duality does not work for the tensor product decomposition for $\GL_2(\cO_\ell)$ and \autoref{thm:main-theorem} provides examples of such representations. This has already been observed for $\GL_n(\F_q)$ case in \cite{letellier2024ennoladualitydecompositiontensor}.

We now outline the ideas underlying the proof. Recall that a representation $\rho$ of $\G(\cO_\ell)$ is called a {\bf twist} of $\rho'$, if $\rho \cong \chi \otimes \rho'$ for a one dimensional representation $\chi$ of $\G(\cO_\ell)$. It is easy to note that the decomposition of a representation into irreducible constituents determines the decomposition for any of its twists. Hence, in determining the multiplicities of irreducible constituents of $\rho_1 \otimes \rho_2$, we may work with suitable twists of $\rho_1$ and $\rho_2$. We also note that for any representations $\rho_1, \rho_2, \rho_3$, we have $\langle \rho_1 \otimes \rho_2, \rho_3 \rangle =  \langle \rho_1, \rho_2^\vee \otimes \rho_3 \rangle,$ where $\rho_2^\vee$ denotes the dual representation of $\rho_2$. For any regular representation $\rho$ of $\G(\cO_\ell)$, we have $\tt(\rho) = \tt(\rho^\vee).$ This allows us to permute $(\tt(\rho_1), \tt(\rho_2), \tt(\rho_3))$ as required.

As mentioned earlier, the case of $\ell = 1$ and $\G = \GL_2$ is already settled in \cite{gupta-Hassain2025tensor}. We extend these results to $\GU_2(\cO_1)$ in \autoref{sec:unitary over field}. 

For $\ell \geq 2$, we classify the pairs of regular representations $(\rho_1, \rho_2)$ by their types as follows:

\begin{itemize}
    \item $\Xi_1 = \{(\rho_1, \rho_2) \mid \tt(\rho_1) = \ss, \tt(\rho_2) = \sns \}$
    \item $\Xi_2 = \{(\rho_1, \rho_2) \mid \tt(\rho_1) \neq \tt(\rho_2),\ \tt(\rho_1) = \cus \}$
    \item $\Xi_3 = \{(\rho_1, \rho_2) \mid \tt(\rho_1) = \tt(\rho_2) = \cus \}$
    \item $\Xi_4 = \{(\rho_1, \rho_2) \mid \tt(\rho_1) = \tt(\rho_2) = \ss \}$
    \item $\Xi_5 = \{(\rho_1, \rho_2) \mid \tt(\rho_1) = \tt(\rho_2) = \sns \}$
\end{itemize}
Since $\rho_1 \otimes \rho_2 \cong \rho_2 \otimes \rho_1$, the above five families exhaust all tensor products of regular irreducible representations of $\G(\cO_\ell)$. We use $\mathrm{Irr}(\G(\cO_\ell))$ and $\mathrm{Irr}^{\mathrm{reg}}(\G(\cO_\ell))$ to denote the set of all in-equivalent irreducible representations and the set of all regular representations of   $\G(\cO_\ell),$ respectively. We prove the following result based on the above classification of types. 

\begin{thm}
\label{thm:main-theorem-2}
 For $\ell \geq 2$, the following hold:
 \begin{enumerate}
     \item For $(\rho_1, \rho_2) \in \Xi_1,$ $\langle \rho_1\otimes \rho_2, \rho \rangle \leq 2$ for every $\rho \in \mathrm{Irr}(\G(\cO_\ell)).$ Further equality holds only if $\tt(\rho) = \ss.$ 
     \item For $(\rho_1, \rho_2) \in \Xi_2,$ $\langle \rho_1\otimes \rho_2, \rho \rangle \leq 1$ for every $\rho \in \mathrm{Irr}(\G(\cO_\ell)).$
     \item For $(\rho_1, \rho_2) \in \Xi_3,$ $\langle \rho_1\otimes \rho_2, \rho \rangle \leq 1$ for every $\rho \in \mathrm{Irr}^{\mathrm{reg}}(\G(\cO_\ell)).$
     \item For $(\rho_1, \rho_2) \in \Xi_4,$  $\langle \rho_1\otimes \rho_2, \rho \rangle \leq \ell+1$ for every $\rho \in \mathrm{Irr}^{\mathrm{reg}}(\G(\cO_\ell))$ such that $\tt(\rho) = \ss,$ 
     \item There exists $(\rho, \rho) \in \Xi_5$ such that $\langle \rho \otimes \rho, \rho \rangle \geq (q-2) q^{\lfloor \frac{\ell}{2} \rfloor-1}. $  
 \end{enumerate}
\end{thm}
We note that for $\ell \geq 2$,
\autoref{thm:main-theorem}, \autoref{cor:tensor-product-results}, and \autoref{thm:sns-theorem}  directly  follow from the above result. Hence  major part of this article will be dedicated to prove  \autoref{thm:main-theorem-2}. For this, we use the fact that every regular irreducible representation $\rho$ of $G = \G(\cO_\ell)$ is \emph{imprimitive}, i.e., there exists a proper subgroup $H \subsetneq G$ and an irreducible representation $\phi$ of $H$ such that 
\[
\rho \cong \mathrm{Ind}_H^G(\phi).
\]

To understand the tensor product $\rho_1 \otimes \rho_2$ where $\rho_i = \mathrm{Ind}_{H_i}^G(\phi_i)$, we use Mackey's formula:
\[
\mathrm{Ind}_{H_1}^G(\phi_1) \otimes \mathrm{Ind}_{H_2}^G(\phi_2) \cong \bigoplus_{g \in H_1 \backslash G / H_2} \mathrm{Ind}_{H_1 \cap H_2^g}^G\left( \phi_1 \otimes \phi_2^g \right).
\]

To compute the multiplicity of an irreducible representation $\rho$ as a constituent of $\rho_1 \otimes \rho_2$, we proceed via the following steps:
\begin{itemize}
    \item[(A)] Determine double coset representatives in $H_1 \backslash G / H_2$.
    \item[(B)] Analyze the decomposition of the induced representation 
    \[
    V(\phi_1, \phi_2^g) := \mathrm{Ind}_{H_1 \cap H_2^g}^G(\phi_1 \otimes \phi_2^g)
    \]
    for each $g \in H_1 \backslash G / H_2$.
    \item[(C)]  Understand the intertwining space
    \[
\mathrm{Hom}_G(V(\phi_1, \phi_2^g), V(\phi_1, \phi_2^h))
    \]
    for distinct double coset representatives $g, h \in H_1 \backslash G /H_2$.
\end{itemize}

We conclude this section with an outline of the article.
Basic notation used throughout is listed in \autoref{sec: notation}.
In \autoref{sec:unitary over field}, we prove \autoref{thm:main-theorem} and \autoref{thm:sns-theorem} for the case $\ell = 1$. From \autoref{subsec:construction} onward, we assume $\ell \geq 2$. For the reader’s convenience, \autoref{subsec:construction} includes a brief review of the construction of $\G(\cO_\ell)$, along with alternative constructions from the literature that we use later in the paper.

In \autoref{sec:related-results-construction}, we list several results related to this construction. While these results follow from known methods, we could not find them explicitly stated in the literature. Therefore, for completeness, we include their statements and proofs. Step (A) of our analysis for $\Xi_1, \Xi_2$ and $\Xi_3$ that is, a description of $S_{A_1} \backslash G /S_{A_2}$ is carried out in \autoref{sec:double-coset-description}.
A proof of \autoref{thm:main-theorem-2}(1)-(3) is completed in \autoref{sec:proof-of-them-s1-s3}.
The analysis for types $\Xi_4$ and $\Xi_5$ is independent of the earlier cases and is completed in \autoref{sec:proof-for-Sigma-4} and \autoref{sec:results-Sigma-5}, respectively and these sections also include a proof of \autoref{thm:main-theorem-2}(4) and \autoref{thm:main-theorem-2}(5), respectively. Finally, in \autoref{sec:further-discussion-questions}, we include further discussion and some natural questions arising from this work.

\section{Notation}
\label{sec: notation}
Recall that $\lri$ is a complete discrete valuation ring with residue field $\mathsf{k}$ of cardinality $q$ and odd characteristic~$p$.  Let~$\mfp$ be the maximal ideal and let $\pi$ be a fixed uniformizer. Let $\Lri$ be an unramified quadratic extension. It follows that there exists $\gunonsq \in \Lri$ with $\gunonsq^2 \in \lri^\times \smallsetminus (\lri^\times)^2$ such that $\Lri=\lri[\gunonsq]$. Let $\mfP=\pi \Lri$ be the maximal ideal in $\Lri$ and $\mathcal{K}=\Lri/\mfP$ the residue field, a quadratic extension of $\mathsf{k}$ generated by the image of $\gunonsq$. For $\ell \in \mathbb N$, we let $\lri_\ell=\lri/\mfp^\ell$ and  $\Lri_\ell=\Lri/\mfP^\ell$ denote the finite quotients. We denote by $x \mapsto x^\circ$ the non-trivial Galois automorphism of $\Lri/\lri$, characterised by $\gunonsq^\circ= -\gunonsq$. The image of $\gunonsq$ in $\Lri_i$ will also be denoted by $\gunonsq$ for all $i$.

\subsection{The unitary group and its Lie algebra} 
\label{subsec: unitary group const.}
In this section, we describe our unitary group and its Lie algebra. We will restrict our definitions to the group $\GU_2$.
Let $W = \smat{0}{1}{1}{0} \in \GL_2(\Lri_\ell)$ denote the permutation matrix corresponding to the  longest Weyl element.  Consider the involution on $\gl_2(\Lri_\ell)$ defined by
\begin{equation}\label{staroperation}
  (a_{i,j})^\star \coloneqq W(a_{j,i}^\circ)W^{-1},
\end{equation}
and its associated Hermitian form on $\Lri_\ell^2$ given by:
\[
  \langle (u_1, u_2),(v_1, v_2) \rangle \coloneqq v_1^\circ u_2 + v_2^\circ u_1.
\]

For $\ell \in \mathbb N \cup \{\infty\}$ the unitary group with respect to $\star$ and its Lie algebra of anti-Hermitian matrices are given by
\[
  \begin{split}
    \GU_2(\lri_\ell) &\coloneqq \left\{ A \in \GL_2(\Lri_\ell) \mid A^\star A=\mathrm{I}_2 \right\},  \\ 
    \gu_2(\lri_\ell) &\coloneqq \left\{ A \in \gl_2(\Lri_\ell) \mid A+ A^\star =0\right\}.
  \end{split}
\]

  By definition of $\gu_2(\cO_\ell)$, any $A \in \gu_2(\cO_\ell)$ is of the form $\left[ \begin{smallmatrix}
      x & \gunonsq y \\
      \gunonsq z & -x^\circ  
  \end{smallmatrix} \right],$ for $x \in \Lri_\ell$ and $y, z \in \cO_\ell.$ Observe that $A = \smat{a}{b}{c}{d} \in \GUU2\ell$ if and only if the following holds: \begin{enumerate}
    \item $a d^\circ + c b^\circ = 1$
    \item $a b^\circ + a^ \circ b = 0$
    \item $a c^\circ + a^\circ c = 0$
    \item $d b^\circ + d^\circ b = 0$
    \item $d c^\circ + d^\circ c = 0$
\end{enumerate}
We will use the above conditions as the defining conditions of the unitary group whenever needed. The elements of the sets $\{a,b\}, \{b,d\}, \{c,d\}, \{a, c\}$ are called neighbors of $A$. 
One can easily show that, whenever defined, the ratio of the squares of the neighbors of  $A$ is either zero or a non-square in $\cO_\ell$, i.e., in $\gunonsq^2 (\cO_\ell)^2$.  
Further $A= [a_{i,j}]
\in \gu_2(\lri_\ell)$ if and only if
$a_{i,j}+a_{3-j,3-i}^\circ=0$ for $i,j \in \{1,2\}$. 

\smallskip

Throughout this paper we consider  $\GL_2$ and $\GU_2$ as $\lri$-group schemes, where the $R$-points of the latter are the fixed points of $A \mapsto (A^\star)^{-1}$ for every $\lri$-algebra $R$ and $A \in \gl_2(R)$. Let $\g$ be the lie algebra scheme of $\G$. Then $\g$ is either $\gl_2$ or $\gu_2$ as $\lri$-Lie algebra schemes, the latter being the fixed points of $A \mapsto -A^\star$. The adjoint action of a group on its Lie algebra will be denoted by $\mathrm{Ad}$.
Recall $\cO_1 = \mathbb F_q.$ 
\smallskip

Define \[\rl\coloneqq\begin{cases}
    \co_\ell, & \text{ for } \G=\GL_2;\\
    \Lri_\ell, & \text{ for } \G=\GU_2.
\end{cases}\] For the uniformity in the proofs, we define
\[
\nonsq = \begin{cases} 
1, & \text{ for } \G=\GL_2;\\
    \gunonsq, & \text{ for } \G=\GU_2. 

\end{cases} 
\]
We will use these notations throughout this article.
\section{Proof of \autoref{thm:main-theorem} and \autoref{thm:sns-theorem} for $\ell = 1$}
\label{sec:unitary over field}

In this section we discuss the decomposition of the tensor product of irreducible representations of $\G(\cO_1)$. This problem for $\GL_2(\F_q)$ has already been addressed by the first two authors of this article, see \cite{gupta-Hassain2025tensor}. In this section, we will focus on the parallel results for $\GU_2(\F_q)$.   

The representation theory of the group $\GU_2(\F_q)$ is parallel to that of $\GL_2(\F_q)$. We follow   \cite{Campbell-msc-thesis} to include a few details regarding this.  Let $\alpha,\beta\in \widehat{\F_{q^2}^\times}$ and $x,y \in \F_{q^2}^\times$. Denote $x+y$ and $x-y$ by $m$ and $n$, respectively. The character table of $\mathrm{GU}_2(\F_q)$ is given in \autoref{table:Character-tableGU2} (see \cite[Page-21]{Campbell-msc-thesis}).
\begin{table} 
\begin{center}
\begin{tabular}{|c|c|c|c|c|}
\hline
& $\left[\begin{smallmatrix}
    x&0\\0&x
\end{smallmatrix}\right] $ & $ \left[\begin{smallmatrix}
    x&y\\0&x
\end{smallmatrix}\right]$ & $ \left[\begin{smallmatrix}
    x&0\\0&y
\end{smallmatrix}\right] $ & $ \left[\begin{smallmatrix}
    x&y\\y&x
\end{smallmatrix}\right] $ \\ \hline
$ \chi_\alpha^1$ & $\alpha(x)^2$ & $\alpha(x)^2 $& $\alpha(x)\alpha(y)$ & $\alpha(x^2 - y^2)$ \\ \hline
$\chi_\alpha^q$ & $q\alpha(x)^2$ & 0 & $\alpha(x)\alpha(y)$ & $-\alpha(x^2 - y^2)$ \\ \hline
$\chi_{\alpha,\beta}^{q+1}$& $(q+1)\alpha(x)\beta(x) $& $\alpha(x)\beta(x)$ & $\alpha(x)\beta(y) + \alpha(y)\beta(x) $& 0 \\ \hline
$\chi_{\alpha,\beta}^{q-1} $& $(q-1)\alpha(x)\beta(x)$ &$ -\alpha(x)\beta(x)$ & 0 & $-[\alpha(n)\beta(m)+\alpha(n)\beta(m)]$\\ \hline
\end{tabular}
\end{center}
\caption{Character table of $\GU_2(\mathbb F_q)$}
\label{table:Character-tableGU2}
\end{table} 
From now on in this section, we denote $\GU_2(\F_q)$ by $G$. Let $\U$ be the subgroup consisting of unipotent upper triangular matrices. Fix $\psi$ to be a non-trivial character of $\F
_{q^2}^+$ such that $\psi$ is non-trivial on the additive subgroup $\{t\in \F_{q^2}^+ \mid t+t^\circ=0\}\subseteq  \F_{q^2}^+$. Let $\Z$ be center of the group $G$. 
Define the following two subgroups of $G$: 
\begin{gather*}
    H_1:=\{\left[\begin{smallmatrix}
    x & 0 \\ 0 & y
\end{smallmatrix} \right] \mid x,y \in \F_{q^2}\} \cap G, \\
 H_2:=\{\left[\begin{smallmatrix}
    x & y \\ y & x
\end{smallmatrix}\right]   \mid x,y \in \F_{q^2}  \} \cap G.
\end{gather*}
For $\alpha,\beta \in \widehat{\F_{q^2}^\times},$ define characters $(\alpha,\beta)$ of $H_1$, $H_2$ and character $(\alpha, \beta)\psi$ of $\Z\U$ as follows: 
\begin{eqnarray*} 
(\alpha, \beta): H_1 \rightarrow \mathbb C^\times; \,\, (\alpha,\beta)\left(\left[\begin{smallmatrix}
    x & 0 \\ 0 & y
\end{smallmatrix}\right]\right)=\alpha(x)\beta(y),\nonumber \\ 
(\alpha, \beta): H_2 \rightarrow \mathbb C^\times; \,\, (\alpha,\beta) \left(\left[\begin{smallmatrix}
    x & y \\ y & x
\end{smallmatrix}\right]\right)=\alpha(x+y)\beta(x-y), \nonumber \\
(\alpha, \beta)\psi: \Z\U \rightarrow \mathbb C^\times; \,\, (\alpha,\beta)\psi\left(\left[\begin{smallmatrix}
    x & y \\ 0 & x
\end{smallmatrix}\right]\right)=\alpha(x)\beta(x) \psi(x^{-1}y).\nonumber
\end{eqnarray*}  
The character of $\ind_{\Z\U}^G (\alpha,\beta)\psi$ is as given below:
\begin{equation} 
\label{induced-character-GU2}
\begin{array}{l|cccc}  & \left[\begin{smallmatrix}
    x&0\\0&x
\end{smallmatrix}\right] & \left[\begin{smallmatrix}
    x&y\\0&x
\end{smallmatrix}\right] & \left[\begin{smallmatrix}
    x&0\\0&y
\end{smallmatrix}\right] & \left[\begin{smallmatrix}
    x&y\\y&x
\end{smallmatrix}\right]\\ \hline
\ind_{\Z\U}^G (\alpha,\beta)\psi & (q-1)(q+1) \alpha(x) \beta(x) & -\alpha(x) \beta(x) & 0 & 0
\end{array}
\end{equation} 
Let $\mathcal{L}\coloneqq \{x\in \F_{q^2}^\times \mid xx^\circ =1\}$. Suppose $\alpha= \beta$ as characters of $\F_q^\times$, then define $\gamma\circ \det:G\to \mathbb{C}^\times$ by $\gamma(\det(g))=\alpha(a)\beta({a^\circ}^{-1})$, where $\det(g)=a{a^\circ}^{-1}$ for some $a\in \F_{q^2}^\times$ which exists by the fact that the map $\mathcal{Q}:\F_{q^2}^\times\to\mathcal{L}$ defined by $\mathcal{Q}(x)=x{x^\circ}^{-1}$ is surjective (\cite[Section~0.0.1 (ii)]{Campbell-msc-thesis}). 
 The following result directly follows from ~\autoref{table:Character-tableGU2} and ~\autoref{induced-character-GU2}.
\begin{proposition}
\label{lem: ZU multiplicity free for GU2}
    \begin{enumerate} 
    \item The representation $ V_\psi := \ind_{\U}^{G} \psi$ is multiplicity free and every non-linear irreducible representation of $G$ is a sub-representation 
    of $V_\psi.$
    \item $\ind_{H_1}^G(\alpha,\beta)=\begin{cases}
            \ind_{\Z \U}^G (\alpha,\beta)\psi +\chi_{\alpha,\beta}^{q+1}, & \text{if } \alpha \neq \beta \text{ on } \F_q^\times;\\
             \ind_{\Z \U}^G (\alpha,\beta)\psi +\chi_{\gamma}^q +\chi_{\gamma}^1, & \text{if } \alpha = \beta \text{ on } \F_q^\times.
            
        \end{cases}$
         \item $\ind_{H_2}^G(\alpha,\beta)=\begin{cases}
            \ind_{\Z \U}^G (\alpha,\beta)\psi -\chi_{\alpha,\beta}^{q-1}, & \text{if } \alpha \neq \beta \text{ on } \F_q^\times;\\
             \ind_{\Z \U}^G (\alpha,\beta)\psi -\chi_{\gamma}^q +\chi_{\gamma}^1, & \text{if } \alpha = \beta \text{ on } \F_q^\times.
        \end{cases}$
    \end{enumerate}
\end{proposition}

The following corollary is evident from \autoref{lem: ZU multiplicity free for GU2}.
\begin{corollary}
\label{cor:multiplicity-free-GU-2fq}
\begin{enumerate}
    \item We have $\langle \ind_{H_1}^G(\alpha,\beta), \chi_{(\alpha,\beta)}^{q+1}\rangle =2$ for $\alpha \neq \beta \text{ on } \F_q^\times$, and $\langle \ind_{H_1}^G(\alpha,\beta), \chi_{\gamma}^{q}\rangle =2 $ for $\alpha = \beta \text{ on } \F_q^\times$.
    \item The representation $\ind_{H_2}^G(\alpha,\beta)$ is multiplicity free. 
\end{enumerate}
\end{corollary}

 \autoref{table:Character-tableGU2} and \autoref{lem: ZU multiplicity free for GU2} directly give the following result regarding the decomposition of the tensor product of the irreducible representations of $\GU_2(\mathbb F_q)$. This result is parallel to Theorem 3.1 in \cite{MR1757476}. 
\begin{proposition}
\label{prop:GU2-Fq}
 Let $\alpha,\beta,\gamma,\delta \in \widehat{\F_{q^2}^\times}$. Then
    \begin{enumerate}
			\item $\chi_\alpha^q\otimes \chi_{\beta, \gamma}^{q+1}=\ind_{H_1}^G(\alpha \circ det)(\beta,\gamma)$. 
			\item $\chi_\alpha^q\otimes \chi_{\beta,\gamma}^{q-1}=\ind_{H_2}^G(\alpha\beta, \alpha \gamma)$.
			\item $\chi_{\alpha, \beta}^{q+1}\otimes \chi_{\gamma, \delta}^{q+1}=\ind_{H_1}^G(\alpha\gamma, \beta\delta)+\chi_{\alpha\delta,\beta\gamma}^{q+1}$.
			\item $\chi_{\alpha, \beta}^{q+1}\otimes \chi_{\gamma,\delta}^{q-1}=\ind_{H_1}^G(\alpha\beta,\gamma\delta)-\chi_{\alpha\beta,\gamma\delta}^{q+1}$.
			\item $\chi_{\alpha,\beta}^{q-1}\otimes\chi_{\gamma,\delta}^{q-1}=\ind_{H_2}^G(\alpha\delta,\beta\gamma)-\chi_{\alpha\gamma,\beta\delta}^{q-1}$.
			\item$\chi_\alpha^q\otimes \chi_\beta^q=\ind_{H_2}^G(\alpha\beta,\alpha\beta)+\chi_{\alpha\beta}^q$.
		\end{enumerate}
\end{proposition}

From \autoref{cor:multiplicity-free-GU-2fq} and \autoref{prop:GU2-Fq}, we obtain the following result. 
\begin{corollary}
     Let $\chi, \chi' \in \mathrm{Irr}(\GU_2(\F_q))$. Then $\chi \otimes \chi'$ is multiplicity free except for the cases $\chi_\alpha^q \otimes \chi_{\beta,\gamma}^{q+1} $ and $ \chi_{\alpha,\beta}^{q+1} \otimes \chi_{\gamma,\delta}^{q+1}$. Further, the highest multiplicity of any irreducible representation in $\chi \otimes \chi'$ is two  and it is due to $q$ or $(q + 1)$-dimensional constituents.
\end{corollary}
The parallel result for $\GL_2(\mathbb F_q)$ also holds, see \cite[Corollary~1.2]{gupta-Hassain2025tensor}. By combining these two results, we obtain a proof of  \autoref{thm:main-theorem} and \autoref{thm:sns-theorem} for $\ell = 1$. 

\section{Construction of regular representations of  $\G(\cO_\ell)$}
\label{subsec:construction}
In this section, we first give a construction of representations of $\G(\cO_\ell)$ as described in \cite[Section~3]{MR3737836}. We then present a few alternative constructions from the literature. These results will be used throughout the remainder of this paper.

For $ i \leq \ell$, let  $\rho_{\ell,i}: \cO_\ell \rightarrow \cO_i$ be the natural projection maps. The corresponding natural projection maps $\G(\cO_{\ell})  \rightarrow \G(\cO_i) $ are also denoted by $\rho_{\ell, i}$. For any matrix $A \in \g(\co_\ell)$, we denote $\rho_{\ell,1}(A)$ by $\bar{A}$. Let $\K^{i} =  \ker (\rho_{\ell,i}) $ be the $i$-th congruence subgroups of $\G(\cO_\ell).$  For $i \geq \ell/2,$  the group  $\K^i $ is isomorphic to the abelian additive subgroup $\g(\cO_{\ell-i})$ of $M_n(R_{\ell-i}).$ Let $\psi: \rl  \rightarrow \mathbb C^\times$ be a fixed primitive one dimensional representation of $\rl$. For $ \rl=\Lri_\ell,$ we assume that $\psi$ satisfies $\psi(x+\nonsq y)= \psi'(x)\psi'(y)$ for some primitive one dimensional representation $\psi'$ of $\cO_\ell$. Therefore,  $\pi^{\ell-1}\cO_\ell \not\subseteq \ker(\psi) $ by our choice of $\psi$.

 For any $i \leq \ell/2$ and $A = [a_{st}] \in \g(\cO_i),$ we will consider lifts $\tilde{A} = [\widetilde{a_{st}}] \in \g(\cO_\ell)$ of $A$ such that $\rho_{\ell,i}(\tilde{A}) = A$ with   $\widetilde{a_{st} } = \nonsq$ for $a_{st} = \nonsq, $ and  $\widetilde{a_{st}} = 0$ for $a_{st} = 0$. In this case, we say $\tilde{A}$ is a {\bf Serre lift} of $A$.  
 
 For any $i \leq \ell/2$ and  $A \in \g(\cO_i),$ let $\tilde{A} \in \g(\cO_\ell)$ be a lift of $A$. Define $\psi_A: \I+ \pi^{\ell-i} \g(\cO_\ell) \rightarrow \mathbb C^\times$ by 
 \[ 
 \psi_A(\I+ \pi^{\ell-i} B) \coloneqq \psi(\pi^{\ell-i}\bm{tr}(\tilde{A}B)),
 \]
for all $\I+ \pi^{\ell-i} B \in \K^{\ell-i}.$  Then $\psi_A$ is a well defined one dimensional representation of $\K^{\ell-i}.$ Further, the following duality for abelian groups $\K^{i}$ and $\g(\cO_{\ell-i})$ holds for $i \geq \ell/2$.
\begin{equation}
\label{eq:duality}
\g(\cO_{\ell-i}) \cong \widehat {\K^{i}}\,\,; A \mapsto \psi_A\,\, \mathrm{where}, \,\, \psi_A(\I+ \pi^{i} B) = \psi(\pi^{i}\bm{tr}(\tilde{A}B))  \,\, \forall \,\, \I+ \pi^{i} B \in \K^{i}. 
\end{equation}

We say a one dimensional representation $\psi_A \in \widehat{\K^{i}}$ for $i \geq \ell/2$ is regular if and only if $A \in \g(\cO_{\ell-i})$ is a regular matrix (that is the characteristic polynomial is equal to its minimal polynomial). In this case the stabilizer of $A$ in $\G(\cO_{\ell-i})$ under the conjugation action is $\{x \I + yA \mid x, y \in R_{\ell-i}\} \cap \G(\cO_{\ell-i}).$
 By (\cite[Lemma~2.3]{MR4399251}), for $i \geq \ell/2$ the representation $\psi_A \in \widehat{\K^{i}}$ is regular if and only if $\psi_A|_{ \K^{\ell-1}}$ is regular. An irreducible representation $\rho$ of $\G(\cO_\ell)$ is called regular if the $\mathrm{Ad}$-orbit of its restriction to $\K^{\ell-1} $ consists of one dimensional representations $\psi_A$ for regular $A$. 

The following lemma describes the orbits of $\g(\cO_\ell)$ under the  $\mathrm{Ad}$-action of $\G(\cO_\ell)$. 
\begin{lemma}
\label{lem:orbit-representatives-gol}
An exhaustive list of $\g(\cO_\ell)$ orbit representatives under the $\mathrm{Ad}$-action of $\G(\cO_\ell)$ is given by matrices $A \in \g(\cO_\ell)$ of the following form:
\begin{enumerate}
    \item[(a)] $x\I+ \pi C $ 
    \item[(b)] $ \begin{bmatrix}
            x & \nonsq \pi \beta \\
           \nonsq & x
        \end{bmatrix}$
    \item[(c)] $ 
      \begin{bmatrix}
            x &\nonsq \delta \\ 
           \nonsq  & x
        \end{bmatrix}$
        with $\delta \in \cO_\ell^\times \setminus (\cO_\ell^\times)^2$ for $\g = \gu_2$ and   $\delta \in 
(\cO_\ell^\times)^2$ for $\g = \gl_2$

        \item[(d)] 
        $ \begin{bmatrix}
        x & \nonsq \sigma \\ \nonsq  & x
        \end{bmatrix} 
        $ with $\sigma \in (\cO_\ell^\times)^2$ for $\g = \gu_2$ and   $\sigma \in 
\cO_\ell^\times \setminus (\cO_\ell^\times)^2$ for $\g = \gl_2.$  
\end{enumerate}
\end{lemma} 
\begin{proof}
For $\GL_2$, proof follows from \cite[Section~2]{MR2584957}. For $\GU_2,$ we note that $A  \in \gu_2(\cO_\ell)$ if and only $A$ is anti-hermitian. If $A$ is a scalar modulo $\pi,$ then $A$ is of type (a). Otherwise the result follows from Lemma~\cite[Lemma~3.5]{MR3471251}. 
\end{proof}
\begin{remark}
\begin{enumerate} 
\item The exhaustive list of $\gu_2(\cO_\ell)$ orbits in the above result differs from \cite[Section 4.F,  Page-34]{Campbell-thesis} up to a translation by a scalar matrix and/or multiplication by an invertible scalar.
Therefore, the cardinalities of the inertia groups and the stabilizers  are the same for ${\it loc. cit.}$ and the above orbit representatives. We will use these computations from \cite{Campbell-thesis}, whenever required. 
\item For part (c) above, let $\delta= r^2 \nonsq^2 $ for some $r\in \co_\ell^\times$. The matrix $\smat{x+r \nonsq^2}{0}{0}{x-r \nonsq^2}$ also represents the same orbit as $\smat{x}{\nonsq\delta}{\nonsq}{x}$. We will use this form of $A$ whenever needed.
\item To describe  the construction as well the decomposition of the tensor product of irreducible representations of $\G(\cO_\ell),$ we can choose suitable twists of $A \in \g(\cO_\ell)$ that is modify $A$ upto an addition of an appropriate scalar matrix. For our case, up to these twists, we can always assume that $A \in \g(\cO_\ell)$ is chosen such that $\bm{tr}(A) = 0.$ Whenever required,  we shall work with such a choice of $A$ without specifically mentioning it.  
\end{enumerate} 
\end{remark} 
 Define $\tt: \g(\lri_\ell) \rightarrow \{\nreg ,  \sns, \ss, \cus  \}   $ by 
$ \tt(A) = \nreg\,(\sns, \ss, \cus)$ if $A$ is equivalent to a matrix given in above (a) ((b), (c), (d)). 
 Now we summarize very briefly the construction of regular representations of $\G(\cO_\ell) $ with emphasis on the statements that we require in this article.  
\subsection{Construction of regular representations of $\G(\cO_\ell)$ for $\ell$ even}
\label{E.construction}
Let $\psi_A \in \widehat{\K^{\ell/2}}$ be a regular one dimensional representation of $\K^{{\ell/2}}$ for $A \in \g(\cO_{\ell/2}).$ Then the following gives the construction in this case. Let   $ S_A = \{ g \in \G(\cO_\ell) \mid \psi_A^g \cong \psi_A  \}  $ be the inertia group of $\psi_A$ in $\G(\cO_\ell).$ 
Let $\tilde{A} \in \g(\cO_\ell)$ be a lift of $A$, and  let  $\tC_{\G(\lri_\ell)}(\tilde{A})$
denote its  stabilizer  in $\G(\co_\ell)$ under the $\mathrm{Ad}$-action. Then $ S_A  =  \tC_{\G(\lri_\ell)}(\tilde{A})       \K^{{\ell/2}} .$  
Let $\rho \in \mathrm{Irr} \left( \G(\cO_\ell)  \mid \psi_A    \right)$ be a regular representation of $\G(\cO_\ell),$ then there exists an extension $\widetilde{\psi_A}$ of $\psi_A$ to $S_A$ such that $\rho \cong \mathrm{Ind}_{S_A}^{\G(\cO_\ell)} (\widetilde{\psi_A}) .$
  Every $\rho \in \mathrm{Irr} \left( \G(\cO_\ell) \mid \psi_A \right)$ has dimension $\frac{|\G(\cO_\ell)|}{|\mathrm{C}_{\G(\cO_{\ell/2})}(A)| |\K^{\ell/2}|} .$
\subsection{Construction of regular representations of $\G(\cO_\ell)$ for $\ell $ odd}  
\label{O.construction}
Let $\ldown=\lfloor\ell /2 \rfloor$ and $\lup=\lceil \ell /2 \rceil$ and let $\psi_A \in \widehat{\K^{\lup}}$ be a regular one dimensional representation of $\K^{\lup}$ for $A \in \g(\cO_{\ldown}).$ Let $S_A=\{ g\in \G(\co_\ell) \mid \psi_A^g\cong\psi_A \}$. Let $\tilde{A} \in \g(\cO_\ell)$ be a lift of $A$.   Define the group $\mathrm{Rad}_{A} := \left(\K^{\ldown}\cap \tC_{\G(\lri_\ell)}(\tilde{A})\right) \K^{\lup}.$ The group $\mathrm{Rad}_{A}$ is the radical of the bilinear form 
\[
\mathcal{B}_{A}: \K^{\ldown}/\K^{\lup} \times \K^{\ldown}/\K^{{\lup}} \rightarrow \mathbb C^\times; \,\, \mathcal{B}_{A} (x \K^{\lup}, y \K^{\lup}) = \psi_A([x,y]). 
\]
Therefore, the one dimensional representation $\psi_A$ extends to $\mathrm{Rad}_{A}.$ Let $\widetilde{\psi_A}$ be an extension of $\psi_A$ to $\mathrm{Rad}_{A}$ and $\sigma \in \mathrm{Irr}(\K^{\ldown}\mid \psi_A) $ be the unique irreducible representation determined by $\widetilde{\psi_A}.$ Then, 
 \[\sigma|_{\mathrm{Rad}_{A}} \cong \underbrace{\widetilde{\psi_A} + \cdots + \widetilde{\psi_A} }_{q-\mathrm{times}}. 
 \]  
 Let $I_{\G(\cO_\ell)}(\sigma) = \{ g \in \G(\cO_\ell) \mid \sigma^g \cong \sigma  \}$ be the inertia groups of $\sigma\in \mathrm{Irr}(\K^{\ldown}\mid \psi_A).$
 Then  $I_{\G(\cO_\ell)}(\sigma)  = S_A = \tC_{\G(\lri_\ell)}(\tilde{A}) \K^{\ldown}.$ Every $\sigma \in \mathrm{Irr}(\K^{\ldown} \mid \psi_A)$ extends to the inertia group $I_{\G(\cO_\ell)}(\sigma) .$ In particular, every such extension induces irreducibly to $\G(\cO_\ell)$ and gives rise to a regular representation of $\G(\cO_\ell).$ Every regular $\rho \in \mathrm{Irr} \left( \G(\cO_\ell) \mid \psi_A \right)$ is obtained in this way and has dimension $\frac{q|\G(\cO_\ell)|}{|\mathrm{C}_{\G(\cO_{\ldown})}(A)| |\K^{\ldown}|}.$
  The following result can be easily obtained from the above construction and we shall use it later.
  \begin{proposition}\label{prop:MF from Subgp}
    Let $A\in \g(\cOd)$ be regular and $H$ be a subgroup of $S_A$ such that  $\K^{\ell_1}\leq H\leq S_A.$ 
    \begin{enumerate} 
   \item  Every irreducible representation of $H$ lying above $\psi_A$ has dimension $q$. 
   \item Let  $\phi$ be a representation of $H$ such that $\mathrm{Res}^H_{\K^{\ell_2}}(\phi) =m \psi_A$ for some positive integer $m.$ Then $\mathrm{Ind}_{H}^{\G(\cO_{\ell})}(\phi)$ is multiplicity free if and only if $\phi$ is multiplicity free.
\end{enumerate} 
\end{proposition}
The following lemma describes a maximal isotropic subgroup in certain special cases. 
\begin{lemma}
\label{lem: common max isotropic for diff type}
    For $i \in \{1,2\}$, let  $ A_i\in \god$ be regular matrices such that $\bar{A_1} \notin \mathrm{span}_{\mathbb F_q}\{ \I, \bar{A_2}\}.$ Define a subgroup $H$ of $\K^{\ldown}$ as
    \[
    H:=(\{\I+ \pi^{\ldown}(x\I+y\tilde{{A_1}}+z\tilde{{A_2}})\}\cap \K^{\ldown})\K^{\lup}.
    \]
     Let $\bar{H}$ be the image of $H$ in $\K^{\ldown}/\K^{\lup}$. Then $\bar{H}$ is a maximal isotropic subgroup for the antisymmetric bilinear forms $\mathcal{B}_{A_i}$ for $i \in \{1,2\}$ as defined above. 
\end{lemma}
\begin{proof} 
By direct computations, we can check that the bilinear forms $\mathcal{B}_{A_i}$ for $i \in \{1,2\}$ are trivial on $\bar{H}$. By $\bar{A_1} \notin \mathrm{span}_{\mathbb F_q}\{ \I, \bar{A_2} \}$ and the cardinality of $\bar{H}$, we obtain that $H$ is a maximal isotropic subspace for $\mathcal{B}_{A_i}$ for $i \in \{1,2\}.$  
\end{proof}

\subsection{Alternate construction for split semisimple representations of $\G(\co_\ell)$}
\label{subsec:alternate const for ss}
Let $\B(\co_\ell)$ be the group of upper triangular matrices in $\G(\co_\ell)$. Let $(\chi_1,\chi_2) \in \widehat{\rl^\times} \times \widehat{\rl^\times}$. Define a character of $\B(\co_\ell)$ as follows:
\[(\chi_1,\chi_2)\left( \mat{a}{b}{0}{c} \right)
=\chi_1(a)\chi_2(c).\]
The pair $(\chi_1, \chi_2)$ is called $\ss$-pair of $\G(\cOl)$ if $\chi_1\chi_2^{-1}|_{1+ \pi^{\ell-1}\cO_\ell} \neq 1 $. The set of $\ss$-pairs will be denoted by $\mathfrak{S}$. Let $\T(\co_\ell)$ be the group of diagonal matrices in $\G(\co_\ell)$.
The following lemma characterizes the $\ss$-pairs and split semisimple representations of $\G(\cOl).$

 \begin{lemma}
 \begin{enumerate}

 \item Let $(\chi_1,\chi_2)\in \widehat{\rl^\times} \times \widehat{\rl^\times} $. If $(\chi_1, \chi_2)$ is $\ss$-pair of $\G(\cOl)$, then $\ind_{\B(\co_\ell)}^{\G(\cOl)}(\chi_1,\chi_2)$ is irreducible.
      \item A representation $\rho$ is a split semisimple regular representation of $\G(\co_\ell)$ if and only if $\rho\cong\ind_{\B(\co_\ell)}^{\G(\cOl)}(\chi_1,\chi_2)$ for some $\ss$-pair $(\chi_1, \chi_2)$ of $\G(\cOl)$. 
    \end{enumerate}
    \end{lemma}
     \begin{proof}
     Assume $ (\chi_1,\chi_2)$ and $  (\chi_1',\chi_2')$ are $\ss$-pairs. Then we have
     \begin{equation}\label{eqn:ss consB}
         \langle \ind_{\B(\co_\ell)}^{\G(\co_\ell)}(\chi_1,\chi_2), \ind_{\B(\co_\ell)}^{\G(\co_\ell)}(\chi_1',\chi_2')\rangle  =  \underset{g\in \B(\co_\ell)\backslash \G(\co_\ell)/  \B(\co_\ell) }{\sum}{\langle (\chi_1,\chi_2), (\chi_1',\chi_2')^g\rangle}_{\B(\co_\ell) \cap \B(\co_\ell)^g}.
     \end{equation}
   We also have the decomposition
\[
\G(\co_\ell)
 \;=\;
\B(\co_\ell) 
\smat{0}{1}{1}{0}
\B(\co_\ell)
\;\;\sqcup\;\; \left(
\bigsqcup_{1 \leq i \leq \ell}
\B(\co_\ell)
\smat{1}{0}{\nonsq\pi^i}{1}
\B(\co_\ell)\right).
\]
    For $ i \in [1, \ell]$, let  $g_i\coloneqq\smat{1}{0}{\nonsq\pi^i}{1}$. We claim that ${\langle (\chi_1,\chi_2), (\chi_1',\chi_2')^{g_i}\rangle}_{\B(\co_\ell) \cap \B(\co_\ell)^{g_i}}=0$ for $i \in [1, \ell-1]$. Let $i \in [1, \ell-1]$. For  $b\in \co_\ell$, define 
    \[
    X_b\coloneqq\mat{1-\nonsq^2 \pi^{\ell-1} b}{\nonsq \pi^{\ell-i-1} b}{0}{1+\nonsq^2\pi^{\ell-1} b} .
    \]
    Then it is easy to see that $X_b\in \B(\co_\ell)\cap \B(\co_\ell)^{g_i}$ for all $b\in \cO_\ell$. To prove the  claim, it is enough to prove that $(\chi_1,\chi_2)(X_b)\neq(\chi_1',\chi_2')^{g_i}(X_b)$ for some $b\in \cO_\ell$. Assume on the contrary that $(\chi_1,\chi_2)(X_b)=(\chi_1',\chi_2')^{g_i}(X_b)$ for all $b\in \cO_\ell$.  Upon simplification, we obtain $\chi_1\chi_2^{-1}(1-\nonsq^2\pi^{\ell-1} b)=1$ for all $b\in \cO_\ell$, which contradicts the assumption that $(\chi_1,\chi_2)$ is an $\ss$-pair. This proves the claim. 
    Now, for $g_\ell=\smat{1}{0}{0}{1}$, we have $\B(\co_\ell) \cap \B(\co_\ell)^{g_\ell}=\B(\co_\ell)$ and for $h=\smat{0}{1}{1}{0}$, $\B(\co_\ell)\cap \B(\co_\ell)^{h}=\T(\co_\ell)$. 
  Then \autoref{eqn:ss consB} becomes
     \begin{equation}\label{eqn:ss consB refined}
         \langle \ind_{\B(\co_\ell)}^{\G(\co_\ell)}(\chi_1,\chi_2), \ind_{\B(\co_\ell)}^{\G(\co_\ell)}(\chi_1',\chi_2')\rangle  = {\langle (\chi_1,\chi_2), (\chi_1',\chi_2')\rangle}_{\B(\co_\ell)} + {\langle (\chi_1,\chi_2), (\chi_2',\chi_1')\rangle}_{\T(\co_\ell)}.
     \end{equation}
   
    To prove (1), we need to show that if $(\chi_1,\chi_2)$ is an $\ss$-pair, then $\langle \ind_{\B(\co_\ell)}^{\G(\co_\ell)}(\chi_1,\chi_2), \ind_{\B(\co_\ell)}^{\G(\co_\ell)}(\chi_1,\chi_2)\rangle  = 1$. By \autoref{eqn:ss consB refined}, we have
       \begin{equation}\label{eqn:ss consB refined, proving irreduciblity}
         \langle \ind_{\B(\co_\ell)}^{\G(\co_\ell)}(\chi_1,\chi_2), \ind_{\B(\co_\ell)}^{\G(\co_\ell)}(\chi_1,\chi_2)\rangle  = 1 + {\langle (\chi_1,\chi_2), (\chi_2,\chi_1)\rangle}_{\T(\co_\ell)}.
     \end{equation}
If ${\langle (\chi_1,\chi_2), (\chi_2,\chi_1)\rangle}_{\T(\co_\ell)} \neq 0$, then  $(\chi_1,\chi_2)\left(\smat{a}{0}{0}{c}\right)= (\chi_2,\chi_1)\left(\smat{a}{0}{0}{c}\right)$ for all $\smat{a}{0}{0}{c}\in \T(\co_\ell)$, which simplifies to $\chi_1\chi_2^{-1}(ac^{-1})=1$ for all $\smat{a}{0}{0}{c}\in \T(\co_\ell)$. Therefore, we obtain 
$\chi_1\chi_2^{-1}|_{\cO_\ell^\times} = 1$. This contradicts the assumption that $(\chi_1,\chi_2)$ is an $\ss$-pair. Thus ${\langle (\chi_1,\chi_2), (\chi_2,\chi_1)\rangle}_{\T(\co_\ell)} = 0$. 
Substituting this in  \autoref{eqn:ss consB refined, proving irreduciblity}, we get $   \langle \ind_{\B(\co_\ell)}^{\G(\co_\ell)}(\chi_1,\chi_2), \ind_{\B(\co_\ell)}^{\G(\co_\ell)}(\chi_1,\chi_2)\rangle  = 1 $.

  To prove (2), observe that, by (1), for an $\ss$-pair $(\chi_1,\chi_2)$, the representation $\ind_{\B(\co_\ell)}^{\G(\co_\ell)}(\chi_1,\chi_2)$ is an irreducible representation of dimension  $ \frac{|\G(\co_\ell)|}{|\B(\co_\ell)|}
= (q+1)q^{\ell-1}.$ Therefore, by definition, $\ind_{\B(\co_\ell)}^{\G(\co_\ell)}(\chi_1,\chi_2)$ is an $\ss$-representation.
 For $\G=\GL_2$, the converse follows from \cite[Lemma 2.5 (3)]{MR4936485}. For $\G=\GU_2$, to prove the converse, we first count the number of inequivalent irreducible representations of the form $\ind_{\B(\co_\ell)}^{\G(\co_\ell)}(\chi_1,\chi_2)$, where $(\chi_1,\chi_2)$ is an $\ss$-pair. Observe that  for  $(\chi_1,\chi_2)\in \widehat{\Lri_\ell^\times} \times \widehat{\Lri_\ell^\times} $,  $(\chi_1,\chi_2)=(\chi_1{{\chi_2}^\circ}^{-1},1)$ as characters of  $\B(\co_\ell)$. Also, for $\ss$-pairs  $(\chi_1,1)$ and $(\chi_2,1)$, by \autoref{eqn:ss consB refined},  we have
      \begin{equation}\label{eqn:ss consB refined for GU2}
         \langle \ind_{\B(\co_\ell)}^{\G(\co_\ell)}(\chi_1,1), \ind_{\B(\co_\ell)}^{\G(\co_\ell)}(\chi_2,1)\rangle  = {\langle (\chi_1,1), (\chi_2,1)\rangle}_{\B(\co_\ell)} + {\langle (\chi_1,1), ({\chi_2^\circ}^{-1},1)\rangle}_{\T(\co_\ell)}.
     \end{equation}
  This gives  
     \[
     \ind_{\B(\co_\ell)}^{\G(\co_\ell)}(\chi_1,1)\cong \ind_{\B(\co_\ell)}^{\G(\co_\ell)}(\chi_2,1)\, \mathrm{if \, and \, only \, if}\,
(\chi_1,1) \in \{(\chi_2,1),({\chi_2^\circ}^{-1},1)\}.
\] 
Therefore, the number of inequivalent irreducible representations of the form $\ind_{\B(\co_\ell)}^{\G(\co_\ell)}(\chi_1,\chi_2)$ is equal to 
$$\frac{|\{\chi\in \widehat{\Lri_\ell^\times}  \mid 1+\pi^{\ell-1}\co_\ell \nsubseteq  \ker(\chi) \}|}{2} =\frac{(q-1)|\Lri_\ell^\times|}{2q}= \frac{q^{2\ell-3}(q-1)^2(q+1)}{2}. $$
By \cite[Table 4.3 (Page-61)]{Campbell-thesis},
this is same as the total number of split semisimple representations of $\GU_2(\cO_\ell)$. Hence the converse of (2) follows for  $\G=\GU_2$.
 \end{proof}

\subsection{Alternate construction for split non-semisimple representations of $\G(\co_\ell)$, $\ell$ odd.}
\label{subsec:alternate const for sns}
In this section, we discuss an alternate construction for split non-semisimple representations of $\G(\co_\ell)$ for odd $\ell$. For proofs of these results; see \cite[Section~3.3.3]{MR2584957} for $\G=\GL_2$ and \cite[Section 4.H.2, part 3, Page-57]{Campbell-thesis} for $\G=\GU_2$.
  Let $A=\left[\begin{smallmatrix}
      \alpha & \nonsq \pi \beta\\
      \nonsq & \alpha
  \end{smallmatrix}\right] \in \g(\cO_{\ldown})$ and the Serre lift $\ti{A}=\left[\begin{smallmatrix}
      \ti{\alpha} &\nonsq \pi \ti{\beta}\\
      \nonsq & \ti{\alpha}
  \end{smallmatrix}\right] \in \g(\co_\ell)$ and corresponding character $\psi_A$ of $\K^{\lup}$. Then $S_A = \mathrm{C}_{\G(\co_\ell)}(\ti{A})   \K^{\ldown} $  is given by 
  \begin{equation*}
      S_A =\left\{\begin{bmatrix}
          x & \pi \ti{\beta} y+\pi^{\ldown}z \\
         y & x+\pi^{\ldown}w
      \end{bmatrix} \mid x,y,z,w \in \rl \right\}\cap \G(\co_\ell).
  \end{equation*}
  Consider a normal subgroup $\mathrm{N}=\left \{\left[\begin{smallmatrix}
      1+\pi^{\ldown}x & \pi^{\lup}z\\
       \pi^{\ldown}y    & 1+\pi^{\ldown}w
\end{smallmatrix}\right] \mid x,y,z,w \in \rl \right \} \cap \G(\co_\ell)$ of $S_A$. We can extend $\psi_A$ to $\mathrm{N}$ and since $N/\K^{\lup}$ is abelian, every character in $\mathrm{Irr}(N \mid \psi_A)$ is one dimensional. Define an extension $\psi_{\ti{A}}'$ of $\psi_A$ to $N$ as follows:
\begin{equation*}
    \psi_{\ti{A}}'\left(\begin{bmatrix}
   1+\pi^{\ldown}x & \pi^{\lup}z\\
       \pi^{\ldown}y    & 1+\pi^{\ldown}w
\end{bmatrix}\right)\coloneqq\psi\left(\pi^{\ldown}\bm{tr}\left(\ti{A}\smat{x}{\pi z}{y}{w}-\frac{\pi^{\ldown}}{2}\ti{A}\smat{x}{\pi z}{y}{w}^2\right)\right). 
\end{equation*} 
We can show that the stabilizer of $\psi_{\ti{A}}'$ in $ S_A$ is $ \mathrm{N}  \mathrm{C}_{\G(\co_\ell)}(\ti{A})$. Since $\mathrm{C}_{\G(\co_\ell)}(\ti{A})$ is abelian, we 
can extend $\psi'_{\ti{A}}$ to a character $\psi_{\ti{A}}''$ of $\mathrm{N} \mathrm{C}_{\G(\co_\ell)}(\ti{A})$ and every character of $\mathrm{N} \mathrm{C}_{\G(\co_\ell)}(\ti{A})$ lying above $\psi_{\ti{A}}'$ is one dimensional. Using Clifford theory for the group $S_A$ and its normal subgroup $N$ having character $\psi_{\ti{A}}'$, we get that $\ind_{\mathrm{N} \mathrm{C}_{\G(\co_\ell)}(\ti{A})}^{S_A} \psi_{\ti{A}}''$ is an  irreducible representation of dimension $q$. Denote $\ind_{\mathrm{N} \mathrm{C}_{\G(\co_\ell)}(\ti{A})}^{S_A} \psi_{\ti{A}}''$ by $\phi$. Then $\ind_{S_A}^{\G(\co_\ell)}{{\phi}}$ is a split non-semisimple representation of $\G(\co_\ell)$ and any split non-semisimple representation of $\G(\co_\ell)$ lying above $\psi_A$ is of the form $\ind_{S_A}^{\G(\co_\ell)}{\phi} \cong \ind_{\mathrm{N}  \mathrm{C}_{\G(\co_\ell)}(\ti{A})}^{\G(\co_\ell)}{\psi_{\ti{A}}''}$ for some lift $\ti{A}$ of $A$ and some extension $\psi''_{\ti{A}}$ of $\psi'_{\ti{A}} $ to the group $\mathrm{N}  \mathrm{C}_{\G(\co_\ell)}(\ti{A}).$ 

\subsection{Alternate construction for cuspidal representations of $\G(\co_\ell)$, $\ell$ odd.}
\label{subsec:alternate const for cus} 
Let $A=\smat{0}{\nonsq \alpha}{\nonsq}{0}\in \g(\co_{\ldown})$ be a regular matrix with $\tt(A)=\cus$. Define  $ D^{\ell_i}(\ti{A})\coloneqq(\C_{\G(\cO_\ell)}(\ti{A})\cap \K^1)\K^{\ell_i} \text{ for } i \in \{1,2\}$.
The character $\psi_A$ can be extended to $\Z D^{\ell_2}(\ti{A})$, say $\widetilde{\psi_A}$. We have $\Z D^{\ell_2}(\ti{A}) \trianglelefteq \Z D^{\ell_1}(\ti{A})$ and every element of $\Z D^{\ell_1}(\ti{A})$ stabilizes $\widetilde{\psi_A}$.
By considering the bilinear form on $\Z D^{\ell_1}(\ti{A})/ \Z D^{\ell_2}(\ti{A})$ parallel to the one given in \autoref{O.construction}, we obtain a construction of irreducible representations of $\G({\co_\ell})$ lying above $\psi_A$.  The difference in this case compared to the previous one is that the current bilinear form is non-degenerate. 
 The process of construction is depicted in the following diagram: 

\begin{align*}
    &\K^{\lup} \xrightarrow{ext} \Z D^{\ell_2}(\ti{A}) \xrightarrow{ext} J \xrightarrow{ind} \Z D^{\ell_1}(\ti{A}) \xrightarrow{ext} S_A \xrightarrow{ind} \G(\co_\ell)\\
&\psi_A \hspace{27pt} \widetilde{\psi_A}  \hspace{45pt} \widetilde{\widetilde{\psi_A}} \hspace{40pt} \theta \hspace{35pt} \phi  \hspace{37pt} \rho
\end{align*}

 There exists a maximal isotropic group $J$ of the above mentioned bilinear form which is normal in $\Z D^{\ell_1}(\ti{A})$ with index $q$.
 The character $\widetilde{\psi_A}$ extends to $J$. Let $\widetilde{\widetilde{\psi_A}}$ denotes this extension, then the inertia group of $\widetilde{\widetilde{\psi_A}}$ in $\Z D^{\ell_1}(\ti{A})$ is $J$ itself. By the Heisenberg lift, $\theta=\ind_{J}^{\Z D^{\ell_1}(\ti{A})}(\widetilde{\widetilde{\psi_A}})$ is a unique irreducible character of  $\Z D^{\ell_1}(\ti{A})$ of degree $q$ lying above $\widetilde{\psi_A}$. Now $\theta$ is invariant under $S_A$ and $\frac{S_A}{\Z D^{\ell_1}(\ti{A})}$ is a cyclic group. Hence we can extend $\theta$ to a character $\phi$  of $S_A$. By Clifford theory, the representation $ \ind_{S_A}^{\G(\co_\ell)}\phi$ of $\G(\co_\ell)$ is an irreducible cuspidal representation of $\G(\cO_\ell)$ lying above $\psi_A$. Moreover, every cuspidal representation of $\G(\cO_\ell)$ lying above $\psi_A$ is of this form. For proofs see \cite[Section 3.3.2]{MR2584957} for $\GL_2$ and \cite[Section 4.H.2, Page-48]{Campbell-thesis} for $\GU_2$. The following result is directly obtained from the above construction.    
\begin{proposition}\label{prop:inner prod from Subgp} 
    Let $A\in \god$ be cuspidal and $H$ be a subgroup of  $\Z D^{\ell_2}(\tilde{A})$ such that $\K^{\ell_2}\leq H\leq \Z D^{\ell_2}(\tilde{A}).$ For  $\phi_1,\phi_2\in \mathrm{Irr}(H\mid \psi_A),$ we have  $\langle \mathrm{Ind}_H^{\G(\cO_\ell)}(\phi_1) ,\mathrm{Ind}_H^{\G(\cO_\ell)}(\phi_2)\rangle\neq 0$ if and only if $\phi_1=\phi_2.$  
\end{proposition}

\section{Results related to the construction of representations of $\G(\cO_\ell)$}
\label{sec:related-results-construction}
In this section, we list several results related to the construction as given in \autoref{subsec:construction}. While these may be well known to the experts but we could not find them explicitly stated in the literature. Therefore, for completeness, we include their statements and proofs. We use the notations of \autoref{subsec:construction} in this section. 

Throughout this section, we assume $A_1,A_2\in \g(\cO_{\ell_1})$ are regular matrices such that $A_1+A_2$ is  regular and 
    $\tt(A_1)= \tt(A_2)=\cus.$ 
    For $\phi_i \in \mathrm{Irr}(S_{A_i}\mid \psi_{A_i})_{1 \leq i \leq 2}$, let  
\[
\mathbf{W}(\phi_1,\phi_2):=\mathrm{Res}^{S_{A_1}}_{S_{A_1}\cap S_{A_2}}(\phi_1)\otimes\mathrm{Res}^{S_{A_2}}_{S_{A_1}\cap S_{A_2}}(\phi_2).
    \]
We prove the following  result in this section and this will be crucially used to prove \autoref{thm:main-theorem} for $\Xi_3$ (cuspidal tensor cuspidal case) in \autoref{sec:proof-of-them-s1-s3}.     

\begin{thm}\label{thm:SA multiplicity free}
     The representation $\mathbf{W}(\phi_1,\phi_2)$
    is multiplicity free.
\end{thm}

We first include a few preliminary results that we require for the proof of \autoref{thm:SA multiplicity free}. Recall, for cuspidal $A=\smat{0}{\nonsq \alpha}{\nonsq}{0}\in \g(\co_{\ldown})$, we defined $ D^{\ell_i}(\ti{A})$ by $D^{\ell_i}(\ti{A}) = (\C_{\G(\cO_\ell)}(\ti{A})\cap \K^1)\K^{\ell_i}$  for $i \in \{1,2\}$ in \autoref{subsec:alternate const for cus}.

\begin{proposition}\label{prop:character values}
Let $\ell$ be odd and $A=\smat{0}{\nonsq \alpha}{\nonsq}{0}\in \g(\cO_{\ell_1})$ be regular such that $\tt(A)=\cus.$
    For  $\phi  \in \mathrm{Irr}(S_{A}\mid \psi_A),$ the character $\chi_\phi$ of $\phi$ satisfies the following:
    \begin{enumerate}
        \item $\chi_\phi (g)=q\widetilde{\psi_A}(g)$ for all $g\in \Z D^{\ell_2}(\tilde{A}),$ where $\widetilde{\psi_A}\in \mathrm{Irr}(\Z D^{\ell_2}(\tilde{A})\mid \psi_A)$ such that \\ $\langle \mathrm{Res}^{S_A}_{\Z  D^{\ell_2}(\tilde{A})}(\phi),\widetilde{\psi_A}\rangle\neq 0.$
        \item $\chi_\phi (g)=0$ for all $g\in \Z D^{\ell_1}(\tilde{A})\setminus \Z D^{\ell_2}(\tilde{A}).$
        \item $ |\chi_\phi(g)|=1$  for all $ g \in S_A\setminus\Z D^{\ell_1}(\tilde{A}).$
    \end{enumerate}
\end{proposition}

\begin{proof}
The proof of (1) and (2) follow from  \autoref{subsec:alternate const for cus}.  
For (3), the result for  $\G=\GL_2$, up to minor changes, was  obtained in \cite[Lemma 5.7]{MR2584957}. We use their ideas to prove the result  uniformly for both $\GL_2$ and $\GU_2$.   
Consider the representation $\Gamma$ of $S_A$ on the vector space $M_{q}(\mathbb{C})$ defined by  $\Gamma(g)(B)= \phi(g) B \phi(g)^{-1}$ for $g\in S_A$ and $B\in M_{q}(\mathbb{C}).$ 
By direct computations with usual basis of $M_{q}(\mathbb{C}),$  it is easy see that it's character $\chi_{\Gamma}= \chi_\phi \overline{\chi_\phi}.$  
Therefore, to show  (3), it is enough to prove that $\chi_\Gamma(g)=1$ for all 
$g\in S_A \backslash \Z D^{\ldown}(\ti{A}).$

From \autoref{subsec:alternate const for cus}, we have  $\mathrm{Res}_{\Z D^{\ldown}(\ti{A})}^{S_A}(\phi)$ is irreducible. Therefore  the $\mathbb{C}$-span of the set $\{\phi(h): h \in \Z D^{\ldown}(\ti{A})\}$ is equal to  $M_{q}(\mathbb{C}).$  Let $\{h_j \mid j \in [1, q^2]  \}\subseteq  \K^{\ldown}$ be a set of coset representatives for $\Z D^{\lup}(\ti{A})$ in $\Z D^{\ldown}(\ti{A})$. Without loss of generality, assume that $h_1=\I.$ 
We claim that for every $h \in \Z D^{\ldown}(\ti{A})$,  $\phi(h)=\widetilde{\psi_A}(h_j^{-1}h)\phi(h_j)$ where $j \in [1, q^2] $ such that $h\in h_j\Z D^{\lup}(\ti{A}).$ Note that $h=h_j(h_j^{-1}h)$ and $h_j^{-1}h\in \Z D^{\lup}(\ti{A}).$
By  (1), we have  $\phi\left(h_j^{-1}h\right)=\widetilde{\psi_A}(h_j^{-1}h)  \I.$ Therefore   $\phi(h)=\widetilde{\psi_A}(h_j^{-1}h)\phi(h_j)$ and  hence the claim follows.
Note that the claim implies that the set $\{\phi(h_j) \mid j \in [1, q^2]\}$ is a generating set of $M_{q}(\mathbb{C})$. Since  dimension of  $M_{q}(\mathbb{C})$  is $q^2,$ the set  $\{\phi(h_j)  \mid j \in [1, q^2]\}$   must form a $\mathbb{C}$-basis of $M_{q}(\mathbb{C})$.

Let $g\in S_A \backslash \Z D^{\ldown}(\ti{A}).$ Then for $j \in [1, q^2],$ we have $\Gamma(g)(\phi(h_j))=\phi(gh_jg^{-1}).$ Since $gh_jg^{-1}\in \Z D^{\ldown}(\ti{A}),$ by the claim, we must have $\Gamma(g)(\phi(h_j))=\widetilde{\psi_A}(h_{m_j}^{-1}gh_jg^{-1})\phi(h_{m_j})$ where  $m_j \in [1, q^2]$ such that $gh_jg^{-1}\in h_{m_j}\Z D^{\lup}(\ti{A}).$ Therefore 
\begin{equation}\label{eqn:zg0}
     \chi_\Gamma(g)= \sum_{j\in [1, q^2] ; m_j=j } \widetilde{\psi_A}(h_{j}^{-1}gh_jg^{-1}).
\end{equation}
We claim that  for $j \in [1, q^2],$ if $h_j^{-1}gh_jg^{-1} \in \Z D^{\lup}(\ti{A}),$ then $h_j\in \Z D^{\lup}(\ti{A})$ (i.e, $j=1$ and $h_j=\I$). By assuming the claim, from \autoref{eqn:zg0}, we obtain that $ \chi_\Gamma(g)=  \widetilde{\psi_A}(\I)=1.$ Hence (3) follows.  

To show the claim,
let $h_j=\I+\pi^{\ldown}C_j $ for some matrix $C_j\in \mathrm{M}_2(\Lri_\ell)$. Then  
$
 h_j^{-1}gh_jg^{-1}=\I+\pi^{\ldown}(gC_jg^{-1}-C_j)+\pi^{2\ldown}(C_j^2-C_jgC_jg^{-1}).
$
 Therefore, if $h_j^{-1}gh_jg^{-1} \in \Z D^{\lup}(\ti{A}),$ then 
\begin{equation}\label{eqn:zg1}
    (gC_jg^{-1}-C_j)\ti{A}=\ti{A}(gC_jg^{-1}-C_j) \mod (\pi).
\end{equation}
By multiplying both sides of  \autoref{eqn:zg1} with $g$ (from left) and 
 rearranging terms, we obtain that  $g(C_jg^{-1}\ti{A}g-\ti{A}C_j)=(C_j\ti{A}-\ti{A}C_j)g \mod (\pi).$
Since $g\in S_A \backslash \Z D^{\ldown}(\ti{A})$,  $g=x\I+y\ti{A}\mod (\pi)$ for some $x\in \rl$, $y\in  \rl^\times.$ Therefore,  we must have
\begin{equation}
\label{eqn:equaton 5.2}
    \ti{A}(C_j\ti{A}-\ti{A}C_j)=(C_j\ti{A}-\ti{A}C_j)\ti{A} \mod (\pi).
\end{equation}
Assume $C_j=\smat{a}{b}{c}{d}$. Then
\begin{align*}
  C_j\ti{A}-\ti{A}C_j&=\mat{a}{b}{c}{d} \mat{0}{\nonsq \ti{\alpha}}{\nonsq}{0} -\mat{0}{\nonsq \ti{\alpha}}{\nonsq}{0}\mat{a}{b}{c}{d}\\
  &=\nonsq\mat{b-\ti{\alpha} c}{\ti{\alpha}(a-d)}{d-a}{c\ti{\alpha}-b}.
\end{align*}
Since $\ti{A}$ is regular,  \autoref{eqn:equaton 5.2} implies that  $C_j\ti{A}-\ti{A}C_j=z\I+w\ti{A} \mod (\pi)$ for some $z,w \in \rl$. This along with $\ti{\alpha}\in \rl^\times$ gives, $b=\ti{\alpha} c \mod (\pi)$ and $a=d \mod (\pi),$ i.e, $C=a\I+c\ti{A} \mod(\pi)$.
This implies $C_j\ti{A}=\ti{A}C_j \mod (\pi),$ which is equivalent    $h_j\in \Z D^{\lup}(\ti{A})$.  Hence the claim.
\end{proof}

Define $ t := \max \{ i \in [0, \ldown] \mid A_1 A_2 = A_2 A_1 \mod (\pi^i)\}$ and 
  $\Delta\coloneqq\begin{cases}
    1, & \text{for } \G=\GU_2;\\
     -1 ,& \text{for } \G=\GL_2.
\end{cases}$
\begin{lemma}
\label{lem:Stab-cardinality} 
    \begin{enumerate}
        \item For $t <  \ell_1$,  $ S_{A_1}\cap S_{A_2}=(\{x\I+\pi^{\ell_1-t} y \tilde{A_1} \mid x,y \in \rl\}\cap\G(\cO_{\ell})) \K^{\ell_1}   $ and $|S_{A_1} \cap S_{A_2}|=(q+\Delta)q^{4\lup +\ldown +t-1}.$
    \item For $t = \ldown$, we have  $S_{A_1}\cap S_{A_2}=S_{A_1}=S_{A_2}$ and 
    $|S_{A_1}|=(q+1)(q+\Delta)q^{3\ell-1}.$
    \end{enumerate}
\end{lemma}
\begin{proof}
For (1), it is easy to see that $(\{x\I+\pi^{\ell_1-t} y \tilde{A_1} \mid x,y \in \rl\}\cap\G(\cO_{\ell})) \K^{\ell_1} \subseteq S_{A_1}\cap S_{A_2}$. To prove the converse, let $g\in S_{A_1}\cap S_{A_2}$. Then $g=u\I+v \tilde{A_1}=zI+w \tilde{A_2} \mod (\pi^{\ldown})$ for some $u,v,z,w\in \rl$. This gives $v \tilde{A_1} = (z-u)\I + w \tilde{A_2}  \mod (\pi^{\ldown})$. Hence $v \tilde{A_1}$ commutes with $ \tilde{A_2}$ modulo $ (\pi^{\ell_1} )$. i.e., 
\begin{equation}\label{eqn:df1}
    v(\tilde{A_1}\tilde{A_2}-\tilde{A_2}\tilde{A_1})=0 \mod(\pi^{\ldown}).
\end{equation}
Since $t< \ell_1,$ $\tilde{A_1}\tilde{A_2}-\tilde{A_2}\tilde{A_1}=\pi^{t}B$ for some $B\in M_2(\rl)$ such that $B\neq 0 \mod (\pi). $ Therefore \autoref{eqn:df1} implies  $v=\pi^{\ldown-t}v'$ for some $v'\in \rl$. Therefore $g=u\I+\pi^{\ldown-t}v' \tilde{A_1} \mod(\pi^{\ldown})$ which implies that  $g\in (\{x\I+\pi^{\ell_1-t} y \tilde{A_1} \mid x,y \in \rl\}\cap\G(\cO_{\ell})) \K^{\ell_1}.$ 
This proves that $S_{A_1}\cap S_{A_2} \subseteq (\{x\I+\pi^{\ell_1-t} y \tilde{A_1} \mid x,y \in \rl\}\cap\G(\cO_{\ell})) \K^{\ell_1}$. Next, to find $|S_{A_1}\cap S_{A_2}|,$ note that 
\begin{equation}\label{eqn:df2}
    |S_{A_1}\cap S_{A_2}|=\frac{ |\{x\I+\pi^{\ell_1-t} y \tilde{A_1} \mid x,y \in \rl\}\cap\G(\cO_{\ell})|\times|\K^{\ldown}|}{|\{x\I+\pi^{\ell_1-t} y \tilde{A_1} \mid x,y \in \rl\}\cap \K^{\ldown}|}
\end{equation}
It is easy to see that  $|\K^{\lup}|=q^{4\ell_1}.$ Using the fact that $\K^{\ldown}/\K^{\lup} \cong \g(\cO_1)$, we obtain $|\K^{\ldown}|=q^{4\ell_2}.$ 

  For $\G=\GL_2$, since $ x\I+\pi^{\ell_1-t} y \tilde{A_1} \in \GL_2(\co_\ell)$ if and only if $x\in \rl^\times $, we obtain that $|\{x\I+\pi^{\ell_1-t} y \tilde{A_1} \mid x,y \in \rl\}\cap\G(\cO_{\ell})|=(q-1)q^{\ell-1}\times q^{\ell_2+t}.$
  Similarly, since $ x\I+\pi^{\ell_1-t} y \tilde{A_1} \in \K^{\ldown}$ if and only if $x\in 1+\pi^{\ell_1}\rl $ and $ \pi^{\ell_1-t} y \in \pi^\ldown \rl $, we  obtain that $|\{x\I+\pi^{\ell_1-t} y \tilde{A_1} \mid x,y \in \rl\}\cap\K^{\ldown}|=q^{\ell_2}\times q^{\ell_2}.$ By substituting these values in \autoref{eqn:df2}, we obtain that $|S_{A_1}\cap S_{A_2}|=(q-1)q^{4\lup +\ldown +t-1}$

For $\G=\GU_2$, note that   $x\I+\pi^{\ell_1-t} y \tilde{A_1} \in \GU_2(\co_\ell)$ if and only if $x\in \rl^\times$ and  there exists $r\in \cO_\ell$ such that  $\pi^{\ell_1-t}y= \pi^{\ell_1-t} r \nonsq x $ and $xx^\circ(1-\pi^{2(\ldown-t)}\nonsq^2r^2\tilde{\alpha_1})=1.$ Therefore 
\begin{eqnarray*}
|\{x\I+\pi^{\ell_1-t} y \tilde{A_1} \mid x,y \in \rl\}\cap\G(\cO_{\ell})|&=&\left|\left\{(x,\pi^{\ell_1-t} r \nonsq x) \mid  \begin{array}{l}  r\in\cO_\ell ,\, x\in \rl^\times \text{ and }   \\  xx^\circ=(1-\pi^{2(\ldown-t)}\nonsq^2r^2\tilde{\alpha_1})^{-1} \end{array}\right\}\right|\\
    &=& |\pi^{\ell_1-t}\cO_\ell|\times |\{z\in \rl^\times : zz^\circ=1\}|\\
    &=& q^{\ell_2+t}\times (q+1)q^{\ell-1}. 
\end{eqnarray*}
Similarly, note that  $ x\I+\pi^{\ell_1-t} y \tilde{A_1} \in \K^{\ldown}$ if and only if 
$$x\I+\pi^{\ell_1-t} y \tilde{A_1}=\mat{1+\pi^\ldown z}{\pi^\ldown a (1+\pi^\ldown z) \nonsq \tilde{\alpha}}{\pi^\ldown a (1+\pi^\ldown z) \nonsq }{1+\pi^\ldown z}$$
for some $ z\in \rl$ and  $ a\in \cO_\ell$ such that  $(1+\pi^\ldown z)(1+\pi^\ldown z)^\circ = 1+\pi^{2\ldown} a^2  \nonsq^2 \tilde{\alpha}.$
Since the map $x\mapsto xx^\circ$ is a surjective map from $1+\pi^\ldown\rl$ to $1+\pi^\ldown\cO_\ell,$ for a given $a\in \co_\ell,$ we have $|\{x\in 1+\pi^\ldown  \rl  \mid  xx^\circ = 1+\pi^{2\ldown} a^2  \nonsq^2 \tilde{\alpha} \}|=|1+\pi^\ldown\rl|/|1+\pi^\ldown\cO_\ell|=q^\lup.$ Thus 
$$|\{x\I+\pi^{\ell_1-t} y \tilde{A_1} \mid x,y \in \rl\}\cap \K^{\ldown}|=|\pi^\ldown \cO_\ell|\times q^\lup=q^{2\lup}.$$
By substituting these values in \autoref{eqn:df2}, we obtain that $|S_{A_1}\cap S_{A_2}|=(q+1)q^{4\lup +\ldown +t-1}.$

For (2),   $t=\ldown$ implies   $\tilde A_1 \tilde A_2 = \tilde A_2 \tilde A_1 \mod (\pi^\ldown).$ Since $\tilde A_i$ for $i\in \{1,2\}$ are regular matrices, 
we have $S_{A_1} = \tC_{\G(\lri_\ell)}(\tilde{A_1}) \K^{\ldown}$ and $\tC_{\G(\lri_\ell)}(\tilde{A_1})=\{x\I+ y \tilde{A_1} \mid x,y \in \rl\}\cap\G(\cO_{\ell}).$  Therefore $S_{A_1}=S_{A_2}.$  See \cite[Section~3.3]{MR2584957} and \cite[Section~4.H.2, Page-48]{Campbell-thesis} for the expression of     $|S_{A_1}|.$   
\end{proof}

   Define the subsets $\Gamma_i$ for $i \in [1,4]$ of $S_{A_1} \cap S_{A_2}$ by $\Gamma_1 \coloneqq (\Z D^{\ell_2}(\tilde{A_1})) \cap (\Z D^{\ell_2}(\tilde{A_2})),$ $\Gamma_2 \coloneqq (S_{A_1} \setminus (\Z D^{\ell_1}(\tilde{A_1}))) \cap(\Z D^{\ell_2}(\tilde{A_2}))$, $\Gamma_3 \coloneqq (\Z D^{\ell_2}(\tilde{A_1}))\cap (S_{A_2} \setminus (\Z D^{\ell_1}(\tilde{A_2})))$ and $\Gamma_4 \coloneqq (S_{A_1} \setminus (\Z D^{\ell_1}(\tilde{A_1}))) \cap(S_{A_2} \setminus (\Z D^{\ell_1}(\tilde{A_2}))).$ First note that $\Gamma_2=\Gamma_3=\emptyset$. The following description of $\Gamma_1$ and $\Gamma_4$ will be useful.  
   \begin{enumerate}
         \item $\Gamma_1=(\{x\I+\pi^{\ell_2-t} y\tilde{A_1} \mid x,y\in \rl\}\cap\G(\cO_{\ell})) \K^{\ell_2}.$ 
    \item $
       \Gamma_4 =\begin{cases}
       S_{A_1}\setminus(\Z D^{\ell_1}(\tilde{A_1})), & \mathrm{if}\, t=\ell_1;\\
        \emptyset, & \mathrm{if}\, t<\ell_1. \end{cases}$
        \end{enumerate}
      By using the same ideas as the proof of \autoref{lem:Stab-cardinality}(1), we also obtain  $|\Gamma_1|=(q+\Delta)q^{4\ldown +\lup +t-1}$. Further $|\Z  D^{\ell_i}(\tilde{A_j})|=(q+\Delta) q^{4\ell -2\ell_i -2}$ for $i,j \in \{1,2\}$ are easy to prove for $\G = \GL_2$ and follow from  \cite[Section~4.H.2, Pages~53--54]{Campbell-thesis} for $\G = \GU_2$.

\begin{lemma}
\label{lem:Wq}
For odd $\ell$, 
  we have $\langle \mathbf{W}(\phi_1,\phi_2), \mathbf{W}(\phi_1,\phi_2)\rangle=q$.
\end{lemma}
\begin{proof}
For $i \in \{1,2\},$ by \autoref{prop:character values}, we have
$$|\chi_{\phi_i} (g)|=\begin{cases}
    q, & g\in  \Z D^{\ell_2}(\tilde{A_i});\\
    0, &  g \in (\Z D^{\ell_1}(\tilde{A_i})) \setminus (\Z D^{\ell_2}(\tilde{A_i})); \\
    1, & g \in S_{A_i} \setminus (\Z D^{\ell_1}(\tilde{A_i})).
\end{cases}$$
 Therefore 
\begin{eqnarray*}
    \langle \mathbf{W}(\phi_1,\phi_2), \mathbf{W}(\phi_1,\phi_2)\rangle&=&\frac{1}{|S_{A_1}\cap S_{A_2}|}\sum_{g\in S_{A_1}\cap S_{A_2} }|\chi_{\phi_1} (g)|^2|\chi_{\phi_2}(g)|^2 \\
     &=& \frac{1}{|S_{A_1}\cap S_{A_2}|}(q^4|\Gamma_1| + q^2(|\Gamma_2| + |\Gamma_3|)+ |\Gamma_4|). 
\end{eqnarray*}
where $\Gamma_j$ for  $j\in [1,4] $ are as defined above. For $t< \ell_1,$ using $\Gamma_2 = \Gamma_3= \Gamma_4 = \emptyset$ and 
 \autoref{lem:Stab-cardinality}(1), we obtain
$$ \langle \mathbf{W}(\phi_1,\phi_2), \mathbf{W}(\phi_1,\phi_2)\rangle=\frac{q^4 \times |\Gamma_1|}{|S_{A_1}\cap S_{A_2}|}=q.$$
For $t= \ell_1,$ 
$\Gamma_1=\Z D^{\ell_2}(\tilde{A_1})$ and $\Gamma_4=S_{A_1} \setminus (\Z D^{\ell_1}(\tilde{A_1})).$
By \autoref{lem:Stab-cardinality}(2) and using $|\Z  D^{\ell_i}(\tilde{A_j})|$ from above, we obtain  
\begin{equation*}
    \langle \mathbf{W}(\phi_1,\phi_2), \mathbf{W}(\phi_1,\phi_2)\rangle=\frac{q^4 \times |(\Z D^{\ell_2}(\tilde{A_1})) |}{|S_{A_1}|} + \frac{|S_{A_1} \setminus (\Z D^{\ell_1}(\tilde{A_1}))  |}{|S_{A_1}|}=q.
    \end{equation*}
\end{proof}

\begin{proof}[Proof of \autoref{thm:SA multiplicity free}]  For even $\ell,$ both $\phi_1$ and $\phi_2$ are one dimensional. Therefore $\mathbf{W}(\phi_1,\phi_2)$ is one dimensional and hence multiplicity free. Assume $\ell$ is odd. We first claim that each irreducible constituent  of $\mathbf{W}(\phi_1,\phi_2)$  has dimension $q.$
    Note that $\K^{\ell_{2}}\leq S_{A_1}\cap S_{A_2} \leq S_{A_1+A_2}.$
    Since $\mathrm{Res}^{S_{A_i}}_{\K^{\ell_{2}}}(\phi_i)=q \psi_{A_i}$, we obtain $\mathrm{Res}^{S_{A_1}\cap S_{A_2}}_{\K^{\ell_{2}}}(\mathbf{W}(\phi_1,\phi_2))=q^2 (\psi_{A_1}\otimes \psi_{A_2})=q^2 \psi_{A_1+A_2}.$ Therefore any irreducible constituent  of $\mathbf{W}(\phi_1,\phi_2)$ belongs to $\mathrm{Irr}(S_{A_1}\cap S_{A_2} \mid \psi_{A_1+A_2}).$ Since $A_1+A_2$ is  regular and $\K^{\ell_1}\leq S_{A_1}\cap S_{A_2}\leq S_{A_1+A_2},$ each irreducible constituent  of $\mathbf{W}(\phi_1,\phi_2)$ has dimension $q$, by \autoref{prop:MF from Subgp}(1).  

 Let $\mathbf{W}(\phi_1,\phi_2)=m_1 \theta_1 \oplus m_2\theta_2 \oplus \cdots \oplus m_r\theta_r,$  where  $\theta_i$ for $ i \in [1,r]$ 
 are the in-equivalent irreducible constitutes of $\mathbf{W}(\phi_1,\phi_2)$ with multiplicities $m_i$. Since  $\dim(\theta_i)=q$ for all $ i \in [1,r]$ 
 and $\dim(\mathbf{W}(\phi_1,\phi_2))=q^2,$ we must have $\sum_{i=1}^{r}m_i q=q^2$ and hence $\sum_{i=1}^{r}m_i =q.$ By \autoref{lem:Wq},  $\langle \mathbf{W}(\phi_1,\phi_2), \mathbf{W}(\phi_1,\phi_2)\rangle=q.$  Hence $\sum_{i=1}^{r}m_i^2 =q.$ Since 
$m_i$'s are positive integers, the equality $\sum_{i=1}^{r}m_i =q=\sum_{i=1}^{r}m_i^2$ gives 
$m_i=1$ for all $i\in [1,r]$. Hence $\mathbf{W}(\phi_1,\phi_2)$ is a multiplicity free representation. 
\end{proof}

\section{Description of $S_{A_1} \backslash G /S_{A_2}$ for $\Xi_1, \Xi_2$ and $\Xi_3$} 
\label{sec:double-coset-description}

In this section, we carry out Step (A) of our analysis for $\Xi_1, \Xi_2$ and $\Xi_3$ that is, we give various results to describe $S_{A_1} \backslash G /S_{A_2}$ for these cases. Throughout this section, we use $\tilde{A}\in \g(\cO_\ell)$ to denote a Serre lift of $A \in \g(\cO_\ldown)$. Further,  we use $\tilde{x} \in \cO_\ell$ to denote a lift of $x \in \cO_\ldown$.

For $A_1, A_2 \in \god$ and $g \in \Gol$, define the set $W_g(A_1, A_2)$ by 
	$$ W_g(A_1, A_2)\coloneqq \{S_{A_1} h S_{A_2} \mid h  \in \G(\cO_{\ell}) \text{ and } \ti{A_1}+g \ti{A_2}g^{-1} \sim \ti{A_1}+h\ti{A_2}h^{-1} \mod (\pi^{\ldown})\}.$$
    Whenever $A_1, A_2$ are clear from the context, we shall denote $W_g(A_1, A_2)$ by $W_g$ itself. In this section, our focus is on describing $|W_g (A_1, A_2)|$ for the following cases:
\begin{enumerate}
    \item $\tt(A_1) = \ss$, $\tt(A_2) = \sns$.
    \item $\tt(A_1) = \cus$ and $\tt(A_2) \in \{\ss, \sns\}.$
    \end{enumerate}

\begin{lemma}
\label{lem:ss sns g* and g'}
   Let $A_1=\smat{ a}{0}{0}{- a}\in \g(\co_{\ell_1})$ with $a \in \co_{\ldown}^\times$ 
   and $A_2=\smat{0}{ \nonsq\pi\beta}{\nonsq}{0}\in \g(\co_{\ell_1})$ be such that $\tt(A_1)=\ss$ and $\tt(A_2)=\sns$. For  $g=\left[g_{ij}\right]\in \G(\cO_\ell),$ let
        $$ g'= \mat{1}{\frac{g_{12}g_{22}-\pi \tilde{\beta} g_{11}g_{21} }{\det(g)}}{0}{1} \,\, \mathrm{and} \,\,  g^*= \mat{0}{1}{1}{\frac{\pi \tilde{\beta} g_{11}g_{21} - g_{12}g_{22} }{\det(g)}} .$$
    Then $g' , g^* \in \G(\cO_\ell)$ and
	$W_g=\{S_{A_1}h  S_{A_2} \mid h \in \{ g' , g^*\}  \}.$
\end{lemma}
\begin{proof}
 For $\G=\GL_2$,
 it is clear that $g',g^* \in \GL_2(\co_\ell)$. 
 For $\G=\GU_2$, to prove $g',g^* \in \GU_2(\co_\ell)$, it is enough to show that $\frac{g_{12}g_{22}-\pi \tilde{\beta} g_{11}g_{21} }{\det(g)} \in \nonsq \co_\ell.$ 
 Since $g\in \GU_2(\co_\ell),$ we have   $\{ g_{12},g_{22}\} \cap \Lri_\ell^\times \neq \emptyset$. First assume $g_{12} \in \Lri_\ell^\times$. Using $g_{21}=(1-g_{11}g_{22}^\circ){g_{12}^\circ}^{-1}$ and $g_{12}g_{22}^\circ=-g_{12}^\circ g_{22}$, we get $\det(g)=-g_{12}{g_{12}^\circ}^{-1}$. 
 Then, using $g_{12}g_{22}^\circ=-g_{12}^\circ g_{22}$, $g_{12}g_{11}^\circ=-g_{12}^\circ g_{11}$ and $g_{11} g_{21}^\circ=-g_{11}^\circ g_{21}$, we obtain
\begin{align*}
    \frac{g_{12}g_{22}-\pi \tilde{\beta} g_{11}g_{21} }{\det(g)}+ 
    \left( \frac{g_{12}g_{22}-\pi \tilde{\beta} g_{11}g_{21} }{\det(g)}\right)^\circ&= \frac{g_{12}^\circ (\pi\tilde{\beta} g_{11}g_{21}-g_{12}g_{22})}{g_{12}}+\frac{g_{12}(\pi\tilde{\beta} g_{11}^\circ g_{21}^\circ -g_{12}^\circ g_{22}^\circ)}{g
    _{12}^\circ}\\
    &=\pi\tilde{\beta} \left(\frac{g_
    {12}^\circ g_{11}g_{21}}{g_{12}}+\frac{g_{12}g_{11}^\circ g_{21}^\circ}{g_{12}^\circ }\right)-\left(g_{12}^\circ g_{22} + g_{12}g_{22}^\circ  \right)\\
    &=\pi\tilde{\beta} \left(\frac{-g_
    {12} g_{11}^\circ g_{21}}{g_{12}}+\frac{-g_{12}^\circ g_{11} g_{21}^\circ}{g_{12}^\circ }\right)\\
    &=-\pi \tilde{\beta} (g_{11}^\circ g_{21}+ g_{11} g_{21}^\circ)=0.
\end{align*}
Therefore $\frac{g_{12}g_{22}-\pi \tilde{\beta} g_{11}g_{21} }{\det(g)} \in \nonsq \co_\ell.$
For $g_{22}\in \Lri^\times$, we can similarly prove $\frac{g_{12}g_{22}-\pi \tilde{\beta} g_{11}g_{21} }{\det(g)} \in \nonsq \co_\ell$. 

Let $h=[h_{ij}]\in \G(\cO_\ell)$ be such that $S_{A_1}h  S_{A_2}\in W_g.$  
  Then, by definition of $W_g,$ we obtain
  \begin{equation}
      \label{eqn: ss-sns}
      \det( \ti{A_1}+h\ti{A_2}h^{-1})-\det(\ti{A_1}+g\ti{A_2}g^{-1})=0 \mod (\pi^{\ldown} ).
  \end{equation}
  We show that  either $ S_{A_1} hS_{A_2}=S_{A_1} g'S_{A_2}$ or $S_{A_1} hS_{A_2}= S_{A_1} g^*S_{A_2}.$
	Since  $h\in \G(\cO_\ell),$ we must have either $h_{12}\in \rl^\times$ or  $h_{22}\in \rl^\times.$ 
    
    For $h_{12}\in \rl^\times,$ choose $x\in \rl^\times$, $y=-h_{11}x(h_{12}\nonsq)^{-1}$ and $B=\smat{ h_{12} x^{-1}(h_{12}^2-h_{11}^2 \pi \tilde{\beta})^{-1}}{0}{0}{-h_{12}x^{-1} \det(h)^{-1}}
    $. 
    Then,  by direct computation,
    \[
    B h (x\mathrm{I}+y\ti{A_2})-g^*=\mat{0}{0}{0}{\frac{\det(\ti{A_1}+h\ti{A_2}h^{-1})-\det(\ti{A_1}+g\ti{A_2}g^{-1})}{2a\nonsq}}.
    \]
   For $\G=\GL_2$, it is clear that $B\in \mathrm{C}_{\GL_2(\cO_{\ell})}(\tilde{{A_1}})$ and $(x\mathrm{I}+y\ti{A_2})\in \mathrm{C}_{\GL_2(\cO_{\ell})}(\tilde{{A_2}})$ for any $x\in \rl^\times$. For $\G=\GU_2$ we choose $x$ to be a solution of the equation $xx^\circ=\frac{h_{12} h_{12}^\circ}{h_{12} h_{12}^\circ +\pi \tilde{\beta} h_{11} h_{11}^\circ}$. Using this choice of $x$ and the fact that $h\in \GU_2(\co_\ell)$ with $\det(h)=-h_{12} {h_{12}^\circ}^{-1}$, we can easily show that $B\in \mathrm{C}_{\GU_2(\cO_{\ell})}(\tilde{{A_1}})$ and $(x\mathrm{I}+y\ti{A_2})\in \mathrm{C}_{\GU_2(\cO_{\ell})}(\tilde{{A_2}}).$ Using \autoref{eqn: ss-sns}, we get $B h (x\mathrm{I}+y\ti{A_2})-g^*=0 \mod (\pi^{\ldown} ).$ 
	Hence, we obtain $S_{A_1} hS_{A_2}=S_{A_1} g^*S_{A_2}.$

    For $h_{22}\in \rl^\times$,  choose $x\in \rl^\times$, $y=-h_{21}x(h_{22}\nonsq)^{-1}$ and $B=\smat{h_{22} x^{-1} \det(h)^{-1}}{0}{0}{h_{22}x^{-1} (h_{22}^2-h_{21}^2 \pi \tilde{\beta})^{-1}}$. 
    Then,  by direct computation,
    \[
    B h (x\mathrm{I}+y\ti{A_2})-g'=\mat{0}{\frac{\det(\ti{A_1}+g\ti{A_2}g^{-1})-\det(\ti{A_1}+h\ti{A_2}h^{-1})}{2a\nonsq}}{0}{0} .
    \]
 Now, for $\G=\GU_2$, we choose $x$ to be a solution of the equation $xx^\circ=\frac{h_{22} h_{22}^\circ}{h_{22} h_{22}^\circ +\pi \tilde{\beta} h_{21} h_{21}^\circ}$.  The rest of the argument then follows similarly to the previous case, and we obtain $B h (x\mathrm{I}+y\ti{A_2})-g'=0 \mod (\pi^{\ldown} )$, which implies $S_{A_1} hS_{A_2}=S_{A_1} g' S_{A_2}.$
\end{proof}

\begin{thm}\label{thm:distinct non irred mAB(C)}
	 Let $A_1=\smat{ a}{0}{0}{- a}\in \g(\co_{\ell_1})$ with $a \in \co_{\ldown}^\times$ 
   and $A_2=\smat{0}{ \nonsq\pi\beta}{\nonsq}{0}\in \g(\co_{\ell_1})$ be such that $\tt(A_1)=\ss$ and $\tt(A_2)=\sns$. For $g=[g_{ij}]\in \G(\cO_{\ell}), $    
	$$	|W_g| =\begin{cases}
		1, & \text{if }\, g_{12},\,g_{22}\in \rl^\times; \\
		2, & \text{otherwise.}
	\end{cases}$$ 
\end{thm}

\begin{proof}
	 By \autoref{lem:ss sns g* and g'}, the result follows if we show the following:
	\begin{enumerate}
		\item For $ g_{12},\,g_{22}\in \rl^\times ,$  $S_{A_1} g' S_{A_2} =S_{A_1}g^*S_{A_2}.$
		\item If either $ g_{12}\in \pi \rl $ or $ g_{22}\in \pi \rl ,$ then $S_{A_1} g' S_{A_2} \neq S_{A_1}g^*S_{A_2}.$
	\end{enumerate}
	Recall $g'= \smat{1}{\frac{g_{12}g_{22}-\pi \tilde{\beta} g_{11}g_{21} }{\det(g)}}{0}{1}$ and $g^*= \smat{0}{1}{1}{\frac{\pi \tilde{\beta} g_{11}g_{21} - g_{12}g_{22} }{\det(g)}}.$  
	Assume $ g_{12},\,g_{22}\in \rl^\times .$ Define $\lambda:=\frac{g_{12}g_{22}-\pi \tilde{\beta} g_{11}g_{21} }{\det(g)}$. 
    Choose $x\in \rl^\times$, $y=\frac{x }{\nonsq\lambda }$ and 
   $ X=\smat{\lambda x^{-1} }{0}{0}{\lambda x^{-1}(\pi \tilde{\beta} -\lambda^2)^{-1}}.$
  Then,  by direct computation,
    $$X g^* (x\mathrm{I}+y\tilde{{A_2}}) = g' .$$
   For $\G=\GL_2$, it is clear that $X\in \mathrm{C}_{\GL_2(\cO_{\ell})}(\tilde{{A_1}})$ and $(x\mathrm{I}+y\ti{A_2})\in \mathrm{C}_{\GL_2(\cO_{\ell})}(\tilde{{A_2}})$ for any $x\in \rl^\times$. For $\G=\GU_2$, we choose $x$ to be a solution of the equation $xx^\circ=\frac{\lambda\lambda^\circ}{\lambda\lambda^\circ+\pi \tilde{\beta} }$. Using this choice of $x$ and the relation $\lambda^\circ=-\lambda$, which follows from $g'= \smat{1}{\lambda}{0}{1} \in  \GU_2(\co_\ell)$, we can easily show that $X\in \mathrm{C}_{\GU_2(\cO_{\ell})}(\tilde{{A_1}})$ and $(x\mathrm{I}+y\ti{A_2})\in \mathrm{C}_{\GU_2(\cO_{\ell})}(\tilde{{A_2}}).$ Therefore, $S_{A_1} g^* S_{A_2}=S_{A_1} g' S_{A_2}$.

 Next we assume that either $ g_{12}\in \pi \rl $ or $ g_{22}\in \pi \rl .$ Then we have  $ g_{12} g_{22}\in \pi \rl.$ 
If $S_{A_1} g'S_{A_2}= S_{A_1} g^*S_{A_2},$ then 
    there exist $X= \smat{c}{0}{0}{d} \in \C_{\G(\cO_{\ell})}(\tilde{{A_1}})$  and $(x\mathrm{I}+y\tilde{{A_2}})\in \C_{\G(\cO_{\ell})}(\tilde{{A_2}})$ such that $Xg^* (x\mathrm{I}+y\tilde{{A_2}})=g' \mod (\pi^{\ell_1} ).$ 
	By equating $(1,2)^{th}$ entries of both sides, we obtain $c x=\frac{g_{12}g_{22}-\pi \tilde{\beta} g_{11}g_{21}}{\det(g)}  \mod (\pi^{\ell_1} ).$
	Since $ g_{12} g_{22}\in \pi \rl$ and $c\in \rl^\times,$ we  obtain $x \in \pi \rl.$ Hence $  (x\mathrm{I}+y\tilde{{A_2}})\notin \G(\cO_{\ell}).$ It is contradiction to the fact that $(x\mathrm{I}+y\tilde{{A_2}})\in \C_{\G(\cO_{\ell})}(\tilde{{A_2}}) \subseteq \G(\cO_{\ell}).$ Therefore we must have $S_{A_1} g'S_{A_2}\neq S_{A_1} g^*S_{A_2}.$
\end{proof}

\begin{lemma}\label{lem:xI+yA inv for A irr}
	Let $A\in \gl_2(\co_\ell)$ be such that $\tt(A)=\cus$. Then $(x\mathrm{I}+ y A) \in \GL_2(\co_\ell) $ for all $x,y \in \co_\ell$ such that $\{x,y\}\cap \co_\ell^\times \neq \emptyset.$
\end{lemma}
\begin{proof}
	Let $x,y \in \cO_{\ell}$ such that $\{x,y\}\cap \cO_{\ell}^\times \neq \emptyset.$ By direct calculations, we obtain that $ \det(x\mathrm{I}+ y A )= x^2 + \bm{tr}(A)\,xy+\det(A)\,y^2.$  If $y \notin \cO_{\ell}^\times,$ then $x\in \cO_{\ell}^\times$ and  $\det(x\mathrm{I}+ y A )= x^2   \mod (\pi).$ Therefore   $\det(x\mathrm{I}+ y A )\in  \cO_{\ell}^\times $, which gives   $(x\mathrm{I}+ y A) \in \GL_2(\cO_{\ell}) .$ If $ y \in \cO_{\ell}^\times,$ then  $ \det(x\mathrm{I}+ y A )= y^2\left ((\frac{x}{y})^2 + \bm{tr}(A)(\frac{x}{y})+\det(A)\right).$ Since $\tt(A)=\cus$, we must have $(\frac{x}{y})^2 + \bm{tr}(A)(\frac{x}{y})+\det(A)\neq 0 \mod (\pi).$ Therefore  $\det(x\mathrm{I}+ y A )\in  \cO_{\ell}^\times$,   which gives  $(x\mathrm{I}+ y A) \in \GL_2(\cO_{\ell}) .$
\end{proof}
Recall that the residue field  is of odd characteristic. Therefore    $\alpha \in (\cO_\ell^\times)^2$ if and only if $\bar{\alpha} \in (\cO_1^\times)^2. $ We will use this fact without specifically mentioning it. 

\begin{lemma}
\label{lem: 21 entry of coset is zero with cuspidal }
    For $i \in  \{1,2\}$, let  $ A_i\in \g(\cO_{\ldown})$	
  be regular matrices such that $A_1=\smat{0}{\nonsq\alpha }{\nonsq}{0}$ and $\tt(A_1)=\cus$. Then for any $g \in  \G(\co_\ell),$ there exists an element $h \in \G(\co_\ell)$ such that $h_{21}=0$ and $S_{A_1}gS_{A_2}=S_{A_1}{h}S_{A_2}$. 
\end{lemma}
\begin{proof} 
    Let $g=\smat{x}{y}{z}{w} \in \G(\co_\ell)$. Then   $\{x,z\}\cap \rl^\times\neq\emptyset $. 
    We first consider $\G=\GL_2$.  For $x\in \co_\ell^\times$, choose $a=1$ and $b=-zx^{-1}$;  for $ z\in \co_\ell^\times$, choose $a=-xz^{-1}$ and  $b=1$. By \autoref{lem:xI+yA inv for A irr}, we have $(a\mathrm{I}+b\ti{A_1})\in S_{A_1}$. Take $ h= (a\mathrm{I}+b\ti{A_1})g $. Then $S_{A_1} g S_{A_2} = S_{A_1} h S_{A_2}$ and by direct calculation, we obtain $h_{21}=0$. This proves the result for $\G=\GL_2$. 
    
We now assume that $\G=\GU_2$.  For $x\in \Lri_\ell^\times$, the relation $xz^\circ+x^\circ z=0$ gives $zx^{-1} \in \nonsq \cO_\ell$ and   $1+zz^\circ (xx^\circ)^{-1}\tilde{\alpha}=1-(zx^{-1})^2\tilde{\alpha} \in 1-\nonsq^2(\cO_\ell)^2\subseteq \cO_\ell^\times,$ where   $\tilde{\alpha}\in (\co_\ell^\times)^2 $ because $\tt(A_1) = \cus$.
Choose $a$  to be a solution of the equation $aa^\circ=(1+zz^\circ (xx^\circ)^{-1}\tilde{\alpha})^{-1}$ and  $b=-az(\nonsq x)^{-1}$. 
For $x\notin \Lri_\ell^\times$, we have $z\in \Lri_\ell^\times$ and $\tilde{\alpha}+ xx^\circ (zz^\circ)^{-1} \in \cO_\ell^\times.$
Choose $b$  to be a solution of the equation $bb^\circ=-\nonsq^{-2}(\tilde{\alpha}+ xx^\circ (zz^\circ)^{-1})^{-1}$ and  $a=-\nonsq bxz^{-1}$. Then,  using the the relation $xz^\circ+x^\circ z=0,$ we can easily show that $(a\mathrm{I}+b\ti{A_1}) \in S_{A_1}$ in  both the cases  $x\in \Lri_\ell^\times$ and $x\notin \Lri_\ell^\times$. 
Take $ h= (a\mathrm{I}+b\ti{A_1})g $. Then $S_{A_1} g S_{A_2} = S_{A_1} h S_{A_2}$ and by direct calculation, we obtain $h_{21}=0$. This proves the result for $\G=\GU_2$. 
  \end{proof}

For $i \in \{1,2\}$, let $A_i=\left[\begin{smallmatrix}
    0 & \nonsq\alpha_i\\ \nonsq & 0
\end{smallmatrix}\right] \in \g(\co_{\ldown}) \text{ and } g_i=\smat{a_i}{b_i}{0}{c_i} \in \G(\co_\ell)$. 
Here $\ti{A_i}=\smat{0}{\nonsq\ti{\alpha_i}}{\nonsq}{0}\in \g(\co_\ell)$ are Serre lifts of $A_i$.  Define $\dtag$ by 
\[
\dtag  := \mat{a_2 c_1-a_1 c_2}{ \nonsq(a_2 b_1+a_1 b_2 )}{b_2 c_1-b_1 c_2}{\nonsq(b_1 b_2+a_1 a_2 \tilde{\alpha}_2- c_1 c_2\tilde{\alpha}_1)}. 
\]
We will denote $\dtag$ by $D$ whenever the meaning is clear from the context. We now list some of the properties of $\dtag.$
\begin{lemma}
\label{lem-detD-computation}
We have $ \mathrm{det}(\ti{A_1}+g_1\ti{A_2}g_1^{-1})-\mathrm{det}(\ti{A_1}+g_2\ti{A_2}g_2^{-1})=\frac{\nonsq }{a_1a_2c_1c_2}\times \det(\dtag).$ 
\end{lemma}
\begin{proof}
By direct computations,
     \begin{equation*}
    \label{eqn: deteminant condition for diff. type}
      \mathrm{det}(\ti{A_1}+g_1\ti{A_2}g_1^{-1})-\mathrm{det}(\ti{A_1}+g_2\ti{A_2}g_2^{-1})=\frac{\nonsq^2(a_2c_1-a_1c_2)(a_1a_2\ti{\alpha_2} - c_1c_2\ti{\alpha_1})+\nonsq^2 a_2b_1^2 c_2-\nonsq^2 a_1b_2^2c_1}{a_1a_2c_1c_2}  .
    \end{equation*}
    This directly gives the result. 
\end{proof}

\begin{lemma}\label{lem:D=0 implies} 
     Suppose $\dtag =0\mod (\pi^{k} )$ for some $ k \in [1, \ell_1].$ Then the following hold.
    \begin{enumerate}
        \item $a_1^{-1} c_1=a_2^{-1} c_2\mod (\pi^{k}) $ and $b_2=c_1^{-1}c_2b_1 \mod (\pi^{k}).$
        \item $b_1=b_2=0\mod (\pi^{k}) .$
        \item  $ \tilde{\alpha}_2 =  a_1^{-2} c_1^{2} \tilde{\alpha}_1 \mod (\pi^k ).$
        \item For $i\in \{1,2\},$ if  $\ti{A_1} +g_i \ti{A_2}g_i^{-1} $ are regular, then $1+a_i^{-1}c_i \in \rl^\times $.
    \end{enumerate}
\end{lemma}
\begin{proof}
    Note that (1)-(3) directly follows from $D =0\mod (\pi^{k} )$ 
 and the fact that 
     $$b_1=\frac{c_1 \left(a_2 b_1+a_1 b_2\right)-a_1 \left(b_2 c_1-b_1 c_2\right)+b_1 \left(a_2 c_1-a_1 c_2\right)}{2a_2c_1}=\frac{c_1 \nonsq^{-1}D_{12}-a_1 D_{21}+b_1 D_{11}}{2a_2c_1}.$$
     To show (4), observe that
      \begin{eqnarray*}
 \ti{A_1} +g_i \ti{A_2}g_i^{-1}&=&\nonsq\mat{a_i^{-1}b_i}{\tilde{\alpha}_1+ a_i c_i^{-1}\tilde{\alpha}_2 -(a_ic_i)^{-1}b_i^2}{1+a_i^{-1}c_i}{-a_i^{-1}b_i}\\
 &=&
    \nonsq \mat{0}{\tilde{\alpha}_1+ a_i c_i^{-1}( a_1^{-2} c_1^{2} \tilde{\alpha}_1)}{1+a_i^{-1}c_i}{0}\mod (\pi ^k)
 \\
&=&\nonsq(1+a_i^{-1}c_i)\mat{0}{\tilde{\alpha}_1 }{1}{0}\mod (\pi^k ),
\end{eqnarray*}
where the last equality follows because $a_1^{-1} c_1=a_2^{-1} c_2\mod (\pi ^k).$ Therefore, since $ \ti{A_1} +g_i \ti{A_2}g_i^{-1}$ is regular, we must have $1+a_i^{-1}c_i\in \rl^\times.$
\end{proof}

\begin{lemma}
\label{lem: D has some invertible element}
            For $i \in  \{1,2\}$, let  $ A_i\in \g(\co_{\ldown})$	
  be regular matrices such that $A_1=\smat{0}{\nonsq\alpha_1}{\nonsq}{0}$ with $\tt(A_1)=\cus$, and $A_2=\smat{0}{\nonsq\alpha_2}{\nonsq}{0}$ with $\tt(A_2)\in \{\ss,\sns\} $. Let $g_i=\smat{a_i}{b_i}{0}{c_i} \in \G(\co_\ell)$ for $i \in \{1,2\}$. Then there exists $i,j \in \{1,2\} $ such that $\dtag_{ij} \neq 0 \mod (\pi) $.
\end{lemma}
\begin{proof}  
 We consider $\tt(A_2)=\ss$ and $\tt(A_2)=\sns$ cases separately.

  For $\tt(A_2)=\sns$, we show that $\{D_{21},D_{22}\}\cap \rl^\times \neq \emptyset$. Note that $\alpha_2=0 \mod (\pi)$ in this case. Assume on the contrary that $D_{21}=0 \mod (\pi)$ and $D_{22}=0 \mod (\pi)$. Then we obtain $\tilde{\alpha_1} =b_1^2/c_1^2 \mod (\pi)$, which is a contradiction  both when $\G=\GL_2 $ (since $\tilde{\alpha_1}$ is a non-square unit) and when $\G=\GU_2$ (since $\tilde{\alpha_1} \in (\cO_\ell^\times)^2$ and the fact that the ratio of the squares of neighbours of $\smat{a_1}{b_1}{0}{c_1}$ is  in $\nonsq^2 (\cO_\ell)^2$).

For $\tt(A_2)=\ss$, assume on the contrary that $D_{ij}=0 \mod(\pi)$ 
for all $ i,j \in \{1,2\} $. 
By substituting the value of $c_2$  from $D_{11}=0 \mod(\pi)$, i.e. $c_2=a_1^{-1}a_2c_1\mod(\pi)$, in $D_{21}=0 \mod (\pi)$, we get $b_2=a_1^{-1}a_2b_1 \mod (\pi)$.
Then using $D_{12}=0 \mod(\pi)$, we obtain $b_1=b_2=0 \mod (\pi)$.
Therefore,   $D_{22}=0 \mod(\pi)$ and $c_2=a_1^{-1}a_2c_1\mod(\pi)$ imply
    \[
    \tilde{\alpha_1} -\frac{a_1^2}{c_1^2}\tilde{\alpha_2}=0 \mod (\pi).
    \]
    This is a contradiction to the fact that $\tilde{\alpha_2} $ is a square (respectively a non-square) and $\tilde{\alpha_1} $ is a non-square  (respectively a square ) in $\co_\ell^\times$ for $\G=\GL_2$ (resp. $\G=\GU_2$).
\end{proof}

\begin{thm}\label{thm:D theorem}
 	For $i \in \{ 1, 2\}$, let 
    $A_i=\smat{0}{\nonsq\alpha_i}{\nonsq}{0} \in  \g(\cO_{\ell_1}) \, $  such that $\tt(A_1)=\cus$ and $A_2$ is any regular matrix.
    For $i \in \{1,2  \},$ let $g_i=\smat{a_i}{b_i}{0}{c_i} \in \G(\cO_{\ell})$.  The following are equivalent.
 		\begin{enumerate}
 			\item $S_{A_1} g_1 S_{A_2}=S_{A_1} g_2 S_{A_2}.$
     		\item  There exist $x,y\in \rl$  such that $\{x,y\}\cap \rl^\times \neq \emptyset$ and $\dtag \cmat{x}{y}=\cmat{0}{0} \mod (\pi^{\ell_1} ).$
 					\end{enumerate}
 \end{thm}

To prove \autoref{thm:D theorem}, we need the following result.
\begin{proposition}\label{prop:solution for D}
    Let $T=\smat{a}{ b}{\nonsq c}{\nonsq d}\in M_2(\Lri_\ell)$ with $a,b,c,d\in \cO_\ell$ such that $T\neq 0 \mod (\pi).$
     Let $A=\left[\begin{smallmatrix}
       0 & \nonsq\beta \\ \nonsq & 0
   \end{smallmatrix}\right] \in \gu_2(\co_{\ell})$ such that $\tt(A)=\cus.$ For $i\in [1, \ell],$
   if  there exist $x,y\in \Lri_\ell$ such that $\{x,y\}\cap \Lri_\ell^\times\neq \emptyset$ and $T \cmat{x}{y}=\cmat{0}{0}\mod (\pi^{i} ),$ then 
   there exist $x',y'\in \Lri_\ell$ such that  $x'\mathrm{I}+y'A \in \GU_2(\co_\ell)$ and 
	$$T \left[ \begin{matrix}
	x'  \\  y' \end{matrix}\right]=\left[ \begin{matrix}
	0  \\  0 \end{matrix}\right]  \mod (\pi^{i} ) .$$
\end{proposition}
\begin{proof}
Since $T\neq 0 \mod (\pi),$ we have 
    $\{a,b,c,d\}\cap\cO_{\ell}^\times \neq \emptyset.$ 
   We prove the result for  $a\in \cO_{\ell}^\times$. The proof for the remaining cases follow along the same lines. Let $a\in \cO_{\ell}^\times.$
   Since  $\{x,y\}\cap \Lri_\ell^\times\neq \emptyset$ and $a x+  b y=0 \mod (\pi^i),$ we must have $y\in \Lri_\ell^\times$ and $x=- ba^{-1}y\mod (\pi^i).$   Choose $x'=- ba^{-1}yz$ and $y'=yz$ for some $z\in \Lri_\ell^\times$. Then we have the following:
    \begin{eqnarray}
        x'(\nonsq y^{\prime})^\circ+x^{\prime\circ}(\nonsq y^{\prime}) =yy^\circ zz^\circ(\nonsq b a^{-1}-\nonsq b a^{-1})=0.\label{eqn:bnv1}\\
         x'x^{\prime\circ}+(\nonsq y^{\prime})(\nonsq y^{\prime}\beta)^\circ=yy^\circ zz^\circ( b^2 a^{-2}-\nonsq^2\beta).\label{eqn:bnv2}
    \end{eqnarray}
Since $\tt(A)=\cus,$ we have $\beta\in (\cO_\ell^\times)^2 $ and hence $( b^2 a^{-2}-\nonsq^2\beta)\in \cO_\ell^\times.$ Now choose $z\in \Lri_\ell^\times$ such that 
$$zz^\circ=\frac{1}{yy^\circ ( b^2 a^{-2}-\nonsq^2\beta)}.$$
For this choice of $z,$ by \autoref{eqn:bnv1} and \autoref{eqn:bnv2}, we have $x'\mathrm{I}+y'A =\smat{x'}{\nonsq y' \beta}{\nonsq y'}{x'}\in \GU_2(\co_\ell).$ Also, we have
$$T \left[ \begin{matrix}
	x'  \\  y' \end{matrix}\right]=zT \left[ \begin{matrix}
	x  \\  y \end{matrix}\right]=\left[ \begin{matrix}
	0  \\  0 \end{matrix}\right]  \mod (\pi^i ) .$$
This completes the proof. 
\end{proof}

\begin{proof}[Proof of \autoref{thm:D theorem}] 
	 Let  $S_{A_1} g_1 S_{A_2}=S_{A_1} g_2 S_{A_2}.$ 	Then  there exist  $x_1,x_2,y_1,y_2\in \rl$ such that $x_1\mathrm{I}+y_1 \tilde{A}_1  \in S_{A_1}$, $ x_2\mathrm{I}+y_2 \tilde{A}_2 \in S_{A_2}$ and 
\begin{equation}\label{eqn:hg}
    (x_1\mathrm{I}+y_1 \tilde{A}_1)g_1 -g_2(x_2\mathrm{I}+y_2 \tilde{A}_2)  =0 \mod (\pi^{\ell_1} ).
\end{equation}
By direct computation, we have 
$$(x_1\mathrm{I}+y_1 \tilde{A}_1)g_1 -g_2(x_2\mathrm{I}+y_2 \tilde{A}_2)=\mat{a_1 x_1-a_2 x_2- \nonsq b_2  y_2}{b_1
   x_1-b_2 x_2+\nonsq \tilde{\alpha}_1 c_1   y_1 -\nonsq  a_2 \tilde{\alpha}_2  y_2}{\nonsq  \left(a_1 y_1-c_2 y_2\right)}{\nonsq b_1  y_1+c_1
   x_1-c_2 x_2}.$$
Equating the second rows in both sides of \autoref{eqn:hg}, 
we obtain that  $y_2=c_2^{-1} a_1 y_1  \mod (\pi^{\ell_1} )$ and $x_2=c_2^{-1}(\nonsq b_1 y_1+c_1 x_1) \mod (\pi^{\ell_1} ).$
On substituting these values into the first row on the left-hand side of \autoref{eqn:hg} and simplifying, we obtain
 $D\cmat{x_1}{y_1}=\cmat{0}{0} \mod (\pi^{\ell_1} ),$ where $D = \dtag.$
Since $x_1\mathrm{I}+y_1 \tilde{A}_1  \in S_{A_1},$ we must have $\{x,y\}\cap \rl^\times \neq \emptyset.$
  This gives that (1) implies (2).
  
  To show (2) implies (1), let $x,y\in \rl$ such that 
  $\{x,y\}\cap \rl^\times \neq \emptyset$
  and $D \cmat{x}{y}=\cmat{0}{0}\mod (\pi^{\ell_1} ).$ 
  We first claim that we can further assume that $(x\mathrm{I}+y\tilde{A_1}) \in \G(\co_\ell).$
  For $\G=\GL_2 ,$ by \autoref{lem:xI+yA inv for A irr}, we have $(x\mathrm{I}+y\tilde{A_1}) \in \G(\co_\ell).$ 
  For $\G=\GU_2$, by using the fact that $g_i \in \GU_2(\cO_\ell),$ we obtain $c_i={a_i^\circ}^{-1}$ and $b_i=\nonsq a_i t_i$ for some $t_i\in \cO_\ell$.
   Using these in the expression of $D$, we obtain 
\begin{equation}\label{eqn:as1}
    D=\frac{1}{a_1^\circ a_2^\circ}\mat{a_2a_2^\circ - a_1a_1^\circ}{\nonsq^2 a_1a_1^\circ a_2a_2^\circ(  t_1 + t_2) }{\nonsq ( a_2a_2^\circ  t_2 -a_1a_1^\circ t_1)}{\nonsq a_1a_1^\circ a_2a_2^\circ( \nonsq^2 t_1 t_2 + \tilde{\alpha}_2)-\nonsq \tilde{\alpha}_1}=\frac{1}{a_1^\circ a_2^\circ}\mat{d_1}{ d_2}{\nonsq d_3}{\nonsq d_4}
\end{equation}
for some $d_j \in \cO_\ell$  with $j\in [1,4].$ 
If $D= 0  \mod (\pi^{\ell_1}),$ then we choose $x=1$ and $y=0,$ which satisfy 
 $D \cmat{x}{y}=\cmat{0}{0} \mod (\pi^{\ell_1} )$ and $x\mathrm{I}+y\tilde{A_1} =  \mathrm{I} \in \G(\co_\ell).$ 
If $D\neq 0  \mod (\pi^{\ell_1}),$
let $0\leq k < \ell_1$  be such that $D=\pi^k D' $ for some $D'\in M_2(\rl)$ with $D'\neq 0  \mod (\pi ).$ By \autoref{eqn:as1}, we can make sure that $D'=\frac{1}{a_1^\circ a_2^\circ}\smat{d'_1}{ d'_2}{\nonsq d'_3}{\nonsq d'_4}$ for some  $d'_j\in \cO_\ell$ with $ j\in [1,4]$. Since $D \cmat{x}{y}=\cmat{0}{0}\mod (\pi^{\ell_1} ),$ we have $\smat{d'_1}{ d'_2}{\nonsq d'_3}{\nonsq d'_4}\cmat{x}{y}=\cmat{0}{0}\mod (\pi^{\ell_1-k} ).$ Therefore by \autoref{prop:solution for D},
there exist  $x',y'\in \rl$ such that  $x'\mathrm{I}+y'\tilde{A_1} \in \G(\co_\ell)$ and $\smat{d'_1}{ d'_2}{\nonsq d'_3}{\nonsq d'_4}\cmat{x'}{y'}=\cmat{0}{0}\mod (\pi^{\ell_1-k} ).$ Now  choose $x=x'$ and $y=y'$ and hence we obtain that $(x\mathrm{I}+y\tilde{A_1})  \in \G(\co_\ell)$ and  
 $D \cmat{x}{y}=\frac{\pi^k}{a_1^\circ a_2^\circ}\smat{d'_1}{\nonsq d'_2}{\nonsq d'_3}{d'_4}\cmat{x'}{y'} = \cmat{0}{0} \mod (\pi^{\ell_1} ).$ Hence the claim.

  Let $X=(x\mathrm{I} +  y \tilde{A}_1)$ and $Y=c_2^{-1}(\nonsq b_1 y+c_1 x)\mathrm{I}+c_2^{-1} a_1 y \tilde{A}_2  .$ 
  By direct calculation, we have 
  \begin{eqnarray*}
Xg_1-g_2Y&=&\mat{\frac{(a_1 c_2-a_2 c_1)x-\nonsq( a_2 b_1+a_1 b_2)y}{c_2}}{\frac{(b_1 c_2-b_2 c_1)x-\nonsq(b_1 b_2+a_1 a_2 \tilde{\alpha}_2- c_1 c_2\tilde{\alpha}_1)y}{c_2}}{0}{0}\\
&=&-c_2^{-1}\mat{x}{y}{0}{0}D^t=0 \mod (\pi^{\ell_1} ).
  \end{eqnarray*}
  Since $(x\mathrm{I}+y\tilde{A}_1) \in \G(\co_\ell)$,  we have  $X\in S_{A_1}$. Therefore $g_2^{-1}Xg_1\in \G(\cO_\ell).$  
  Since $Y=g_2^{-1}Xg_1\mod (\pi^{\ell_1} )$ and the map $\rho_{\ell, \ldown}:\G(\cO_\ell)\rightarrow \G(\cO_i)$ is a projection,  there exists $Z\in M_2(\rl)$ such that $Y+\pi^{\ell_1}Z\in \G(\cO_\ell).$
Note that $Y+\pi^{\ell_1}Z\in S_{A_2}$ and $Xg_1=g_2(Y+\pi^{\ell_1}Z)\mod (\pi^{\ell_1} ).$ Therefore $S_{A_1} g_1 S_{A_2}=S_{A_1} g_2 S_{A_2}.$ 
    This gives (2) implies (1) and hence  completes the proof.
\end{proof}

\begin{thm}\label{prop: cus with other case |W_g|}
        Let  $ A_1 =\smat{0}{\nonsq\alpha_1 }{\nonsq}{0}$ and  ${A_2}=\smat{0}{\nonsq\alpha_2}{\nonsq}{0}$ be in $\g(\cO_{\ldown})$ with $\tt(A_1)=\cus$ and   $\tt({A_2})\in \{\ss,\sns\}$. For $g\in \G(\cO_{\ell}), $   
  \[
  |W_g| = 1. 
  \]
\end{thm}
\begin{proof}
Let $g_1,g_2\in \G(\cO_{\ell}) $ be such that $S_{A_1}g_1S_{A_2}, S_{A_1}g_2S_{A_2}\in W_g.$ 
    By \autoref{lem: 21 entry of coset is zero with cuspidal }, we can assume that $g_i=\smat{a_i}{b_i}{0}{c_i} \in \G(\co_\ell)$ for $i \in \{1,2 \} $.

    By the definition of $ W_g$, we obtain $\mathrm{det}(\ti{A_1}+g_1\ti{A_2}g_1^{-1})-\mathrm{det}(\ti{A_1}+g_2\ti{A_2}g_2^{-1})=0 \mod (\pi^{\ldown})$.  
     To prove \autoref{prop: cus with other case |W_g|}, we have to show  $S_{A_1}g_1S_{A_2}= S_{A_1}g_2S_{A_2}.$ By \autoref{thm:D theorem}, this is equivalent to showing that  there exist $x,y\in \rl$  such that $\{x,y\}\cap \rl^\times\neq \emptyset$ and 
    \begin{equation*}
    \label{eqn: system of eqn in x,y for diff type }
\dtag
\begin{bmatrix}
    x\\
    y
\end{bmatrix}=\begin{bmatrix}
    0\\0
\end{bmatrix} \mod (\pi^{\ldown}).
    \end{equation*}
For $D = \dtag,$ by \autoref{lem-detD-computation}, we have 
       $ \frac{\nonsq \, \det(D)}{a_1a_2c_1c_2}=\mathrm{det}(\ti{A_1}+g_1\ti{A_2}g_1^{-1})-\mathrm{det}(\ti{A_1}+g_2\ti{A_2}g_2^{-1}).$ 
       Therefore $\mathrm{det}(D)=0\mod (\pi^{\ldown})$. 
    Since $\tt(A_1)=\cus$ and  $\tt(A_2)\in \{\ss,\sns\}$, 
    by \autoref{lem: D has some invertible element}, there exists $ i,j \in \{1,2\} $ such that $D_{ij} \neq 0 \mod (\pi) $.
     Choose $x=D_{i2}$ and $y=-D_{i1}.$ 
      For this choice, we have  $\{x,y\}\cap \rl^\times\neq \emptyset$ and 
$$D\left[ \begin{matrix} x  \\  y \end{matrix}\right] = D\left[ \begin{matrix} D_{i2} \\  -D_{i1} \end{matrix}\right]=
\begin{cases}
    \cmat{0}{-\det(D)} ,& \text{if } i=1;\\
    \cmat{\det(D)}{0}, & \text{if } i=2.
\end{cases}$$
Hence the result follows because $\mathrm{det}(D)=0\mod (\pi^{\ldown})$.
\end{proof}

\section{Proof of  \autoref{thm:main-theorem-2}(1)-(3)}
\label{sec:proof-of-them-s1-s3}

Any regular representation is of the form $\ind_{S_A}^{\G(\cO_\ell)}(\phi)$ for some regular matrix $A$ and an irreducible representation $\phi$ of $S_A$ lying above $\psi_A$. For $\rho_i = \ind_{S_{A_i}}^{\Gol}(\phi_i),$ to determine the multiplicity of a regular representation in the tensor product $\rho_1 \otimes \rho_2,$ we observe that
\begin{eqnarray} 
\label{eqn:decomposition-of-regular-tensor}
\rho_1 \otimes \rho_2 \cong \oplus_{g \in S_{A_1} \backslash \G(\cO_\ell) / S_{A_2}} \ind_{S_{A_1} \cap S_{A_2}^g}^{\G(\cO_\ell)} (\phi_1 \otimes \phi_2^g).
\end{eqnarray}
\subsection{Proof of \autoref{thm:main-theorem-2}(1)-(2)}
Let  $ A_1, A_2 \in \g(\cO_{\ell_1})$ be regular matrices with  $\tt(A_1) \neq \tt(A_2).$ For $i\in \{1,2\}$, let $\phi_i \in \mathrm{Irr}(S_{A_i} \mid \psi_{A_i})$ and $\chi_i\in \mathrm{Irr}(\Z)$ such that  
$\langle \phi_i,\chi_i \rangle _{\Z } \neq 0.$ 
Recall $V(\phi_1, \phi_2) = \mathrm{Ind}_{S_{A_1} \cap S_{A_2}}^{\G(\cO_\ell)}(\phi_1 \otimes \phi_2)$. 
By definition, $\ind_{S_{A_1} \cap S_{A_2}^g}^{\G(\cO_\ell)} (\phi_1 \otimes \phi_2^g) \cong V(\phi_1, \phi_2^g)$ for every $g \in S_{A_1} \backslash \G(\cO_\ell) / S_{A_2}$. We note that $S_{A_1} \cap S_{A_2}^g = \Z \K^{\ldown}$ for every $g \in \G(\cO_\ell),$ for otherwise $\bar{A_1} \in \C_{\G(\cO_1)}(\bar{A_2^g})$ and that is not possible because $\tt(A_1) \neq \tt(A_2^g)$. Since $A_1 + A_2^g$ is regular for $g \in \G(\cO_\ell)$,  every irreducible constituent of $V(\phi_1 \otimes \phi_2^g)$ is a regular representation.

\begin{proposition}\label{prop:different types}
    \begin{enumerate}
    \item An irreducible representation $\rho$ of $\G(\cO_\ell)$ is a sub-representation of $V(\phi_1, \phi_2)$ if and only if $\langle \rho , \psi_{A_1 + A_2} \rangle_{ \K^{\ell_2}} \neq 0$ and  $\langle \rho , \chi_1. \chi_2 \rangle_{\Z } \neq 0.$
    \item $V(\phi_1, \phi_2)$ is a multiplicity free representation of $\G(\cO_\ell)$.
    \item For $g, h \in S_{A_1}\backslash \G(\co_\ell)/S_{A_2}$, one of the following holds:
    \begin{enumerate} 
     \item $V(\phi_1, \phi_2^g) \cong V(\phi_1, \phi_2^h).$
     \item $\mathrm{Hom}_{\G(\co_\ell)}(V(\phi_1, \phi_2^g), V(\phi_1, \phi_2^h)) = 0.$
    \end{enumerate} 
    \end{enumerate}
\end{proposition}
\begin{proof} 
For even $\ell$, this result follows immediately from the construction of the regular representations of $\G(\cO_\ell)$. Hence we will now assume that $\ell $ is odd. 
For (1), if $\rho \in V(\phi_1,\phi_2)$ then  $\langle \rho , \psi_{A_1 + A_2} \rangle_{ \K^{\ell_2}} \neq 0$ and  $\langle \rho , \chi_1. \chi_2 \rangle_{\Z } \neq 0$. For the converse, we first prove that the representation $(\phi_1 \otimes \phi_2)|_{\K^\ldown}$ is a multiplicity free representation of $\K^{\ell_1}.$  Let $H$ be as in \autoref{lem: common max isotropic for diff type}. Then $\bar{H}$ is a maximal isotropic  for $\mathcal{B}_A$ and therefore by the construction of regular representations of $\G(\co_\ell)$, we have $\phi_i|_{\K^\ldown} \cong \ind_{H}^{\K^{\ldown}}f_i$ for some  $f_i\in \mathrm{Irr}(H \mid \psi_{A_i})$. Using the fact that $H$ is a normal subgroup of $\K^{\ldown}$, we have 
\[
 (\phi_1\otimes \phi_2)|_{\K^\ldown} \cong \ind_{H}^{\K^{\ldown}}f_1
 \otimes \ind_{H}^{\K^{\ldown}}f_2 \cong \oplus_{g \in \K^{\ldown}/H} \ind_{H}^{\K^{\ldown}}(f_1\otimes f_2^g).
\] 
Note that $f_1 \otimes f_2^g \in \mathrm{Irr}(H \mid \psi_{A_1+A_2})$ for every $g \in \K^{\ldown}$ and $\mathrm{Rad}_{A_1 + A_2} \subseteq H$. Now to show that $(\phi_1\otimes \phi_2)|_{\K^\ldown}$ is multiplicity free it is enough to show that 
 \[
 f_1 \otimes f_2^g|_{\mathrm{Rad}_{A_1 + A_2}} \neq f_1 \otimes f_2^h|_{\mathrm{Rad}_{A_1 + A_2}}  
 \]
for $g h^{-1} \notin H$. Assume on the contrary that $f_2^g = f_2^h$ for $gh^{-1} \notin H$. Therefore  $\psi_{A_2}(hg^{-1}xg h^{-1}x^{-1})=1$ for all $x \in {\mathrm{Rad}_{A_1 + A_2}}.$ By the definition of $\mathrm{Rad}_{A_2},$ we also have $\psi_{A_2}(hg^{-1}yg h^{-1}y^{-1})=1$ for all $y \in \mathrm{Rad}_{A_2}.$ Since $H$ is generated by $\mathrm{Rad}_{A_2}$ and $\mathrm{Rad}_{A_1 + A_2},$ we obtain \[
\psi_{A_2}(hg^{-1}zg h^{-1}z^{-1})=1
\]
for all $z \in H$. Since $\bar{H}$ is maximal isotropic, we obtain $gh^{-1} \in H$. This is a contradiction to $gh^{-1} \notin H.$ Hence $(\phi_1 \otimes \phi_2)|_{\K^\ldown}$ is a multiplicity free representation of $\Kd$. We note that $(\phi_i)|_{\Z  \K^{\ell_2}}=q \chi_i \psi_{A_i}.$ Therefore, by the general theory of Heisenberg lifts for the construction of $\Z \K^\ldown$ representations, we have
\begin{eqnarray}
\label{eq-decomposition-of-phi1-tensor-phi2}
(\phi_1 \otimes \phi_2)|_{\Z  \K^{\ell_1}}\cong \bigoplus_{W\in\mathrm{Irr}(\Z  \K^{\ell_1}\mid \chi_1\chi_2 \psi_{A_1+A_2})}W.
\end{eqnarray} 
Hence (1) follows. Next, (2) 
 follows from \autoref{prop:MF from Subgp}(2) and \autoref{eq-decomposition-of-phi1-tensor-phi2} and (3) follows from (1) and (2).  
\end{proof}

 For  $ \Xi_1$, there exists at most one double coset representative $ h \in S_{A_1} \backslash \ G /S_{A_2}$
  distinct from $g$ such that $\langle V(\phi_1, \phi_2^g) , V(\phi_1, \phi_2^h)\rangle \neq 0$, by \autoref{lem:ss sns g* and g'}. We also note that in \autoref{thm:distinct non irred mAB(C)},  $|W_g| = 2$ occurs only for the case where $\tt( A_1+gA_2g^{-1})=\ss$.  Further, 
 $V(\phi_1, \phi_2^g)$ is multiplicity free by \autoref{prop:different types}. This combined with \autoref{eqn:decomposition-of-regular-tensor} gives us the proof of \autoref{thm:main-theorem-2}(1).

 Similarly,  by \autoref{prop: cus with other case |W_g|}, \autoref{prop:different types} and \autoref{eqn:decomposition-of-regular-tensor}, we obtain a  proof of \autoref{thm:main-theorem-2}(2). 
 
\subsection{Proof of \autoref{thm:main-theorem-2}(3) }

\begin{proposition}
\label{prop:multiplicity-free-component-cuspidal}    Let $A_1,A_2\in \g(\cO_{\ell_1})$ be  regular matrices such that $\tt(A_1)=\tt(A_2)=\cus $. Suppose $\phi_1\in \mathrm{Irr}(S_{A_1}\mid \psi_{A_1})$ and $\phi_2\in \mathrm{Irr}(S_{A_2}\mid \psi_{A_2}).$  Let $g\in\G(\cO_{\ell})$ be such that the representation $V(\phi_1,\phi_2^g) = \mathrm{Ind}_{S_{A_1}\cap S_{A_2}^g}^{\G(\cO_{\ell})}(\phi_1 \otimes \phi_2^g)$ contains a regular irreducible representation as a constituent.  Then $V(\phi_1,\phi_2^g)$ is a multiplicity free representation of $\G(\cO_\ell)$.
\end{proposition}
\begin{proof} 
Since $V(\phi_1,\phi_2^g)$ contains a regular representation,
the matrix $\ti{A_1}+g\ti{A_2}g^{-1}$ must be regular.
Therefore, by \autoref{thm:SA multiplicity free},  $\phi_1 \otimes \phi_2^g$ is multiplicity free as a representation of $S_{A_1}\cap S_{A_2}^g.$
We note that $\K^{\ell_1}\leq S_{A_1}\cap S_{A_2}^g \leq S_{A_1+gA_2g^{-1}}$ and 
$$\mathrm{Res}^{S_{A_1}\cap S_{A_2}^g}_{\K^{\ell_2}}(\phi_1 \otimes \phi_2^g)= \begin{cases}
\psi_{A_1+gA_2g^{-1}}, & \mathrm{for} \, \mathrm{even} \, \ell ;\\
    q^2\psi_{A_1+gA_2g^{-1}}, & \mathrm{for} \, \mathrm{odd} \, \ell.
\end{cases}$$
 By \autoref{prop:MF from Subgp}(2), we obtain that 
$V(\phi_1,\phi_2^g)$ is a multiplicity free representation.
\end{proof}
Let $A_1,A_2\in \g(\cO_{\ell_1})$ be  regular matrices
   such that  $\tt(A_1)=\tt(A_2)=\cus$. Suppose $g_1,g_2\in\G(\cO_\ell)$ such that both  $V(\phi_1, \phi_2^{g_1})$ and  $V(\phi_1, \phi_2^{g_2})$  
   contain  regular representations. By \autoref{lem:orbit-representatives-gol}, \autoref{lem: 21 entry of coset is zero with cuspidal } and up to a twist by a linear character, we may assume the following choices of matrices: 
$$A_i=\mat{0}{\nonsq\alpha_i}{\nonsq}{0},\, \tilde{A_i}=\mat{0}{\nonsq\tilde{\alpha_i}}{\nonsq}{0},\, g_i=\mat{a_i}{b_i}{0}{c_i},$$
for $i\in \{1,2\}$. We will use these notations for the rest of this section. 

\begin{lemma}
    \label{lem:trace-computation}
 Let $h\in \G(\cO_\ell)$ and $1\leq k < \ell_1$ be such that 
 $h^{-1}\tilde A_1 h = \tilde A_1  \mod(\pi^{k})$  and $g_1^{-1}\tilde A_1 g_1 = g_2^{-1}\tilde A_1 g_2= \frac{a_1}{c_1}\tilde A_2 \mod (\pi^{k}).$ For $w\in \cO_\ell,$ let $Z_w=\mathrm{I}+ \pi ^{\ell_2-k} w \tilde A_1$,  $X_w=Z_wh^{-1}Z_w^{-1}h$ and $Y_w=g_1^{-1}Z_w^{-1}g_1g_2^{-1}h^{-1}Z_whg_2.$ 
 Then the following hold. 
 \begin{enumerate}
     \item $ \bm{tr}\left(\tilde A_1(X_w-\mathrm{I})\right)
    =\frac{\pi ^{\ell_2-k} w}{\det(Z_w)}\,\bm{tr}\left((h \tilde A_1h^{-1}-\tilde A_1)\tilde A_1\right)$. 
    \item $\bm{tr}\left(\tilde A_2(Y_w-\mathrm{I})\right)=\frac{\pi ^{\ell_2-k} w}{\det(Z_w)}\,\bm{tr}\left(\left(hg_2\tilde A_2  g_2^{-1}h^{-1} - g_1\tilde A_2g_1^{-1}\right)\tilde A_1\right)$. 
 \end{enumerate}

  \end{lemma}
\begin{proof}
By direct calculation, we have $Z_w^{-1}=\frac{1}{\det(Z_w)}\mathrm{I} - \frac{\pi ^{\ell_2-k} w }{\det(Z_w)}\tilde A_1 $.
    For (1), note that $\tilde A_1(X_w-\mathrm{I})=\tilde A_1Z_w(h^{-1}Z_w^{-1}h-Z_w^{-1})=\frac{\pi ^{\ell_2-k} w }{\det(Z_w)} \tilde A_1Z_w(\tilde A_1- h^{-1}\tilde A_1h).$ Since $\tilde A_1 ^2=(\nonsq^2\tilde{\alpha_1})\mathrm{I},$ we have  $\tilde A_1Z_w=\tilde A_1 + \pi ^{\ell_2-k} w \nonsq^2 \tilde{\alpha_1}\mathrm{I}$. Therefore, by using the fact that
 $\bm{tr}(\tilde A_1-h^{-1}\tilde A_1h)=0,$ 
we obtain
\begin{eqnarray*}
    \bm{tr}\left(\tilde A_1(X_w-\mathrm{I})\right)
    &=&\frac{\pi ^{\ell_2-k} w}{\det(Z_w)}\,\bm{tr}\left(\tilde A_1(\tilde A_1 -h^{-1}\tilde A_1h)\right)\\
    &=&\frac{\pi ^{\ell_2-k} w}{\det(Z_w)}\,\bm{tr}\left(\tilde A_1^2 -h\tilde A_1h^{-1}\tilde A_1\right)\\
    &=&\frac{\pi ^{\ell_2-k} w}{\det(Z_w)}\,\bm{tr}\left((\tilde A_1- h \tilde A_1h^{-1})\tilde A_1\right).
\end{eqnarray*}

For (2), note that 
 \begin{eqnarray*}
     \tilde A_2(Y_w-\mathrm{I})&=&\tilde A_2g_1^{-1}Z_w^{-1}g_1\left(g_2^{-1}h^{-1}Z_whg_2-g_1^{-1}Z_wg_1\right)\\
     &=&\tilde A_2g_1^{-1}Z_w^{-1}g_1\pi ^{\ell_2-k} w \left(g_2^{-1}h^{-1}\tilde A_1hg_2-g_1^{-1}\tilde A_1g_1\right).
 \end{eqnarray*}
Since $g_2^{-1}h^{-1}\tilde A_1hg_2=g_1^{-1}\tilde A_1g_1 \mod (\pi^{k})$ and  $g_1^{-1}\tilde A_1 g_1= \frac{a_1}{c_1}\tilde A_2 \mod (\pi^{k}),$ we have $\pi ^{\ell_2-k} w (g_2^{-1}h^{-1}\tilde A_1hg_2-g_1^{-1}\tilde A_1g_1)=0\mod (\pi^{\ell_2})$ and $ g_1^{-1}Z_w^{-1}g_1=\frac{1}{\det(Z_w)}\mathrm{I} - \frac{\pi ^{\ell_2-k} w }{\det(Z_w)}g_1^{-1}\tilde A_1g_1=\frac{1}{\det(Z_w)}\mathrm{I} - \frac{\pi ^{\ell_2-k} w a_1}{\det(Z_w)c_1}\tilde A_2
\mod (\pi^{\ell_2}).$ Therefore we can replace $g_1^{-1}Z_w^{-1}g_1$ by $\frac{1}{\det(Z_w)}\mathrm{I} - \frac{\pi ^{\ell_2-k} w a_1}{\det(Z_w)c_1}\tilde A_2$ in the last equation. Hence, we get 
\begin{eqnarray*}
    \tilde A_2(Y_w-\mathrm{I})&=&\tilde A_2\left(\frac{1}{\det(Z_w)}\mathrm{I} - \frac{\pi ^{\ell_2-k} w a_1}{\det(Z_w)c_1}\tilde A_2\right)\pi ^{\ell_2-k} w \left(g_2^{-1}h^{-1}\tilde A_1hg_2-g_1^{-1}\tilde A_1g_1\right)\\
    &=&\frac{\pi ^{\ell_2-k}w}{\det(Z_w)}  \times
     \left(\tilde A_2 - \frac{\pi ^{\ell_2-k} w a_1}{c_1}\tilde A_2^2\right) \left(g_2^{-1}h^{-1}\tilde A_1hg_2-g_1^{-1}\tilde A_1g_1\right)
\end{eqnarray*}
Since $\tilde A_2^2=(\nonsq^2\tilde{\alpha_2})\mathrm{I}$ and $\bm{tr} (g_2^{-1}h^{-1}\tilde A_1hg_2-g_1^{-1}\tilde A_1g_1)=0,$ we obtain
\begin{eqnarray*}
    \bm{tr}(\tilde A_2(Y_w-\mathrm{I}))
    &=&\frac{\pi ^{\ell_2-k} w}{\det(Z_w)}\,\bm{tr}\left(\tilde A_2  \left(g_2^{-1}h^{-1}\tilde A_1hg_2-g_1^{-1}\tilde A_1g_1\right)\right) \\
    &=&\frac{\pi ^{\ell_2-k} w}{\det(Z_w)}\,\bm{tr}\left(hg_2\tilde A_2  g_2^{-1}h^{-1}\tilde A_1 - g_1\tilde A_2g_1^{-1}\tilde A_1\right)\\
    &=&\frac{\pi ^{\ell_2-k} w}{\det(Z_w)}\,\bm{tr}\left(\left(hg_2\tilde A_2  g_2^{-1}h^{-1} - g_1\tilde A_2g_1^{-1}\right)\tilde A_1\right).
\end{eqnarray*}
\end{proof}

\begin{proposition}
    \label{prop:irr with irr H lemma}
    Suppose $\tilde{A}_1+g_i\tilde{A}_2g_i^{-1}$ for $i\in\{1,2\}$ are regular matrices. 
Let $h\in \G(\cO_\ell)$ and $1\leq k < \ell_1$ be as in \autoref{lem:trace-computation}.  Further assume that 
$h(\tilde{A}_1+g_2\tilde{A}_2g_2^{-1})h^{-1}=\tilde{A}_1+g_1\tilde{A}_2g_1^{-1} \mod (\pi^{\ell_1}).$
     If $\langle V(\phi_1, \phi_2^{g_1}),V(\phi_1, \phi_2^{g_2})\rangle\neq 0,$  then 
      we have     
$$\psi\left(\pi ^{\ell_2-k} w \, \bm{tr}\left( \left( h (\tilde A_1+g_2\tilde A_2  g_2^{-1} )h^{-1}-(\tilde A_1 + g_1\tilde A_2g_1^{-1})\right) \tilde A_1\right)
\right)=1 \text{ for all } w \in\cO_\ell.$$
\end{proposition}

\begin{proof}
To prove this, we prove 
\[
\psi\left(\pi ^{\ell_2-k} w \, \bm{tr}\left( \left( \tilde A_1 - h \tilde A_1 h^{-1} \right) \tilde A_1
\right) \right)= \psi\left(\pi ^{\ell_2-k} w \, \bm{tr}\left( \left( h g_2\tilde A_2  g_2^{-1}h^{-1} -g_1\tilde A_2g_1^{-1}   \right) \tilde A_1\right)\right)
\]
for all $w \in \cO_\ell.$ 
Let $H=\Z  \left(\K^{\ell_2-k} \cap \mathrm{C}_{\G(\cO_\ell)}(\tilde{A_1}) \right)\K^{\ell_2}.$
Since the elements of $H$ are of the form $x\mathrm{I}+ \pi ^{\ell_2-k} y \tilde A_1+ \pi ^{\ell_2}B $ for some $x,y\in \rl$ and $B\in M_2(\rl),$ 
by given conditions, we obtain 
$$H\leq \Z  D^{\ell_2}(\tilde{A_1})\cap \Z  D^{\ell_2}(g_1\tilde{A_2}g_1^{-1})\cap \Z D^{\ell_2}(h\tilde{A_1}h^{-1})\cap \Z D^{\ell_2}(hg_2 \tilde{A_2}g_2^{-1}h^{-1}).$$
In particular, 
$H\leq S_{A_1}\cap S_{A_2}^{g_1} $ and $H\leq (S_{A_1}\cap S_{A_2}^{g_2})^{h}.$
Note that $V(\phi_1, \phi_2^{g_1})=\mathrm{Ind}_{S_{A_1}\cap S_{A_2}^{g_1}}^{\G(\cO_{\ell})}(\phi_1 \otimes \phi_2^{g_1})$ and 
$V(\phi_1, \phi_2^{g_2})=\mathrm{Ind}_{S_{A_1}\cap S_{A_2}^{g_2}}^{\G(\cO_{\ell})}(\phi_1 \otimes \phi_2^{g_2})\cong \mathrm{Ind}_{(S_{A_1}\cap S_{A_2}^{g_2})^{h}}^{\G(\cO_{\ell})}(\phi_1 \otimes \phi_2^{g_2})^{h}.$
Therefore 
 $V(\phi_1, \phi_2^{g_1})$ and  $V(\phi_1, \phi_2^{g_2})$ are  subrepresentations of  $\mathrm{Ind}_{H}^{\G(\cO_{\ell})}(\mathrm{Res}_{H}^{S_{A_1}\cap S_{A_2}^{g_1}}(\phi_1\otimes \phi_2^{g_1}))$ and  $\mathrm{Ind}_{H}^{\G(\cO_{\ell})}(\mathrm{Res}_{H}^{(S_{A_1}\cap S_{A_2}^{g_2})^h}(\phi_1\otimes \phi_2^{g_2})^h)$ respectively. Hence, our assumption
 $\langle V(\phi_1, \phi_2^{g_1}),V(\phi_1, \phi_2^{g_2})\rangle\neq 0$ implies
 \begin{equation}\label{eqn:nan}
 \langle \mathrm{Ind}_{H}^{\G(\cO_{\ell})}(\mathrm{Res}_{H}^{S_{A_1}\cap S_{A_2}^{g_1}}(\phi_1\otimes \phi_2^{g_1})),\mathrm{Ind}_{H}^{\G(\cO_{\ell})}(\mathrm{Res}_{H}^{(S_{A_1}\cap S_{A_2}^{g_2})^h}(\phi_1\otimes \phi_2^{g_2})^h)\rangle\neq 0.
 \end{equation}
 Let $\eta_i\in\mathrm{Irr}(\Z D^{\ell_2}(\tilde{A}_i)\mid \psi_{A_i})$ be such that $\mathrm{Res}_{\Z D^{\ell_2}(\tilde{A}_i)}^{S_{A_i}}(\phi_i)=q \eta_i$, see \autoref{subsec:alternate const for cus}.
We have  $\phi_1\otimes \phi_2^{g_1}=q^2\, ( \eta_1\otimes \eta_2^{g_1}) $ on $H$ and $(\phi_1\otimes \phi_2^{g_2})^h=q^2 \,(\eta_1^h\otimes\eta_2^{hg_2})$  on $H.$  Therefore \autoref{eqn:nan} implies 
 \begin{equation*}\label{eqn:nan1}
\langle \mathrm{Ind}_{H}^{\G(\cO_{\ell})}( \mathrm{Res}_{H}(\eta_1\otimes \eta_2^{g_1})),\mathrm{Ind}_{H}^{\G(\cO_{\ell})}(\mathrm{Res}_{H} (\eta_1^h\otimes \eta_2^{hg_2}))\rangle\neq 0.
\end{equation*}
Since $\eta_i $ for $i\in\{1,2\}$ are one-dimensional representations, we  have $\mathrm{Res}_{H}(\eta_1\otimes \eta_2^{g_1})\in \mathrm{Irr}(H\mid \psi_{B_1})$ and $\mathrm{Res}_{H} (\eta_1^h\otimes \eta_2^{hg_2})\in \mathrm{Irr}(H\mid \psi_{B_2}),$ where $B_1,B_2\in \g(\cO_{\ell_1})$ such that $B_1=\tilde{A}_1+g_1\tilde{A}_2g_1^{-1} \mod (\pi^\ldown)$ and $B_2=h(\tilde{A}_1+g_2\tilde{A}_2g_2^{-1})h^{-1} \mod (\pi^\ldown).$ 
Since $h(\tilde{A}_1+g_2\tilde{A}_2g_2^{-1})h^{-1}=\tilde{A}_1+g_1\tilde{A}_2g_1^{-1} \mod (\pi^{\ell_1})$ and  $\tilde{A}_1+g_i\tilde{A}_2g_i^{-1}$ for $i\in\{1,2\}$ are regular matrices, we have $B_1=B_2$ and $B_1$ is regular. Therefore,
 by \autoref{prop:inner prod from Subgp}, \autoref{eqn:nan} implies 
 $\mathrm{Res}_{H}(\eta_1\otimes \eta_2^{g_1})=\mathrm{Res}_{H} (\eta_1^h\otimes \eta_2^{hg_2}).$
Therefore 
\begin{equation}\label{eqn:sd0}
\eta_1(Zh^{-1}Z^{-1}h)=\eta_2(g_1^{-1}Z^{-1}g_1g_2^{-1}h^{-1}Zhg_2), \text{ for all } Z\in H.
\end{equation}
For $w\in \cO_\ell,$ let $Z_w=\mathrm{I}+ \pi ^{\ell_2-k} w \tilde A_1.$ We claim that 
\begin{equation}\label{eqn:sd}
\eta_1(Z_wh^{-1}Z_w^{-1}h)=\eta_2(g_1^{-1}Z_w^{-1}g_1g_2^{-1}h^{-1}Z_whg_2) \text{ for all }w\in \cO_\ell.
\end{equation}
For $\G=\GL_2,$ since $Z_w\in H,$ the claim  directly follows from
\autoref{eqn:sd0}.
For $\G=\GU_2,$ choose $\lambda_w \in \rl, $ such that $\lambda_w\lambda_w^\circ =(1-\pi^{2\ell_2-2k}\nonsq^2w^2 \tilde{\alpha_1})^{-1}.$ Then it is easy to sea that $\lambda_w Z_w \in \GU_2(\cO_\ell) $ and hence $\lambda_w Z_w \in H.$ Therefore the claim follows by substituting $Z=\lambda_w Z_w $ in \autoref{eqn:sd0}.

Let $X_w=Z_wh^{-1}Z_w^{-1}h$ and $Y_w=g_1^{-1}Z_w^{-1}g_1g_2^{-1}h^{-1}Z_whg_2.$ 
Since $h^{-1}\tilde A_1 h = \tilde A_1  \mod (\pi^{k})$ and  $g_1^{-1}\tilde A_1 g_1 = g_2^{-1}\tilde A_1 g_2 \mod (\pi^{k}),$ we must have $X_w, Y_w\in \K^{\ell_2}.$
 Therefore \autoref{eqn:sd} implies $\psi_{A_1}(X_w)=\psi_{A_2}(Y_w),$ which is equivalent to 
\begin{equation*}
   \psi\left(\bm{tr}\left(\tilde A_1(X_w-\mathrm{I})\right)\right)=\psi\left(\bm{tr}\left(\tilde A_2(Y_w-\mathrm{I})\right)\right).
\end{equation*}
Hence by \autoref{lem:trace-computation},
 we obtain
\begin{equation*}
   \psi\left(\frac{\pi ^{\ell_2-k} w}{\det(Z_w)}\,\bm{tr}\left((\tilde A_1- h \tilde A_1h^{-1})\tilde A_1\right)\right)=\psi\left(\frac{\pi ^{\ell_2-k} w}{\det(Z_w)}\,\bm{tr}\left(\left(hg_2\tilde A_2  g_2^{-1}h^{-1} - g_1\tilde A_2g_1^{-1}\right)\tilde A_1\right)\right).
\end{equation*}
Note that by Hensel's lemma, we have  $\{\frac{ w }{\det(Z_w)}=\frac{ w}{1-\pi^{2\ell_2-2k}w^2 \tilde{\alpha}_1} \mid w \in\cO_\ell\}=\cO_\ell$. Therefore we must have 
$$ \psi\left(\pi ^{\ell_2-k} w\,\bm{tr}\left((\tilde A_1- h \tilde A_1h^{-1})\tilde A_1\right)\right)=\psi\left(\pi ^{\ell_2-k} w\,\bm{tr}\left(\left(hg_2\tilde A_2  g_2^{-1}h^{-1} - g_1\tilde A_2g_1^{-1}\right)\tilde A_1\right)\right)$$
for all $w \in\cO_\ell.$
\end{proof}

\begin{proposition}
    \label{prop:irr with irr}
   Suppose that both  $V(\phi_1, \phi_2^{g_1})$ and  $V(\phi_1, \phi_2^{g_2})$   contain  regular representations.
   If $\langle V(\phi_1, \phi_2^{g_1}),V(\phi_1, \phi_2^{g_2})\rangle\neq 0,$  then we must have $S_{A_1}g_1S_{A_2}=S_{A_1}g_2S_{A_2}.$
\end{proposition}
\begin{proof}
We will use \autoref{thm:D theorem} to prove our result. Both  $V(\phi_1, \phi_2^{g_1})$ and  $V(\phi_1, \phi_2^{g_2})$ contain  regular representations, therefore both $\tilde A_1+g_1\tilde A_2g_1^{-1}$ and $\tilde A_1+g_2\tilde A_2g_2^{-1}$ 
are regular. Since $\langle V(\phi_1, \phi_2^{g_1}),V(\phi_1, \phi_2^{g_2})\rangle\neq 0,$ both $\tilde A_1+g_1\tilde A_2g_1^{-1}$ and $\tilde A_1+g_2\tilde A_2g_2^{-1}$ 
are conjugate  modulo $ (\pi^{\ell_1} )$. Therefore $\det(\tilde{A_1}+g_1\tilde{A}_2g_1^{-1})-\det(\tilde{A_1}+g_2\tilde{A}_2g_2^{-1})\in \pi^{\ell_1}\rl.$ 
For $D = \dtag$, 
 by \autoref{lem-detD-computation}, we have $\det(\tilde{A_1}+g_1\tilde{A}_2g_1^{-1})-\det(\tilde{A_1}+g_2\tilde{A}_2g_2^{-1}) = \frac{\nonsq}{a_1 a_2c_1 c_2}\times \det(D)$ and hence $\det(D)=0\mod (\pi^{\ell_1} )$.  

If $D= 0  \mod (\pi^{\ell_1}),$ then 
 $D \cmat{1}{0}=\cmat{0}{0} \mod (\pi^{\ell_1} )$ and hence,  by \autoref{thm:D theorem},  $S_{A_1} g_1 S_{A_2}=S_{A_1} g_2 S_{A_2}.$
Assume $D\neq 0  \mod (\pi^{\ell_1}).$
Let $0\leq k < \ell_1$  be such that $D=\pi^k D' $ for some $D'\in M_2(\rl)$ with $D'\neq 0  \mod (\pi ).$ We first claim that 
$\det(D)=0  \mod (\pi^{\ell_1+k} ).$
 If $k=0,$ since $\det(D)=0\mod (\pi^{\ell_1} ),$ the claim follows trivially.
 Assume $1\leq k <{\ell_1}.$
Let 
$$h'= \mat{1+a_2^{-1}c_2}{a_1^{-1}b_1-a_2^{-1}b_2}{0}{1+a_1^{-1}c_1}.$$
 Since $D= 0  \mod (\pi^k )$ and  $\tilde{A}_1 +g_i \tilde{A}_2g_i^{-1}$ for  $i\in \{1,2\}$ are regular, by \autoref{lem:D=0 implies}, we have $1+a_i^{-1}c_i \in \rl^\times.$ Hence $h'$ is an invertible matrix.
 For $\G=\GL_2,$ let $h=h'.$
 For $\G=\GU_2$, by using $g_i \in \GU_2$ and the relations for $c_i$ and $b_i$, we get
 $h'=\smat{1+c_2^\circ c_2}{\nonsq  (t_1-t_2)}{0}{1+c_1^\circ c_1}.$
  Choose $d\in \rl$ such that $dd^\circ=(1+c_1^{\circ}c_1)(1+c_2^{\circ}c_2).$  Let $h=d^{-1}h'$. Note that $h\in \GU_2(\cO_\ell).$
By direct calculation, we obtain the following. 
\begin{equation}\label{eqn:det12}
    h(\tilde{A}_1 + g_2 \tilde{A}_2g_2^{-1})h^{-1}- (\tilde{A}_1 +g_1 \tilde{A}_2g_1^{-1})= \frac{\det(D)}{(1+a_1^{-1}c_1)a_1a_2c_1c_2} \mat{0}{1}{0}{0}.
\end{equation}
Therefore, $h(\tilde{A}_1 + g_2 \tilde{A}_2g_2^{-1})h^{-1}=\tilde{A}_1 +g_1 \tilde{A}_2g_1^{-1}\mod (\pi^{\ell_1} ).$ 
Since $D= 0  \mod (\pi^k ),$ by \autoref{lem:D=0 implies}, we obtain that $h'=(1+a_2^{-1}c_2)\mathrm{I}\mod (\pi^{k}),$ $g_2= \frac{c_2}{c_1}g_1 \mod (\pi^{k}),$ $g_1=\smat{a_1}{0}{0}{c_1}\mod (\pi^{k} )$ and $ \tilde{\alpha}_2 =  a_1^{-2} c_1^{2} \tilde{\alpha}_1 \mod (\pi^k ).$ 
Therefore  $ h^{-1}\tilde A_1 h=h^{\prime-1}\tilde A_1 h'=\tilde A_1\mod (\pi^k )$ and 
$$ g_2^{-1}\tilde A_1 g_2=g_1^{-1}\tilde A_1 g_1=\frac{a_1}{c_1}\mat{0}{\nonsq a_1^{-2} c_1^{2} \tilde{\alpha}_1}{\nonsq}{0}=\frac{a_1}{c_1}\tilde A_2\mod (\pi^k ).$$
Hence by \autoref{prop:irr with irr H lemma}, 
$$\psi\left(\pi ^{\ell_2-k} w\, \bm{tr}\left( \left( h (\tilde A_1+g_2\tilde A_2  g_2^{-1} )h^{-1}-(\tilde A_1 + g_1\tilde A_2g_1^{-1})\right) \tilde A_1\right)
\right)=1 \text{ for all } w \in\cO_\ell.$$ 
Therefore 
by substituting \autoref{eqn:det12}, we obtain 
\begin{equation}\label{eqn:lkjh}
    \psi\left( \frac{\pi ^{\ell_2-k} w  \,\nonsq \, \det(D)
}{(1+a_1^{-1}c_1)a_1a_2c_1c_2} \right)=1  \,\,\,\text{for all } w\in \cO_\ell.
\end{equation}
 Recall that $\pi^{\ell-1}\cO_\ell \not\subseteq \ker(\psi) .$ Therefore, for $\G=\GL_2,$ 
 since $(1+a_1^{-1}c_1)a_1a_2c_1c_2 \in \cO_\ell^\times,$ \autoref{eqn:lkjh} implies  $\pi^{\ell_2-k}\det(D)=0,$ which is equivalent to   $\det(D)=0\mod (\pi^{\ell_1+k} ).$ 
 For $\G=\GU_2,$ since $g_i \in \GU_2(\cO_\ell),$ we have $c_i=(a_i^\circ)^{-1}$. By \autoref{eqn:as1}, we obtain that $\det(D)= (a_1^\circ a_2^\circ)^{-2} \nonsq \lambda$ for some $\lambda\in \cO_\ell.$ Therefore $\frac{\nonsq \, \det(D)
}{(1+a_1^{-1}c_1)a_1a_2c_1c_2} =\frac{\nonsq^2 \, \lambda
}{(1+(a_1a_1^\circ)^{-1})a_1a_2a_1^\circ a_2^\circ}\in \cO_\ell.$ Hence \autoref{eqn:lkjh} implies 
$\pi^{\ell_2-k}\det(D)=0,$ which is equivalent to   $\det(D)=0\mod (\pi^{\ell_1+k} ).$ This proves the claim.

 We now proceed to show that there exist  $x,y\in \rl$  such that $\{x,y\}\cap \rl^\times \neq \emptyset$ and $D\cmat{x}{y}=0  \mod (\pi^{\ell_1} )$. 
Since $D'\neq 0  \mod (\pi ),$ there exists 
 $m\in \{1,2\}$ such that  $\{D'_{m1},D'_{m2}\}\cap \rl^\times\neq \emptyset.$ Choose $x=D_{m2}'$ and $y=-D_{m1}'.$ For this choice, we have  $\{x,y\}\cap \rl^\times\neq \emptyset$ and 
$$D\left[ \begin{matrix} x  \\  y \end{matrix}\right] =\pi^k D'\left[ \begin{matrix} D_{m2}' \\  -D_{m1}' \end{matrix}\right]=
\begin{cases}
    \pi^k\cmat{0}{-\det(D')} ,& \text{if } m=1;\\
    \pi^k\cmat{\det(D')}{0}, & \text{if } m=2.
\end{cases}$$
 Since $\pi^{2k}\det(D')=\det(D)=0  \mod (\pi^{\ell_1+k} ),$ we must have $\pi^{k}\det(D')=0  \mod (\pi^{\ell_1} ).$ Therefore $D\cmat{x}{y}=0  \mod (\pi^{\ell_1} ).$ Hence the result follows from \autoref{thm:D theorem}. 
\end{proof}
The proof of \autoref{thm:main-theorem-2}(3) follows from \autoref{eqn:decomposition-of-regular-tensor}, \autoref{prop:multiplicity-free-component-cuspidal} and \autoref{prop:irr with irr}. 

\section{Proof of \autoref{thm:main-theorem-2}(4)} 
\label{sec:proof-for-Sigma-4}
In this section, we prove that for any three split semisimple regular representations $\rho_1, \rho_2, \rho_3$ of $\Gol$, we have $\langle \rho_1 \otimes \rho_2, \rho_3 \rangle \leq \ell + 1$.  
Recall from \autoref{subsec:alternate const for ss}, a pair $(\chi_1, \chi_2)$ of characters of $\rl^\times$ is called a $\ss$-pair of $\G(\cOl)$ if and only if $ \chi_1\chi_2^{-1}|_{1+ \pi^{\ell-1}\cO_\ell} \neq 1 $ and $\mathfrak{S}$ denotes the set of all $\ss$-pairs.
Further, a representation $\rho$ of $\G(\cO_\ell)$ is a split semisimple regular representation if and only if $\rho\cong\ind_{\B(\co_\ell)}^{\G(\cOl)}(\chi_1,\chi_2)$ for some $\ss$-pair $(\chi_1, \chi_2)$ of $\G(\cOl)$.   

Now onward, we fix $\ss$-pairs $(\chi_1, \chi_2)$ and $(\chi_3, \chi_4)$ and representations $\rho_1 \cong \ind_{\B(\co_\ell)}^{\G(\co_\ell)} (\chi_1,\chi_2) $ and $\rho_2 \cong \ind_{\B(\co_\ell)}^{\G(\co_\ell)} (\chi_3,\chi_4) $. 
    We have
    \[
     \rho_1 \otimes \rho_2 \cong \underset{g\in {\B(\co_\ell)}\backslash {\G(\co_\ell)} / {\B(\co_\ell)}}{\oplus} \ind_{{\B(\co_\ell)} \cap {\B(\co_\ell)}^{g}}^{\G(\co_\ell)} (\chi_1,\chi_2)\otimes{(\chi_3,\chi_4)}^g.
    \]
 It is well known that the double cosets representatives of $\B(\co_\ell)$ in $\G(\co_\ell)$  are given by the set
 $$\left\{\left[\begin{smallmatrix}
                  0&1\\
                  1&0
        \end{smallmatrix}\right],
\left[\begin{smallmatrix}
 1&0\\
 \nonsq\pi^i&1
  \end{smallmatrix}\right]
  ;1\leq i\leq \ell \right\}.$$ For $i \in [ 1,\ell-1]$, we denote $\smat{1}{0}{\nonsq \pi^i}{1}
 $ by $g_i$ and $\B(\cOl) \cap \B(\cOl)^{g_i}$ by $\B^i$. By direct computation, we have
   \begin{equation*}
        \B^i=\left\{\left[\begin{smallmatrix}
         a-\nonsq \pi^i b & b\\
         0 & c+ \nonsq \pi^i b 
\end{smallmatrix}\right]\mid \left[\begin{smallmatrix}
      a & b\\
      0& c
\end{smallmatrix}\right] \in \G(\co_\ell), \,\ a=c+\nonsq \pi^i b \mod(\pi^{\ell-i})  \right\}.
   \end{equation*}

 We denote $\ind_{\B^i}^{\B(\co_\ell)} (\chi_1,\chi_2)\otimes {(\chi_3,\chi_4)}^{g_i}$ by $\delta_i$ and  the group of diagonal matrices in $\G(\co_\ell)$  by ${\T(\co_\ell)}$.   
     Then we have,
        \begin{equation}
        \label{eqn: ind B tensor ind B}
            \rho_1 \otimes \rho_2 \cong \ind_{\B(\co_\ell)}^{\G(\co_\ell)}(\chi_1\chi_3,\chi_2\chi_4) \oplus \ind_{\T(\co_\ell)}^{\G(\co_\ell)}(\chi_1\chi_4,\chi_2\chi_3)\oplus \left(
            \underset{1\leq i \leq \ell-1}{\oplus} \ind_{\B(\co_\ell)}^{\G(\co_\ell)}\delta_i\right).
        \end{equation}
     To understand the multiplicity of a split semisimple irreducible representation in $\rho_1 \otimes \rho_2,$ we understand its multiplicities in the above constituents of $\rho_1 \otimes \rho_2$. We shall carry this out in the next few lemmas before proceeding to the proof of our main result.  

\begin{lemma}
\label{lem: form change of b cap bg rep}
 The representations $\delta_i$  are irreducible for every $i \in [1, \ell-1]$.
 \end{lemma}
\begin{proof}
To prove this, we need to show that 
     ${\langle \delta_i, \delta_i \rangle}_{\B(\co_\ell)} = 1.$  If not, then there exists a non-trivial double coset representative $h$ of $\B^i \backslash \B(\cO_\ell) / \B^i$ such that    
\begin{equation*}
{(\chi_1,\chi_2)\otimes(\chi_3,\chi_4)^{g_i}=((\chi_1,\chi_2)\otimes(\chi_3,\chi_4)^{g_i})^h } \text{ on }{\B^i\cap (\B^i)^h}. 
\end{equation*}
Since $(\chi_1,\chi_2)^h=(\chi_1,\chi_2),$ we obtain that 
\begin{equation}\label{eqn:ss-ss ind B to G irreducible eqn}
{(\chi_3,\chi_4)^{g_i}=(\chi_3,\chi_4)^{hg_i}} \text{ on }{\B^i\cap (\B^i)^h}. 
\end{equation}
Note that for $g\in  \B(\co_\ell)  $,  there exists $\smat{x}{0}{0}{y} \in \B(\co_\ell)$ such that $\B^i g \B^i=\B^i \smat{x}{0}{0}{y} \B^i$. Hence, we assume that $h=\smat{x}{0}{0}{y} \in \B(\co_\ell)$.
Since $h$ is a non-trivial double coset representative, we have 
$x\neq y \mod (\pi^{\ell-i})$. Let $1- xy^{-1}= \pi^k u$ for some $k\in [0,\ell-i-1]$ and $u\in \cO_\ell^\times.$
For $b\in \cO_\ell,$ let 
$$X_b=
    \mat{\lambda}{\pi^{\ell-i-k-1} xy^{-1}\nonsq \lambda b }{0}{\lambda+ \pi^{\ell-k-1}\nonsq^2 \lambda b},
$$
where $\lambda=1$ for $\G=\GL_2,$ and $\lambda\in \Lri_\ell^\times$ be such that $\lambda^\circ \lambda=(1+\pi^{\ell-k-1}\nonsq^2 b)^{-1}$ for $\G=\GU_2$. Using $ \pi^{\ell-k-1}\nonsq^2 \lambda b-\pi^{\ell-k-1} xy^{-1}\nonsq^2 \lambda b =\pi^{\ell-1}\nonsq^2 \lambda b u $, one can easily show that $X_b\in \B^i \cap (\B^i)^h$ for all $b\in \cO_\ell$.
Therefore,  \autoref{eqn:ss-ss ind B to G irreducible eqn} implies that $(\chi_3,\chi_4)(g_i^{-1}X_bg_i)=(\chi_3,\chi_4)(g_i^{-1}h^{-1}X_b hg_i)$ for all $b\in \cO_\ell$. Upon simplification, we get 
$$\chi_3(\lambda+\pi^{\ell-k-1} xy^{-1} \nonsq^2 \lambda b)\chi_4(\lambda + \pi^{\ell-k-1}  \nonsq^2 \lambda b(1-xy^{-1})) = \chi_3(\lambda+\pi^{\ell-k-1} \nonsq^2 \lambda b)\chi_4(\lambda ).$$
Substituting $xy^{-1}=1-\pi^k u$ and then dividing both sides by $\chi_3(\lambda+\pi^{\ell-k-1}  \nonsq^2 \lambda b)\chi_4(\lambda),$ we obtain 
$$\chi_3(1+\pi^{\ell-1}   \nonsq^2  b u)\chi_4(1 + \pi^{\ell-1}  \nonsq^2  b u ) = 1.$$
Since $(1+\pi^{\ell-1}   \nonsq^2  b u)^2=1,$ this gives $\chi_3\chi_4^{-1}(1+\pi^{\ell-1}   \nonsq^2  b u)=1$ for all $b\in \cO_\ell$. This contradicts the fact that $(\chi_3,\chi_4)$ is  a  $\ss$-pair. Therefore ${\langle \delta_i, \delta_i \rangle}_{\B(\co_\ell)} = 1.$
 \end{proof}
For any subgroup $H$ of $\B(\cOl)$, we denote the restriction of $(\chi_1, \chi_2)$ to $H$ by $(\chi_1, \chi_2)$ itself. 
Let $\U(\cOl)$ be the subgroup of $\G(\cOl)$ consisting of upper triangular matrices with diagonal entries equal to $1$. For $t  \in [0, \ell],$ let $\psi_t$ denote a character of $\U(\cOl)$ defined by:
\[
\psi_t\left(\mat{1}{\nonsq x}{0}{1}\right):=\psi(\pi^{\ell-t} \nonsq x).
\]
For $t  \in [0, \ell],$ let $\Z_{t}(\co_\ell)$ be the subgroup $\{\smat{a}{0}{0}{a+\pi^{t}d} \mid a, d \in \rl\} \cap \G(\co_\ell)$ of $\G(\co_\ell)$. Note that $\Z_{0}(\co_\ell)=\T(\co_\ell)$. For $\chi,\chi' \in \widehat{\rl^\times}$, define a character $(\chi,\chi',\psi_{t}) $ of the group $\Z_{t}(\co_\ell)\U(\co_\ell)$ as follows:
\[
(\chi,\chi',\psi_{t})\left(\smat{a}{x}{0}{a+\pi^{t}d}\right)=\chi(a)\chi'(a+\pi^t d)\psi_t\left(\smat{1}{a^{-1}x}{0}{1}\right).
\]
The representation $\delta_i$ is an irreducible representation of $\B(\co_\ell)$ of dimension $q^{\ell-i}-q^{\ell-i-1}$. By a description of all irreducible representations of $\B(\cO_\ell)$ using little group method, $\delta_i$ is isomorphic to $\ind_{\Z_{\ell-i}(\co_\ell)\U(\co_\ell)}^{\B(\co_\ell)}(\chi,\chi',\psi_{\ell-i})$ for some $\chi, \chi'\in \widehat{\rl^\times}$. The next lemma gives a necessary condition for this isomorphism.

\begin{lemma}
\label{lem: borel rep isomorphism}
For $i \in [1, \ell-1]$, let $\delta_i$ be as above. Then
    $ \delta_i \cong \ind_{\Z_{\ell-i}(\co_\ell)\U(\co_\ell)}^{\B(\co_\ell)}(\chi,\chi',\psi_{\ell-i})$ 
    for some $\chi, \chi'\in \widehat{\rl^\times}$ gives
    $( \chi_1\chi_3,\chi_2\chi_4)|_{\Z_{\ell-i}(\co_\ell)} = (\chi,\chi')|_{\Z_{\ell-i}(\co_\ell)}$.
\end{lemma}
\begin{proof}
By definition of $\delta_i$ and the hypothesis, we have
\[
\langle  \ind_{\B^i}^{\B(\co_\ell)}((\chi_1,\chi_2)\otimes(\chi_3,\chi_4)^{g_i}), \ind_{\Z_{\ell-i}(\co_\ell)\U(\co_\ell)}^{\B(\co_\ell)}(\chi,\chi',\psi_{\ell-i})\rangle=1.
\]
This implies, $(\chi_1,\chi_2)\otimes(\chi_3,\chi_4)^{g_i}=(\chi,\chi',\psi_{\ell-i})^h$ on $\B^i \cap (\Z_{\ell-i}(\co_\ell)\U(\co_\ell))^h$ for some $h\in \B^i \backslash \B(\co_\ell) / \Z_{\ell-i}(\co_\ell)\U(\co_\ell)$. It is easy to see that we can take  $h=\smat{z}{0}{0}{w}$ for some $z,w \in \rl^\times$. This gives $\Z_{\ell-i}(\cO_\ell) \subseteq
\B^i \cap (\Z_{\ell-i}(\co_\ell)\U(\co_\ell))^h$. Therefore 
\begin{equation}\label{eqn:lk}
\left((\chi_1,\chi_2)\otimes(\chi_3,\chi_4)^{g_i}\right)\mid_{\Z_{\ell-i}(\cO_\ell)}=(\chi,\chi',\psi_{\ell-i})^h\mid_{\Z_{\ell-i}(\cO_\ell)}.
\end{equation}
Since $g_i^{-1}Xg_i=X$ and  $h^{-1}Xh=X$ for all $X\in \Z_{\ell-i}(\cO_\ell)$, the result directly follows from \autoref{eqn:lk}.
\end{proof}
 
\begin{lemma}
\label{lem:info about B_G_U}
Let $k\in [1,\ell]$ and 
 $ \Omega  =  \{\left[\begin{smallmatrix}
       0 & 1\\
       1 & 0
\end{smallmatrix}\right], \left[\begin{smallmatrix}
       1 & 0\\
      \nonsq \pi^jz & 1
\end{smallmatrix}\right]; j \in [1,\ell] \text{ and }  z \in \co_\ell^\times \}$.
\begin{enumerate}
    \item For $g\in \G(\co_\ell)$, there exists $g'\in \Omega$ such that 
    $\B(\co_\ell) g \Z_k(\co_\ell) \U(\co_\ell)=\B(\co_\ell) g' \Z_k(\co_\ell) \U(\co_\ell)$.
   \item For $j\in [1,\ell]$ and $ z,z'\in \co_\ell^\times$ such that  $z=z' \mod (\pi^{j})$, we have 
    $$\B(\co_\ell) \left[ \begin{smallmatrix}
            1&0\\
          \nonsq  \pi^jz&1
        \end{smallmatrix} \right] \Z_k(\co_\ell) \U(\co_\ell)=\B(\co_\ell) \left[ \begin{smallmatrix}
            1&0\\
          \nonsq  \pi^jz'&1
        \end{smallmatrix} \right] \Z_k(\co_\ell) \U(\co_\ell).$$

\end{enumerate} 
   
\end{lemma}
\begin{proof} 
Note that (1) follows from direct computations.  
For (2), let $u\in \cO_\ell$ be such that $z'=z+\pi^j u$.  Then we have,
        $$\mat{\frac{z}{z'}}{ \frac{u}{\nonsq z z'}}{0}{\frac{z'}{z}}
            \mat{1}{0}{\nonsq \pi^jz}{1}
            =\mat{1}{ \frac{u}{\nonsq zz'}}{\nonsq\pi^j z'}{\frac{z'}{z}}
            = \mat{1}{0}{\nonsq\pi^j z'}{1}
        \mat{1}{ \frac{u}{\nonsq z z'}}{0}{1}.$$
        This proves (2).
\end{proof}

 For $k \in [0, \ell-1] $, define the sets 
\begin{eqnarray*}
       S_1^k&:=&\{  (\omega_1,\omega_2) \in \mathfrak{S} \mid     (\omega_1,\omega_2)|_{\Z_{k}(\co_\ell)}  =(\chi_1\chi_3,\chi_2\chi_4)|_{\Z_{k}(\co_\ell)} \},\\
       S_2^k&:=&\{  (\omega_1,\omega_2) \in \mathfrak{S} \mid     (\omega_1,\omega_2)|_{\Z_{k}(\co_\ell)}  =(\chi_2\chi_4,\chi_1\chi_3)|_{\Z_{k}(\co_\ell)} \},\\
        S_3^k&:=&\{  (\omega_1,\omega_2) \in \mathfrak{S} \mid      (\omega_1,\omega_2)|_{\Z_{k}(\co_\ell)}  =(\chi_1\chi_4,\chi_2\chi_3)|_{\Z_{k}(\co_\ell)} \},\\
       S_4^k&:=&\{  (\omega_1,\omega_2) \in \mathfrak{S} \mid     (\omega_1,\omega_2)|_{\Z_{k}(\co_\ell)}  =(\chi_2\chi_3,\chi_1\chi_4)|_{\Z_{k}(\co_\ell)} \},\\
    S_0&:=&\{  (\omega_1,\omega_2) \in \mathfrak{S} \mid    (\omega_1,\omega_2)|_{\Z(\co_\ell)}  =(\chi_1\chi_3,\chi_2\chi_4)|_{\Z(\co_\ell)} \}.
\end{eqnarray*}
Note that for $j\in [1,4]$, we have $S_j^0 \subseteq S_j^1  \subseteq \cdots  \subseteq S_j^{\ell-1} \subseteq S_0$. Also, it is easy to show that if  $j,j'\in [1,4]$ with $j\neq j'$, then $S_j^k \cap  S_{j'}^{k'}=\emptyset $ for all $k,k' \in [0, \ell-1] $.

\begin{proposition}
    \label{prop:condition for ss in phi i}
    For any $i \in [1, \ell-1]$ and $\ss$-pair $(\omega_1,\omega_2),$ we have
    \[
\langle\ind_{\B(\co_\ell)}^{\G(\co_\ell)}\delta_i,\ind_{\B(\co_\ell)}^{\G(\co_\ell)}(\omega_1,\omega_2)\rangle \leq 1\] 
 and equality holds  if and only if  either $(\omega_1,\omega_2)\in S_1^{\ell-i}$ or $(\omega_1,\omega_2)\in S_2^{\ell-i}$.
\end{proposition}
\begin{proof}
Fix $\chi,\chi' \in \rl^\times$ such that  $\ind_{\B(\co_\ell)}^{\G(\co_\ell)}\delta_i \cong \ind_{\Z_{\ell-i}(\co_\ell)\U(\co_\ell)}^{\G(\co_\ell)}(\chi,\chi',\psi_{\ell-i})$. Therefore, 
 \[
\langle\ind_{\B(\co_\ell)}^{\G(\co_\ell)}\delta_i,\ind_{\B(\co_\ell)}^{\G(\co_\ell)}(\omega_1,\omega_2)\rangle= 
\underset{g\in {\B(\co_\ell)}\backslash {\G(\co_\ell)} / {\Z_{\ell-i}(\co_\ell)\U(\co_\ell)}}{\sum} 
\langle (\chi,\chi',\psi_{\ell-i})^g, (\omega_1,\omega_2)  \rangle_{\B(\co_\ell) \cap \left(\Z_{\ell-i}(\co_\ell)\U(\co_\ell)\right)^g}.
\] 
Let $\eta\in \cO_\ell^{\times} $ be such that $\omega_1\omega_2^{-1}(1+\pi^{\lup}b)=\psi({\pi^{\lup}\eta b})$ for all $b\in \cO_\ell$.
Next, we prove the following statements (1)-(3). The result then follows by \autoref{lem:info about B_G_U} and the fact that $S_1^{\ell-i} \cap S_2^{\ell-i} =\emptyset$.  
\begin{enumerate}
    \item For $g=\smat{0}{1}{1}{0}$,  $ (\chi,\chi',\psi_{\ell-i})^g= (\omega_1,\omega_2)$ on $\B(\co_\ell) \cap \left(\Z_{\ell-i}(\co_\ell)\U(\co_\ell)\right)^g$ if and only if $(\omega_1,\omega_2)\in S_2^{\ell-i}$.
    \item For $g=\smat{1}{0}{-\nonsq \pi^i (\eta\nonsq^{2})^{-1}}{1}$,  $ (\chi,\chi',\psi_{\ell-i})^g= (\omega_1,\omega_2)$ on $\B(\co_\ell) \cap \left(\Z_{\ell-i}(\co_\ell)\U(\co_\ell)\right)^g$ if and only if  $(\omega_1,\omega_2)\in S_1^{\ell-i}$.
    \item Let $j\in [1,\ell]$ and  $z\in \cO_\ell^\times$ be such that  $\pi^j z\neq - \pi^i (\eta\nonsq^{2})^{-1} \mod (\pi^{2i})$, and let  $g=\smat{1}{0}{\nonsq \pi^j z}{1}$. For any $\ss$-pair $(\omega_1,\omega_2)$, we have  $ (\chi,\chi',\psi_{\ell-i})^g\neq  (\omega_1,\omega_2)$ on $\B(\co_\ell) \cap \left(\Z_{\ell-i}(\co_\ell)\U(\co_\ell)\right)^g$.
\end{enumerate}

To prove (1), let $g=\smat{0}{1}{1}{0}$. By direct computation, we have $\B(\co_\ell) \cap \left(\Z_{\ell-i}(\co_\ell)\U(\co_\ell)\right)^g=\Z_{\ell-i}(\co_\ell)$.
Also, $(\chi,\chi',\psi_{\ell-i})^g\mid_{\Z_{\ell-i}(\co_\ell)}=(\chi',\chi)\mid_{\Z_{\ell-i}(\co_\ell)}.$ Therefore, by \autoref{lem: borel rep isomorphism}, we obtain 
$$(\chi,\chi',\psi_{\ell-i})^g\mid_{\Z_{\ell-i}(\co_\ell)}=(\chi',\chi)\mid_{\Z_{\ell-i}(\co_\ell)}=(\chi_2\chi_4 , \chi_1\chi_3)|_{\Z_{\ell-i}(\co_\ell)}.$$
This directly gives (1). 

To prove (2), let $g=\smat{1}{0}{-\nonsq \pi^i (\eta\nonsq^{2})^{-1}}{1}$. 
Note that $gXg^{-1}=X$ for all $X\in \Z_{\ell-i}(\co_\ell)$. Therefore
$\B(\co_\ell) \cap \left(\Z_{\ell-i}(\co_\ell)\U(\co_\ell)\right)^g=\Z_{\ell-i}(\co_\ell) \left(\B(\co_\ell) \cap\U(\co_\ell)^g\right)$. Hence, if we show 
$ (\chi,\chi',\psi_{\ell-i})^g= (\omega_1,\omega_2)$ on $\B(\co_\ell) \cap  \U(\co_\ell)^g $, then (2) follows from $(\chi,\chi',\psi_{\ell-i})^g\mid_{\Z_{\ell-i}(\co_\ell)}=(\chi,\chi')\mid_{\Z_{\ell-i}(\co_\ell)}$ and   \autoref{lem: borel rep isomorphism}. For $b\in \cO_\ell$, we have  $g\smat{1}{\nonsq b}{0}{1}g^{-1}=\smat{1+\pi^i b \eta^{-1}}{\nonsq b}{-\pi^{2i}b\nonsq^{-1}\eta^{-2}}{1-\pi^i b \eta^{-1}}$. Therefore 
$$\B(\co_\ell) \cap  \U(\co_\ell)^g=\left\{ \mat{1+\pi^i b \eta^{-1}}{\nonsq b}{0}{1-\pi^i b \eta^{-1}} \mid b\in \cO_\ell \text{ with } \pi^{2i}b=0\right\}. $$
For $X_b:= \smat{1+\pi^i b \eta^{-1}}{\nonsq b}{0}{1-\pi^i b \eta^{-1}}\in \B(\co_\ell) \cap  \U(\co_\ell)^g,$ since $(1+\pi^i b \eta^{-1})^{-1}=1-\pi^i b \eta^{-1}$ and  $\psi(\nonsq x)=\psi( x)$ for all $x\in \cO_\ell$, we obtain that  
$$(\chi,\chi',\psi_{\ell-i})^g(X_b)=\psi(\pi^i \nonsq b)= \psi(\pi^i b)=\omega_1\omega_2^{-1}(1+\pi^i b \eta^{-1})=(\omega_1,\omega_2)(X_b).$$
Therefore $ (\chi,\chi',\psi_{\ell-i})^g= (\omega_1,\omega_2)$ on $\B(\co_\ell) \cap  \U(\co_\ell)^g $.

To prove (3), let  $j\in [1,\ell]$ and  $z\in \cO_\ell^\times$ be such that  $\pi^j z\neq -\pi^i (\eta\nonsq^{2})^{-1} \mod (\pi^{2i})$, and let  $g=\smat{1}{0}{\nonsq \pi^j z}{1}$. By the given conditions,  $\pi^j z+ \pi^i (\eta\nonsq^{2})^{-1} = \pi^k u $  for some $k \in [\mathrm{min}\{i,j\} , \mathrm{min}\{2i-1,\ell-1\}]$ and $u\in \cO_\ell^\times$. This gives 
$k = j$ for $j < i$ and $k \leq 2j-1$ for $j > i$. 
   Therefore  $\ell+2j-k-1\geq \ell$.

For $b\in \cO_\ell$, let 
   $$ Y_b:= g\mat{1}{\pi^{\ell-k-1}\nonsq b}{0}{1}g^{-1} = \mat{1-\pi^{\ell+j-k-1}\nonsq^2 bz}{\pi^{\ell-k-1}\nonsq b}{-\pi^{\ell+2j-k-1}b \nonsq^3 z^2}{1+\pi^{\ell+j-k-1}\nonsq^2 bz} \in \B(\co_\ell) \cap \left(\Z_{\ell-i}(\co_\ell)\U(\co_\ell)\right)^g. $$
 
  For a $\ss$-pair $(\omega_1,\omega_2)$, we show that   $ (\chi,\chi',\psi_{\ell-i})^g(Y_b)\neq  (\omega_1,\omega_2)(Y_b)$ for some $b\in \cO_\ell$. Assume on the contrary that $ (\chi,\chi',\psi_{\ell-i})^g(Y_b)= (\omega_1,\omega_2)(Y_b)$ for all $b\in \cO_\ell$. Then,  using the fact that $2(\ell+j-k-1)=\ell+(\ell-k-1)+(2j-k-1)\geq \ell$, we obtain 
  \begin{equation}\label{eqn:poi}
      \psi(\pi^{\ell+i-k-1}\nonsq b)=\omega_1\omega_2^{-1}(1-\pi^{\ell+j-k-1}\nonsq^2 bz)=\psi(-\eta\pi^{\ell+j-k-1}\nonsq^2 bz) \text{ for all } b\in \cO_\ell.
  \end{equation}
  Since $ \psi(\pi^{\ell+i-k-1}\nonsq b)= \psi(\pi^{\ell+i-k-1} b)$  for all $b\in \cO_\ell$,  \autoref{eqn:poi} gives 
 $\psi(\pi^{\ell-k-1} b (\pi^i +\pi^j \eta\nonsq^2 z))=1$  for all 
 $b\in \cO_\ell$.
 Since $\pi^j z+ \pi^i (\eta\nonsq^{2})^{-1} = \pi^k u $, we obtain that $\psi(\pi^{\ell-1} b \eta\nonsq^2)=1$  for all 
 $b\in \cO_\ell$, which contradicts the fact that $\pi^{\ell-1}\cO_\ell \nsubseteq \ker(\psi)$. Thus  there exists $b\in \cO_\ell$ such that  $ (\chi,\chi',\psi_{\ell-i})^g(Y_b)\neq  (\omega_1,\omega_2)(Y_b)$.  This proves (3).
\end{proof}
 
For $j\in\{3,4\}$ and   $(\omega_1,\omega_2)\in S_j^{\ell-1}$, define 
$$n_j(\omega_1,\omega_2):=
    \mathrm{min}\{ k \in [0,\ell-1]\mid (\omega_1,\omega_2)\in S_j^k\}.$$

\begin{proposition}
    \label{prop:Induction T to G SS multiplicities}
     For any  $\ss$-pair $(\omega_1,\omega_2)$, we have
    \[
    \langle \ind_{\T(\co_\ell)}^{\G(\co_\ell)} (\chi_1\chi_4,\chi_2\chi_3) , \ind_{\B(\co_\ell)}^{\G(\co_\ell)}(\omega_1,\omega_2)\rangle=\begin{cases}
         \ell-n_3(\omega_1,\omega_2)+1, & \text{if } (\omega_1,\omega_2)\in  S_3^{\ell-1} ;\\
         \ell-n_4(\omega_1,\omega_2)+1, & \text{if } (\omega_1,\omega_2)\in  S_4^{\ell-1} ;\\
         1, & \text{if } (\omega_1,\omega_2)\in S_0\setminus ( S_3^{\ell-1} \cup S_4^{\ell-1} );\\
        0, & \text{ otherwise. }
       
    \end{cases}
    \] 
\end{proposition}
\begin{proof}
We have
\begin{equation}\label{eqn:khtf}
    \langle \ind_{\T(\co_\ell)}^{\G(\co_\ell)} (\chi_1\chi_4,\chi_2\chi_3), \ind_{\B(\co_\ell)}^{\G(\co_\ell)}(\omega_1,\omega_2)\rangle=\underset{g\in {\T(\co_\ell)} \backslash {\G(\co_\ell)} /{\B(\co_\ell)}}{\sum}{\langle(\chi_1\chi_4,\chi_2\chi_3),{(\omega_1,\omega_2)}^g\rangle}_{{\T(\co_\ell)} \cap {\B(\co_\ell)}^g}.
\end{equation}
It is easy to verify that the set 
$\Omega:=\{ \left[\begin{smallmatrix}
                  \nonsq &1\\
                  1&0
    \end{smallmatrix}\right], \left[\begin{smallmatrix}
                 \nonsq \pi^i&1\\
                  1&0
  \end{smallmatrix}\right], \left[\begin{smallmatrix}
                 1&0\\
                \nonsq  \pi^i&1
  \end{smallmatrix}\right]; 1\leq i \leq \ell  \}$ 
 forms a complete set of double coset representatives for
$\T(\co_\ell)\backslash {\G(\co_\ell)} /\B(\co_\ell).$ By direct computations, we get  
$${\T(\co_\ell)} \cap {\B(\co_\ell)}^g=\begin{cases}
 \Z(\co_\ell), & \text{if } g=\left[\begin{smallmatrix}
                 \nonsq &1\\
                  1&0
  \end{smallmatrix}\right];\\
    \Z_{\ell-i}(\co_\ell), & \text{if } g\in \{\left[\begin{smallmatrix}
                 \nonsq \pi^i&1\\
                  1&0
  \end{smallmatrix}\right], \left[\begin{smallmatrix}
                 1&0\\
                \nonsq  \pi^i&1
  \end{smallmatrix}\right]\}\text{ with }i\in [1,\ell].
\end{cases}$$
Now we obtain the following necessary and sufficient conditions for  $\ss$-pair $(\omega_1,\omega_2)$ such that $(\chi_1\chi_4,\chi_2\chi_3)={(\omega_1,\omega_2)}^g$ on ${\T(\co_\ell)} \cap {\B(\co_\ell)}^g$ for different choices of $g\in \Omega$.
\begin{enumerate}
    \item For $g=\left[\begin{smallmatrix}
                  \nonsq &1\\
                  1&0
    \end{smallmatrix}\right], (\omega_1,\omega_2)\in S_0$.
    \item For $g=\left[\begin{smallmatrix}
                 \nonsq \pi^i&1\\
                  1&0
  \end{smallmatrix}\right]\text{ with }i\in [1,\ell], \,(\omega_1,\omega_2)\in S_4^{\ell-i}$.
  \item For $g=\left[\begin{smallmatrix}
                 1&0\\
                \nonsq  \pi^i&1
  \end{smallmatrix}\right]\text{ with }i\in [1,\ell], \, (\omega_1,\omega_2)\in S_3^{\ell-i}$.
\end{enumerate}
Therefore the result follows from \autoref{eqn:khtf} and the facts that  $S_j^0 \subseteq S_j^1  \subseteq \cdots  \subseteq S_j^{\ell-1} \subseteq S_0$ for $j\in\{3,4\}$, and $S_3^k \cap  S_{4}^{k'}=\emptyset $ for  $k,k' \in [0, \ell-1] $.
\end{proof}

\begin{lemma}\label{lem:b to G ind}
    \begin{enumerate}
        \item If $(\chi_1\chi_3,\chi_2\chi_4)$ is not a $\ss$-pair, then $\langle \ind_{\B(\co_\ell)}^{\G(\co_\ell)}(\chi_1\chi_3,\chi_2\chi_4), \ind_{\B(\co_\ell)}^{\G(\co_\ell)}(\omega_1,\omega_2)\rangle=0$ for every  $\ss$-pair  $(\omega_1,\omega_2)$.
        \item If $(\chi_1\chi_3,\chi_2\chi_4)$ is  a $\ss$-pair, then $(\chi_1\chi_3,\chi_2\chi_4) \in S_1^k$ for all $k\in [0,\ell-1]$.
    \end{enumerate}
\end{lemma}
\begin{proof}
This follows from the characterisation of $\ss$-pairs.
\end{proof}

The proof of \autoref{thm:main-theorem-2}(4) follows from \autoref{eqn: ind B tensor ind B}, \autoref{prop:condition for ss in phi i}, \autoref{prop:Induction T to G SS multiplicities},  \autoref{lem:b to G ind} and the fact that  $S_j^k \cap  S_{j'}^{k'}=\emptyset $ for all $k,k' \in [0, \ell-1] $ and  $j,j'\in [1,4]$ such that  $j\neq j'$.

\begin{remark}
  The multiplicity $\ell+1$ is always achieved by a split semisimple representation in $\ind_{\B(\co_\ell)}^{\G(\co_\ell)} (\chi_1,\chi_2) \otimes \ind_{\B(\co_\ell)}^{\G(\co_\ell)} (\chi_3,\chi_4)$. For proving this we note that for odd $p$, either $(\chi_1\chi_3,\chi_2\chi_4)$ or $(\chi_1\chi_4,\chi_2\chi_3)$ is $\ss$-pair. Hence, using \autoref{prop:condition for ss in phi i},  either
  \[\langle \ind_{\B(\co_\ell)}^{\G(\co_\ell)} (\chi_1,\chi_2) \otimes \ind_{\B(\co_\ell)}^{\G(\co_\ell)} (\chi_3,\chi_4), \ind_{\B(\co_\ell)}^{\G(\co_\ell)}(\chi_1\chi_3,\chi_2\chi_4)\rangle = \ell+1,\] or
  \[\langle \ind_{\B(\co_\ell)}^{\G(\co_\ell)} (\chi_1,\chi_2) \otimes \ind_{\B(\co_\ell)}^{\G(\co_\ell)} (\chi_3,\chi_4), \ind_{\B(\co_\ell)}^{{\G(\co_\ell)}(\co_\ell)}(\chi_1\chi_4,\chi_2\chi_3)\rangle = \ell+1.\]
\end{remark}

\section{Proof of \autoref{thm:main-theorem-2}(5)}
\label{sec:results-Sigma-5}

In this section, we  will prove \autoref{thm:main-theorem-2}(5) by giving an example of split non-semisimple irreducible representation $\rho$ of $\G(\cO_\ell)$  such that $\langle \rho \otimes \rho, \rho \rangle \geq (q-2)q^{\lfloor \frac{\ell}{2} \rfloor-1}.$ We will also give slightly more general results for the case $\lfloor \frac{\ldown}{2} \rfloor\geq 2.$
 
 Let $A = \smat00\nonsq0 \in \g(\cO_\ldown)$.
  For $ i \in [\lceil \ell_1/2 \rceil, \ell_1]$, 
 let 
 $$\mathcal{X}_i =
 \left\{\mat{a}{\pi^i b}{0}{c} \in \G(\cO_{\ell}) \mid a,b,c \in R_\ell^\times , a+c\in R_\ell^\times \right\} .$$
 
\begin{proposition}\label{prop:double coset rep SNS with SNS}
Let $ i,j \in [\lceil \ell_1/2 \rceil , \ell_1].$
    \begin{enumerate}
        \item If $i\neq j,$ then $\{S_{A}gS_{A}\mid g\in \mathcal{X}_{i} \}\cap \{S_{A}hS_{A}\mid h\in \mathcal{X}_{j} \}=\emptyset.$
        \item For $k \in \{1,2\}$, let $g_k=\smat{a_k}{\pi^ib_k}{0}{c_k} \in \mathcal{X}_i.$ Then $S_{A}g_1S_{A}=S_{A}g_2S_{A}$ if and only if $\pi^ia_1^{-1}b_1=\pi^ia_2^{-1}b_2   \mod (\pi^{\ell_1})$ and $a_1^{-1}c_1=a_2^{-1}c_2  \mod (\pi^i).$
  \item  $|\{S_{A}gS_{A}\mid g\in \mathcal{X}_{i} \}|=\begin{cases}
            (q-1) (q-2)q^{\ell_1-2}_{\,,}& \mathrm{if}\, i<\ell_1;\\
            (q-2)q^{\ell_1-1}_{\,,}& \mathrm{if}\, i=\ell_1.
        \end{cases}$
           \end{enumerate}
\end{proposition}
\begin{proof}
Note that $S_A=\left\{ \smat{x}{\pi^{\ldown} y}{z}{x+\pi^{\ldown} w} \mid x,y,z,w \in \rl\right\}\cap \G(\cO_\ell). $
Therefore, for $g\in \mathcal{X}_{i},$ it is easy to see that $(1,2)^\mathrm{th}$ entry of  any $X\in  S_A gS_A$ is in $\pi^i \rl^\times + \pi^{\ldown}\rl.$ This implies (1).

    To show (2), let $S_{A}g_1S_{A}=S_{A}g_2S_{A}$. Then there exist $x_1,x_2\in \rl^\times$ and $z_1,z_2\in \rl$ such that $\smat{x_1}{0}{z_1}{x_1}g_1=g_2  \smat{x_2}{0}{z_2}{x_2} \mod (\pi^\ldown)$.
    This gives
    \begin{equation}\label{eqn:adww1}
        \mat{a_1 x_1-a_2 x_2- \pi ^ib_2  z_2}{\pi ^i \left(b_1 x_1-b_2
   x_2\right) }{
 a_1 z_1-c_2 z_2 }{c_1 x_1-c_2 x_2 + \pi ^i b_1  z_1 }=0  \mod (\pi^\ldown)
    \end{equation}
    Equating the $(1,1)^\mathrm{th}$ entry on both sides of \autoref{eqn:adww1}, 
we obtain that  $x_1=a_1^{-1}a_2 x_2  \mod (\pi^i ).$  
Since $2i\geq \ldown$,
by substituting this value of $x_1$ into the second column on the left-hand side of \autoref{eqn:adww1} and simplifying, we obtain
 $\pi^ia_1^{-1}b_1=\pi^ia_2^{-1}b_2   \mod (\pi^{\ell_1})$ and $a_1^{-1}c_1=a_2^{-1}c_2  \mod (\pi^i).$ To prove the converse, let 
$\pi^ia_1^{-1}b_1=\pi^ia_2^{-1}b_2   \mod (\pi^{\ell_1})$ and $a_1^{-1}c_1=a_2^{-1}c_2  \mod (\pi^i).$
If $i=\ldown$, then $g_1\smat{a_1^{-1}a_2}{0}{0}{c_1^{-1}c_2}=g_2 \mod (\pi^\ldown)$. It is straightforward to see that $\smat{a_1^{-1}a_2}{0}{0}{c_1^{-1}c_2}\in S_A$. Therefore $S_{A}g_1S_{A}=S_{A}g_2S_{A}$ for $i=\ldown$. Let $i<\ldown$. Then we have $a_1^{-1}b_1=a_2^{-1}b_2   \mod (\pi^{\ell_1-i})$, and hence $a_1^{-1}b_1+a_2^{-1}b_2\in \rl^\times.$ For $i\in \{1,2\}$, since $g_i\in \G(\cO_\ell)$, we have $a_i^{-1}c_i\in \cO_\ell$ and $ a_i^{-1}b_i \in \nonsq\cO_\ell$. Therefore 
$\frac{a_2^{-1}c_2-a_1^{-1}c_1}{a_1^{-1}b_1+a_2^{-1}b_2}=\pi^i \nonsq d$ for some $d\in \cO_\ell$.
Let $X=\smat{1}{0}{\nonsq d}{1}$ and $Y=\smat{\frac{a_1}{a_2}-\frac{\pi ^i \nonsq d a_1 b_2   }{a_2 c_2}}{0}{\frac{a_1 d \nonsq }{c_2}}{\frac{a_1}{a_2}-\frac{\pi ^i \nonsq d a_1 b_2   }{a_2 c_2}}  .$ 
  By direct calculation, we have 
  \begin{equation*}
Xg_1-g_2Y=\mat{0}{\pi ^i \left(b_1-\frac{a_1 b_2}{a_2}\right)}{0}{
 c_1 -\frac{a_1
   c_2}{a_2}+ \pi ^i \nonsq  d  \left(\frac{a_1 b_2}{a_2}+b_1\right)}=0 \mod (\pi^{\ell_1} ).
  \end{equation*}
  Note that  $X \in \G(\co_\ell)$, and hence   $g_2^{-1}Xg_1\in \G(\cO_\ell).$  
  Since $Y=g_2^{-1}Xg_1\mod (\pi^{\ell_1} )$ and the map $\rho_{\ell, \ldown}:\G(\cO_\ell)\rightarrow \G(\cO_\ldown)$ is a projection,  there exists $Z\in M_2(\rl)$ such that $Y+\pi^{\ell_1}Z\in \G(\cO_\ell).$
Note that $Y+\pi^{\ell_1}Z\in S_{A}$ and $Xg_1=g_2(Y+\pi^{\ell_1}Z)\mod (\pi^{\ell_1} ).$ Therefore $S_{A} g_1 S_{A}=S_{A} g_2 S_{A}.$ 
  
To show (3),
    let $$\mathcal{D}_i=\{(\mathrm{Proj}_{\ell_1}(\pi^ia^{-1}b),\mathrm{Proj}_{i}(a^{-1}c))\in R_{\ell_1}\times R_i \mid \smat{a}{\pi^i b}{0}{c} \in \mathcal{X}_{i}\},$$ where $\mathrm{Proj}_{\ell_1}:R_\ell \rightarrow R_{\ell_1}$ and $\mathrm{Proj}_{i}:R_\ell \rightarrow R_{i}$ are canonical projections.  From (2), we obtain 
    $|\{S_{A}gS_{A}\mid g\in \mathcal{X}_{i} \}|=|\mathcal{D}_i|.$ For $\G=\GL_2,$ we have 
    $\mathcal{D}_i=\{(\mathrm{Proj}_{\ell_1}(\pi^i d),\mathrm{Proj}_{i}(e))
    \mid d\in \cO_\ell^\times, e \in \cO_\ell^\times \setminus (-1+\pi\cO_\ell)\}.$
    Therefore, for $\G = \GL_2$, we obtain  that
    $$|\{S_{A}gS_{A}\mid g\in \mathcal{X}_{i} \}|=\begin{cases}
            (q-1) (q-2)q^{\ell_1-2}_{\,,}& \mathrm{if}\, i<\ell_1;\\
            (q-2)q^{\ell_1-1}_{\,,}& \mathrm{if}\, i=\ell_1.
        \end{cases}$$
        We next consider  $\G=\GU_2.$ For this case, 
      $\smat a {\pi^ib}0c \in \mathcal{X}_i$ if and only if $a,b,c \in \Lri_\ell^\times$ with $a^{-1}=c^\circ$, $a^{-1}b\in\nonsq\cO_\ell^\times$ and $a+c\in \Lri_\ell^\times$. We also have $\{c^\circ c\mid c \in \Lri_\ell^\times\, \mathrm{with} \, c^\circ c+1\in \Lri_\ell^\times\}=\cO^\times_\ell\setminus (-1+\pi\cO_\ell).$
      Therefore $$\mathcal{D}_i=\{(\mathrm{Proj}_{\ell_1}(\pi^i d),\mathrm{Proj}_{i}(e))
    \mid d\in \nonsq\cO_\ell^\times, e \in \cO_\ell^\times \setminus (-1+\pi\cO_\ell)\}.$$
   Using  $|\{S_{A}gS_{A}\mid g\in \mathcal{X}_{i} \}|=|\mathcal{D}_i|,$ the result follows for $\G = \GU_2$ also. 
\end{proof} 
\begin{proof}[Proof of \autoref{thm:main-theorem-2}(5)]
Recall the construction of split non-semisimple regular representations from \autoref{E.construction} for even $\ell$ and from   \autoref{subsec:alternate const for sns} for odd $\ell$. Fix a  Serre lift $\ti{A} = \smat00\nonsq0 \in \g(\cO_\ell)$ of $A$. 
Recall that $\mathrm{N}=\left \{\left[\begin{smallmatrix}
      1+\pi^{\ldown}x & \pi^{\lup}z\\
       \pi^{\ldown}y    & 1+\pi^{\ldown}w
\end{smallmatrix}\right] \mid x,y,z,w \in \rl \right \} \cap \G(\co_\ell),$ and let  $H:=\mathrm{N} \mathrm{C}_{\G(\co_\ell)}(\ti{A})$. Note that for even $\ell$, we have $H=S_A$.
Consider the extension $\psi_{\tilde{A}}$ of $\psi_A$ to $\mathrm{N}$ defined by $\psi_{\tilde{A}}(\I + \pi^\ldown B)=\psi(\pi^{\ldown}\bm{tr}(\tilde{A}B))$ for $\I + \pi^\ldown B \in \mathrm{N}$.
Let $\phi$ be the  character  of $H$ such that $\phi|_{N} = \psi_{\tilde{A}}$ and $\phi|_{\mathrm{C}_{\G(\co_\ell)}(\ti{A})}=1$.
Define $\rho = \ind_{H}^{\G(\co_\ell)}{\phi}$. Then $\rho$ is a split non-semisimple irreducible representation of $\G(\cO_\ell)$. We will prove that  $\langle \rho \otimes \rho, \rho \rangle \geq (q-2)q^{\ldown-1}.$   Note that
    \begin{equation}\label{eqn:tre1}
       \rho \otimes \rho \cong \ind_{H}^{\G(\co_\ell)}{\phi} \otimes \ind_{H}^{\G(\co_\ell)}{\phi} \cong \underset{g\in H \backslash 
        \G(\co_\ell) / H }{\oplus} \ind_{H\cap H^g}^{\G(\co_\ell)}(\phi\otimes \phi^g).
    \end{equation}
We claim that for  $g\in  \mathcal{T}:=\left\{\smat{a}{0}{0}{c}\in \G(\cO_\ell)| \, a+c\in \rl^\times\right\}$, $H\cap H^g=H$ and $\ind_{H\cap H^g}^{\G(\co_\ell)}(\phi\otimes \phi^g)\cong \rho$. By assuming the claim, from \autoref{eqn:tre1} we obtain 
\begin{equation}\label{eqn:tre2}
    \langle \rho \otimes \rho, \rho \rangle \geq |\{HgH\mid g\in \mathcal{T}\}| \geq |\{S_AgS_A\mid g\in \mathcal{T} \}|.
\end{equation}
    Note that for $\smat{a}{\pi^\ldown b}{0}{c} \in \mathcal{X}_\ldown$, we have $\smat{a}{\pi^\ldown b}{0}{c} =\smat{a}{0}{0}{c}\smat{1}{\pi^\ldown a^{-1}b}{0}{1}\in \smat{a}{0}{0}{c} S_A.$ Therefore 
$|\{S_AgS_A\mid g\in \mathcal{T}\}|= |\{S_AgS_A\mid g\in \mathcal{X}_\ldown\}|$.  Now the result directly follows from \autoref{eqn:tre2} and  \autoref{prop:double coset rep SNS with SNS}(3).

    To show the claim, let $g=\smat{a}{0}{0}{c}\in \mathcal{T}$. By direct computations, it is straightforward  that $H\cap H^g=H$.
   To show $\ind_{H\cap H^g}^{\G(\co_\ell)}(\phi\otimes \phi^g)\cong \rho$, it is enough to show that $\phi\otimes \phi^g=\phi^h$ for some $h\in \G(\cO_\ell)$.
Let $ h= \smat{d}{0}{0}{d(1+a^{-1}c)}$, where $d=1$ for $\G=\GL_2$ and $d\in \Lri_\ell$ such that $d^\circ d=(1+a^{-1}c)^{-1}$ for $\G=\GU_2$.
Then $h\in \G(\cO_\ell)$. Note that $\mathrm{C}_{\G(\co_\ell)}(\ti{A})^g=\mathrm{C}_{\G(\co_\ell)}(\ti{A})$ and $\mathrm{C}_{\G(\co_\ell)}(\ti{A})^h=\mathrm{C}_{\G(\co_\ell)}(\ti{A})$. Therefore we have
\begin{equation}\label{eqn:uyt}
    \phi(X)\phi(g^{-1}Xg)=1=\phi(h^{-1}Xh) \, \text{ for all }\, X\in \mathrm{C}_{\G(\co_\ell)}(\ti{A}).
\end{equation}
For $ Y=\I+\pi^\ldown\smat{ x}{\pi^{\lup-\ldown} y}{ z}{ w}\in \mathrm{N}$, we have $g^{-1}Yg=\I+\pi^\ldown\smat{ x}{\pi^{\lup-\ldown} a^{-1}c y}{ c^{-1}a z}{ w}$ and $h^{-1}Yh= \I+\pi^\ldown\smat{ x}{\pi^{\lup-\ldown} (1+ a^{-1}c ) y}{ (1+ a^{-1}c )^{-1}z}{ w}$, and hence we obtain that 
$$\phi(Y)\phi(g^{-1}Yg)= \psi(\pi^\lup \nonsq y)\psi(\pi^\lup \nonsq \, a^{-1}c y)= \psi(\pi^\lup \nonsq  (1+ a^{-1}c )y)=\phi(h^{-1}Yh). $$
This, together with \autoref{eqn:uyt}, implies that $\phi\otimes \phi^g=\phi^h$. Hence, the claim holds.
\end{proof}

 We are also able to prove the following stronger result for $\ell \geq 2$. For $\ell$ such that $\lfloor \frac{\ldown}{2} \rfloor \geq 2,$ this result also proves \autoref{cor:residue-dependence}. 

\begin{thm}\label{thm:SNS with SNS} 
 	Let $A=\smat{0}{0}{\nonsq}{0}\in \g(\cO_{\ell_1}).$  For any $\rho_1,\rho_2\in \mathrm{Irr}(\G(\cO_\ell)\mid \psi_{A}),$  there exists  $\rho\in \mathrm{Irr}(\G(\cO_\ell)\mid \psi_{A}) $ such that  
 		$$	    \langle \rho_1 \otimes \rho_2 , \rho \rangle \geq
         (q-2)q^{\lfloor \ell_1/2 \rfloor+\ell_1-\ell_2-1}. $$
 \end{thm}
For its proof, we require the following general result. 

\begin{lemma}\label{lem:multiplicity geq}
    Let $H$ be a subgroup of a finite group $G.$ Suppose $\theta$ and $\chi$ are representations of $G$ and $H$ respectively such that  
    $\{\rho\in \mathrm{Irr}(G)\mid \langle 
     \rho,\theta\rangle\neq 0 \}\subseteq\mathrm{Irr}(G\mid \chi).$ Then there exists a representation $\rho\in \mathrm{Irr}(G\mid \chi)$ such that $\langle \rho,\theta\rangle \geq \frac{\dim(\theta)}{\dim(\mathrm{Ind}_{H}^{G}(\chi))}.$
\end{lemma}
\begin{proof}
    Let $\mathrm{Irr}(G\mid \chi)=\{\rho_1,\rho_2,...,\rho_t\}$ and $m_k=\langle \theta,\rho_k\rangle$ for $k \in [1, t]$. Note that $\sum_{1\leq k\leq t}\dim(\rho_k)\leq \dim(\mathrm{Ind}_{H}^{G}(\chi)).$ Since  $\{\rho\in \mathrm{Irr}(G)\mid \langle 
     \rho,\theta\rangle\neq 0 \}\subseteq\mathrm{Irr}(G\mid \chi),$ we also have  $\dim(\theta)=\sum_{1\leq k\leq t}m_k\dim(\rho_k).$
     To show the result, it is enough to prove that $m^*:=\mathrm{max}\{m_k \mid  k \in  [1, t]\}$ satisfies $m^* \geq \frac{\dim(\theta)}{\dim(\mathrm{Ind}_{H}^{G}(\chi))}. $ This directly follows from the following:
     $$m^* \dim(\mathrm{Ind}_{H}^{G}(\chi))\geq m^*  \sum_{1\leq k\leq t}\dim(\rho_k)\geq  \sum_{1\leq k\leq t}m_k\dim(\rho_k)=\dim(\theta). $$
This completes the proof. 
\end{proof}

\begin{proof}[Proof of \autoref{thm:SNS with SNS}]
For $k\in \{1,2\}$,
let $\phi_k \in \mathrm{Irr}(S_A \mid \psi_A)$ such that $\rho_k \cong \mathrm{ind}_{S_A}^{\G(\cO_\ell)}(\phi_k).$ For $i\in [\lceil \ell_1/2 \rceil , \ell_1]$, denote $|\{S_{A}gS_{A}\mid g\in \mathcal{X}_{i} \}|$ by $n_i$, and 
        let $\{g_{i,j}\mid  \, 1\leq j \leq n_i\}\subseteq \mathcal{X}_{i}$ be a set of distinct double coset representatives of $S_A\backslash \G(\cO_\ell)/S_A$ in  $\mathcal{X}_{i}$.
    Consider the sub-representation 
$$ \Theta:=\oplus_{\lceil \ell_1/2 \rceil \leq i\leq \ell_1} \left( \oplus_{1 \leq j\leq n_i}\mathrm{Ind}_{S_{A}\cap S_{A}^{g_{i,j}}}^{\G(\cO_{\ell})}(\phi_1 \otimes \phi_2^{g_{i,j}})\right)$$
of $\mathrm{Ind}_{S_{A}}^{\G(\cO_{\ell})}(\phi_1)\otimes\mathrm{Ind}_{S_{A}}^{\G(\cO_{\ell})}(\phi_2).$
For $k\in \{1,2\}$, let  $\chi_k\in \mathrm{Irr}(\Z \K^{\ell_2} )$  be such that $\langle \phi_k ,\chi_k \rangle_{\Z\K^{\lup}} \neq 0.$ 
Note that $\chi_1|_{\K^{\ell_2}}=\chi_2|_{\K^{\ell_2}}=\psi_A.$
For any $g\in \bigcup_{\lceil \ell_1/2 \rceil \leq i \leq \ell_1}\mathcal{X}_i ,$ we have $A + g Ag^{-1}$ is conjugate to $2A$. This gives 
     $$\{\rho\in \mathrm{Irr}(\G(\cO_{\ell}))\mid \langle 
     \rho,\mathrm{Ind}_{S_{A}\cap S_{A}^g}^{\G(\cO_{\ell})}(\phi_1 \otimes \phi_2^g)
     \rangle\neq 0 \}\subseteq\mathrm{Irr}(\G(\cO_{\ell})\mid \chi_1\otimes\chi_2).$$
Therefore $\{\rho\in \mathrm{Irr}(\G(\cO_{\ell}))\mid \langle \rho,\Theta\rangle\neq 0 \}\subseteq\mathrm{Irr}(\G(\cO_{\ell})\mid \chi_1\otimes\chi_2).$
By \autoref{lem:multiplicity geq}, there exists a representation $\rho\in\mathrm{Irr}(\G(\cO_{\ell})\mid \chi_1\otimes\chi_2)$ such that $$\langle \Theta,\rho\rangle \geq \frac{\dim(\Theta)}{\dim(\mathrm{Ind}_{\Z \K^{\ell_2}}^{\G(\cO_{\ell})}(\chi_1\otimes\chi_2))}=\frac{\dim(\Theta) |\Z \K^{\ell_2}|}{|\G(\cO_{\ell})|}.$$ Since $\Theta$ is a sub-representation of $\mathrm{Ind}_{S_{A}}^{\G(\cO_{\ell})}(\phi_1)\otimes\mathrm{Ind}_{S_{A}}^{\G(\cO_{\ell})}(\phi_2)$ and $\mathrm{Irr}(\G(\cO_{\ell})\mid \chi_1\otimes\chi_2)\subseteq \mathrm{Irr}(\G(\cO_{\ell})\mid \psi_{2A})=\mathrm{Irr}(\G(\cO_{\ell})\mid \psi_A),$ to prove \autoref{thm:SNS with SNS}, it is enough to show that $\frac{\dim(\Theta)|\Z \K^{\ell_2}|}{|\G(\cO_{\ell})|}\geq \frac{q-2}{q^2}q^{\lfloor \ell_1/2 \rfloor}$.

To calculate $\dim(\Theta),$  note that for $g_{i,j}=\smat{a}{\pi^i b}{0}{c}\in \mathcal{X}_{i},$ we have 
$$g_{i,j}Ag_{i,j}^{-1}=\nonsq \mat{\pi^i a^{-1}b}{-\pi^{2i} a^{-1}c^{-1}b^2}{a^{-1}c}{-\pi^i a^{-1}b}.$$ 
By the definition of $S_A$, we obtain that $S_{A} \cap S_{A}^{g_{i,j}}=(\{x\mathrm{I}+y \tilde{A}\mid x\in \cO_{\ell}^\times, y\in \pi^{\ell_1-i}\cO_{\ell} \}\K^{\ell_1}) \cap \G(\cO_\ell).$ By direct computations, $|S_{A} \cap S_{A}^{g_{i,j}}|=(q+\Delta)q^{4\ell_2+\ell_1+i-1}.$ We also have $\dim(\phi_1 \otimes \phi_2^{g_{i,j}})=q^{2(\ell_2-\ell_1)}.$ 
By using \autoref{prop:double coset rep SNS with SNS}(3), we have   
\begin{eqnarray*}
    \dim(\Theta)
    &=&\sum_{\lceil \ell_1/2 \rceil \leq i\leq \ell_1} 
     \frac{n_i|\G(\cO_{\ell})|q^{2(\ell_2-\ell_1)}}{(q+\Delta)q^{4\ell_2+\ell_1+i-1}} \\
   &=&\frac{(q-2)|\G(\cO_{\ell})|}{(q+\Delta)q^{2\ell_2+2\ell_1+1}} \left[ \left( \sum_{\lceil \ell_1/2 \rceil \leq i\leq \ell_1-1}  \frac{(q-1)}{q^i} \right) +   \frac{q}{q^{\ell_1}} \right]\\
   &=&\frac{(q-2)|\G(\cO_{\ell})|}{(q+\Delta)q^{2\ell_2+2\ell_1+1}} \left[  \frac{1}{q^{\lceil \ell_1/2 \rceil -1}}  \right].
\end{eqnarray*}
Since $|\Z \K^{\ell_2}|=(q+\Delta)q^{4\ell_1 + \ell_2-1}$ and $\ell_1=\lceil \ell_1/2 \rceil+\lfloor \ell_1/2 \rfloor,$ we obtain 
$$\frac{\dim(\Theta) |\Z \K^{\ell_2}|}{|\G(\cO_{\ell})|}=
    (q-2)q^{\lfloor \ell_1/2 \rfloor+\ell_1-\ell_2-1}. 
$$
Hence the result follows. \end{proof}
\section{Further discussion and questions}
\label{sec:further-discussion-questions}
On the basis of computations in GAP, we conjecture the following number of regular constituents in the tensor products of regular representations of different types. 
\begin{table}[h]
\centering
\begin{tabular}{|c|c|c|c|c|}
\hline
 & $\#\cus$ & $\#\sns$ & \multicolumn{2}{|c|}{$\#\ss$} \\
\hline
 \textbf{multiplicity $\rightarrow$}& \textbf{1}  & \textbf{1}  & \textbf{1} & \textbf{2} \\
\hline
$\cus \otimes \ss$ & $\frac{(q^2-1)}{2}q^{\ell-2}$ & $q^{\ell-1}$ & $\frac{(q-1)^2}{2}q^{\ell-2}$ & - \\
\hline
$\cus \otimes \sns$ & $\frac{(q+1)(q-3)}{2}q^{\ell-2}$ & $q^{\ell-1}$ & $\frac{(q-1)^2}{2}q^{\ell-2}$ & - \\
\hline
$\ss \otimes \sns$ & $\frac{(q^2-1)}{2}q^{\ell-2}$ & $q^{\ell-1}$ & $\frac{(q-1)(q-3)}{2}q^{\ell-2}$ & $(q-1)q^{\ell-2}$ \\
\hline
\end{tabular}
\caption{Conjectured number of constituents in tensor products of regular representations with different types}
\label{tab:groupedD}
\end{table}
To determine the multiplicities of the non-regular constituents in tensor products of $\G(\cO_\ell)$ representations is a question we have not addressed in this work. Another natural direction is to study the tensor product problem for automorphism groups of rank two $\cO$-modules.

\vspace{.2cm} 
\noindent {\bf Acknowledgements:} The authors are grateful to Uri Onn and Santosh Nadimpalli for helpful discussions regarding this work. The second named author gratefully acknowledges the Post-Doctoral Fellowship provided by the National Board for Higher Mathematics (NBHM), India. The third named author acknowledges the financial support provided by SERB, India, through grant SPG/2022/001099.

\bibliography{refs}

\end{document}